\documentclass[reqno]{amsart}
\usepackage{amsmath, amsthm, amssymb, amstext}
\usepackage[left=3.5cm,right=3.5cm,top=3.5cm,bottom=3.5cm]{geometry}
\usepackage{hyperref}
\usepackage[dvipsnames]{xcolor}
\usepackage{orcidlink}
\hypersetup{pdfborder={0 0 0},colorlinks}
\usepackage{enumitem}
\allowdisplaybreaks
\newtheorem{theorem}{Theorem}
\newtheorem{remark}[theorem]{Remark}
\newtheorem{lemma}[theorem]{Lemma}
\newtheorem{proposition}[theorem]{Proposition}
\newtheorem{corollary}[theorem]{Corollary}
\newtheorem{definition}[theorem]{Definition}

\newcommand{\Wpzero}[1]{W^{1,#1}_0(\Omega)}
\newcommand{\N}{\mathbb{N}}
\newcommand{\R}{\mathbb{R}}
\renewcommand{\l}{\left}
\renewcommand{\r}{\right}
\def\abs#1{\left|{#1}\right|}
\DeclareMathOperator*{\esssup}{ess\,sup}

\newcommand{\eps}{\varepsilon}
\newcommand{\round}[1]{\left(#1\right)}
\newcommand{\ph}{\varphi}
\newcommand{\close}{\overline{\Omega}}
\newcommand{\cprime}{$'$}

\numberwithin{theorem}{section}
\numberwithin{equation}{section}

\title[H\"{o}lder regularity for logarithmic double phase problems]
{H\"{o}lder regularity for logarithmic double phase problems}

\author[K. Ho]{Ky Ho\orcidlink{0000-0002-1362-9321}}
\address[K. Ho]{Department of Mathematics and Statistics, University of Economics Ho Chi Minh City, 59C, Nguyen Dinh Chieu Street, Ho Chi Minh City, Vietnam}
\email{kyhn@ueh.edu.vn}

\author[Y.-H. Kim]{Yun-Ho Kim\orcidlink{0000-0002-3558-4021}}
\address[Y.-H. Kim]{Department of Mathematics Education, Sangmyung University, Seoul, 03016, Korea}
\email{kyh1213@smu.ac.kr}

\author[P. Winkert]{Patrick Winkert\orcidlink{0000-0003-0320-7026}}
\address[P. Winkert]{Technische Universit\"{a}t Berlin, Institut f\"{u}r Mathematik, Stra\ss{}e des 17. Juni 136, 10623 Berlin, Germany}
\email{winkert@math.tu-berlin.de}


\begin{document}

\begin{abstract}
	We investigate boundedness and regularity properties of weak solutions to a class of generalized logarithmic double phase equations with variable exponents. The considered operators arise from Musielak-Orlicz type energies of logarithmic double phase type and exhibit nonstandard growth features. Under general structural assumptions, we derive a priori boundedness estimates in the subcritical setting and establish boundedness of weak solutions also in the presence of critical growth terms. In addition, we prove global H\"{o}lder continuity up to the boundary by means of the De Giorgi iteration scheme, localization arguments, and the frozen functional technique. The obtained results extend several existing regularity results for double phase and related nonstandard growth problems.
\end{abstract}

\maketitle

%********************************************************************
\section{Introduction and main results}%\label{Introduction}
%********************************************************************

In a recent contribution, Arora--Crespo-Blanco--Winkert \cite{Arora-Crespo-Blanco-Winkert-2025} introduced a new class of double phase operators with logarithmic perturbation, namely
\begin{equation}\label{log-operator}
	\begin{aligned}
		\operatorname{div}\left( |\nabla u|^{p(x)-2} \nabla u  + \mu(x) \left(\log(e+|\nabla u|) + \frac{|\nabla u|}{q(x)(e+|\nabla u|)}\right)|\nabla u|^{q(x)-2} \nabla u \right),
	\end{aligned}
\end{equation}
whose associated energy functional is given by
\begin{align}\label{log-functional}
	u \mapsto \int_\Omega \left(  \frac{ \abs{\nabla u}^{p(x)} }{p(x)} + \mu(x) \frac{ \abs{\nabla u}^{q(x)} }{q(x)} \log (e + \abs{\nabla u}) \right)\,  \mathrm{d}x.
\end{align}
Here, $\Omega \subseteq \R^N$, $N \geq 2$, is a bounded domain with Lipschitz boundary $\partial \Omega$, $e$ denotes Euler's number, $p,q \in C(\overline{\Omega})$ with $1<p(x) \leq q(x)$ for all $x \in \overline{\Omega}$, and $\mu \in L^1(\Omega)$ is nonnegative. The natural variational setting for this functional is the Musielak-Orlicz Sobolev space $W^{1, \mathcal{H}_{\log}}(\Omega)$, generated by the generalized $N$-function
\begin{align*}%\label{N-function}
	\mathcal{H}_{\log}(x,t)=t^{p(x)} + \mu(x) t^{q(x)} \log (e + t)\quad\text{for all }(x,t)\in \overline{\Omega} \times [0,\infty).
\end{align*}
Logarithmic perturbations in variational integrals have been studied extensively over the past decades. In the constant exponent setting, Baroni--Colombo--Mingione \cite{Baroni-Colombo-Mingione-2016} established the local H\"{o}lder continuity of gradients of local minimizers of
\begin{align*}%\label{log-functional-Mingione1}
	w \mapsto \int_\Omega \left[  \abs{D w}^p  + a(x) \abs{D w}^p \log (e + \abs{D w}) \right]  \,\mathrm{d}x,
\end{align*}
under the assumptions $1 < p < \infty$ and $0 \leq a(\cdot) \in C^{0,\alpha} (\close)$. It is worth observing that, when $p=q$ is constant, the functional above coincides with \eqref{log-functional} up to a multiplicative constant. More recently, De Filippis--Mingione \cite{DeFilippis-Mingione-2023} proved the local H\"{o}lder continuity of gradients of local minimizers associated with
\begin{align*}%\label{log-functional-Mingione2}
	w \mapsto \int_\Omega \big(|D w|\log(1+|D w|)+a(x)|D w|^q\big)\,\mathrm{d}x
\end{align*}
assuming that $0 \leq a(\cdot)\in C^{0,\alpha}(\overline{\Omega})$ and $1<q<1+\frac{\alpha}{N}$. The latter functional originates from the model integral
\begin{align*}%\label{log-functional-Mingione3}
	w \mapsto \int_\Omega |D w|\log(1+|D w|)\,\mathrm{d}x,
\end{align*}
whose regularity properties were investigated by Fuchs--Mingione \cite{Fuchs-Mingione-2000}, as well as Marcellini--Papi \cite{Marcellini-Papi-2006}. Functionals of this type naturally arise in plasticity theory with logarithmic hardening, see, for example, Seregin--Frehse \cite{Seregin-Frehse-1999}, as well as the monograph of Fuchs--Seregin \cite{Fuchs-Seregin-2000}, where variational methods for models arising in plasticity and generalized Newtonian fluid mechanics are developed.

Another logarithmically perturbed energy closely related to \eqref{log-functional} was considered by Marcellini \cite{Marcellini-1991}, namely
\begin{align*}%\label{Marcellini-functional}
	w \mapsto \int_{\Omega} (1+|D w|^2)^\frac{p}{2} \log (1 + |D w|) \,\mathrm{d}x.
\end{align*}
This functional  may be viewed as a prototype for variational integrals with nonstandard growth. Such growth conditions were systematically introduced by Marcellini in the seminal paper \cite{Marcellini-1991} and have since become a central topic in nonlinear regularity theory. We refer to the works of Baroni \cite{Baroni-2025,Baroni-2023}, Cupini--Marcellini--Mascolo \cite{Cupini-Marcellini-Mascolo-2023}, De Filippis--Piccinini \cite{DeFilippis-Piccinini-2024}, Giannetti--Passarelli di Napoli \cite{Giannetti-Napoli-2013},  Marcellini \cite{Marcellini-2023}, as well as Ragusa--Tachikawa \cite{Ragusa-Tachikawa-2024},  for further developments and a comprehensive overview of the current state of the art.

The functional \eqref{log-functional} can be viewed as a logarithmically perturbed counterpart of the classical double phase functional
\begin{align}\label{functional-double-phase}
	\int_\Omega \left( |\nabla u|^{p} + \mu(x) |\nabla u|^{q} \right) \,\mathrm{d} x
\end{align}
which is associated with the standard double phase operator
\begin{align}\label{DO}
	-\operatorname{div}(|\nabla u|^{p-2}\nabla u +\mu(x)|\nabla u|^{q-2}\nabla u),
\end{align}
see Crespo-Blanco--Gasi\'{n}ski--Harjulehto--Winkert \cite{Crespo-Blanco-Gasinski-Harjulehto-Winkert-2022} for more details of this operator. Functionals of this type originated in the pioneering works of Marcellini \cite{Marcellini-1989,Marcellini-1991} and were subsequently further developed by Zhikov \cite{Zhikov-1995}. A distinctive feature of the functional \eqref{functional-double-phase} is that its ellipticity changes according to the behavior of the coefficient $\mu$. More precisely, on the set $\{x\in \Omega\colon \mu(x)=0\}$ the energy exhibits $p$-growth, whereas in regions where $\mu(x)>0$ the $q$-growth phase becomes active. This interplay between two different growth regimes is precisely what motivates the terminology double phase. From a modelling perspective, such functionals provide a natural framework for describing strongly heterogeneous and anisotropic media. In particular, when a body is composed of different constituents, the coefficient $\mu$ may be chosen to reflect their spatial distribution, thereby capturing transitions between distinct mechanical responses. Owing to this flexibility, double phase energies have proved useful in a variety of contexts, ranging from elasticity theory and duality theory to the study of the Lavrentiev phenomenon, see, for instance, Zhikov \cite{Zhikov-1986,Zhikov-1995,Zhikov-2011} and the monograph of Zhikov--Kozlov--Ole\u{\i}nik \cite{Zhikov-Kozlov-Oleinik-1994}. Over the past decades, double phase problems have attracted considerable attention and have developed into an active area of research. Beyond their mathematical significance, they have found applications in several branches of mathematical physics. We refer, for example, to the works of Bahrouni--R\u{a}dulescu--Repov\v{s} \cite{Bahrouni-Radulescu-Repovs-2019} on transonic flows, Benci--D'Avenia--Fortunato--Pisani  \cite{Benci-DAvenia-Fortunato-Pisani-2000} in quantum physics, and Cherfils--Il'yasov \cite{Cherfils-Ilyasov-2005} in the study of reaction-diffusion equations.

It is worth emphasizing that double phase functionals of the form \eqref{functional-double-phase} represent only a particular instance of the broader class of variational integrals with nonstandard growth and $(p,q)$-growth conditions introduced by Marcellini in his pioneering works \cite{Marcellini-1991,Marcellini-1989}. In fact, the regularity theory developed in \cite{Marcellini-1991} already encompasses the double phase framework as a special case. We also refer to the more recent contributions of Cupini--Marcellini--Mascolo \cite{Cupini-Marcellini-Mascolo-2023} and Marcellini \cite{Marcellini-2023}, where nonstandard growth problems with explicit dependence on the solution $u$ are investigated, as well as to Marcellini \cite{Marcellini-2021} for further developments in the theory and related references to recent advances in the field. Subsequently, the regularity theory for double phase problems was significantly refined through a series of influential papers by Baroni--Colombo--Mingione \cite{Baroni-Colombo-Mingione-2015,Baroni-Colombo-Mingione-2016,Baroni-Colombo-Mingione-2018} and Colombo--Mingione \cite{Colombo-Mingione-2015a,Colombo-Mingione-2015b}. A major achievement of these works was the substantial relaxation of the assumptions imposed on the modulating coefficient $\mu$. More precisely, while the regularity results in \cite{Marcellini-1991} require $\mu$ to be Lipschitz continuous in the double phase setting, the aforementioned papers establish analogous regularity properties under the considerably weaker assumption that $\mu$ is merely H\"{o}lder continuous.

Further regularity results for double phase functionals, nonautonomous variational integrals, and related classes of nonuniformly elliptic problems have been obtained in a large number of contributions. We refer, among others, to the works of Baasandorj--Byun--Oh \cite{Baasandorj-Byun-Oh-2020}, Baroni--Kuusi--Mingione \cite{Baroni-Kuusi-Mingione-2015}, Beck--Mingione \cite{Beck-Mingione-2020,Beck-Mingione-2019}, Byun--Cho--Ryu \cite{Byun-Cho-Ryu-2026}, Byun--Lee--Song \cite{Byun-Lee-Song-2025}, Byun--Oh \cite{Byun-Oh-2020}, Byun--Ok--Song \cite{Byun-Ok-Song-2022}, De Filippis \cite{De-Filippis-2018}, De Filippis--Mingione \cite{DeFilippis-Mingione-2021-a,DeFilippis-Mingione-2021-b,DeFilippis-Mingione-2020-a,DeFilippis-Mingione-2020-b}, De Filippis--Nastasi--Pacchiano \cite{DeFilippis-Nastasi-Pacchiano-2026}, De Filippis--Oh \cite{DeFilippis-Oh-2019}, De Filippis--Palatucci \cite{DeFilippis-Palatucci-2019},  Harjulehto--H\"{a}st\"{o}--Toivanen \cite{Harjulehto-Hasto-Toivanen-2017}, H\"{a}st\"{o}--Ok \cite{Hasto-Ok-2019}, Kim--Oh \cite{Kim-Oh-2025}, Nastasi--Pacchiano \cite{Nastasi-Pacchiano-2025}, Ok \cite{Ok-2018,Ok-2020}, Ragusa--Tachikawa \cite{Ragusa-Tachikawa-2016,Ragusa-Tachikawa-2020}, Song--Youn \cite{Song-Youn-2026}, and Song--Youn--Zatorska-Goldstein \cite{Song-Youn-ZatorskaGoldstein-2026}. For a comprehensive overview of recent developments in the theory of nonstandard growth and nonuniformly elliptic problems, we refer the reader to the survey article by Mingione--R\u{a}dulescu \cite{Mingione-Radulescu-2021}.

It is worth noting that a substantial part of the regularity theory for double phase problems concerns minimizers of variational functionals. In this setting, the variational structure provides a powerful tool for deriving Caccioppoli-type estimates and comparison inequalities, which constitute fundamental ingredients in regularity arguments based on the De Giorgi method. The situation is substantially different for general weak solutions of equations of the form
\begin{align*}
	-\operatorname{div}\mathcal A(x,u,\nabla u) =\mathcal B(x,u,\nabla u).
\end{align*}
In this case, the required energy estimates have to be derived directly from the weak formulation. In particular, when suitable truncations are used as test functions, the lower-order term involving $\mathcal B$ has to be controlled in a way that is compatible with the nonstandard growth of the principal part. Obtaining Caccioppoli-type estimates with the appropriate structure therefore represents a significant analytical difficulty and is a crucial step in initiating the De Giorgi iteration. Comparatively few results seem to be available for the H\"older regularity of general weak solutions to double phase equations with variable exponents. In this direction, we mention in particular the recent work of Ri--Kwon \cite{Ri-Kwon-2025}, who considered the classical variable exponent double phase setting.

In the present paper, we investigate boundedness and H\"{o}lder continuity properties of weak solutions to a broad class of generalized logarithmic double phase problems of the form
\begin{align}\label{D}
	-\operatorname{div}\mathcal{A}(x,u,\nabla u) =\mathcal{B}(x,u,\nabla u)\quad  \text{in } \Omega,  \quad u = 0  \quad \text{on } \partial \Omega.
\end{align}
Our framework encompasses a large family of nonuniformly elliptic equations whose principal part exhibits variable exponent logarithmic double phase growth. A prototypical example is provided by the operator
\begin{align*}%\label{main:operator}
	\operatorname{div}\bigg (a(x) |\nabla u|^{p(x)-2} \nabla u+ b(x)  \left(\log(\delta+ \eta|\nabla u|) + \frac{\eta|\nabla u|}{q(x)\left(\delta+ \eta|\nabla u|\right)} \right)|\nabla u|^{q(x)-2} \nabla u\bigg),
\end{align*}
which extends the logarithmic double phase operator and the classical double phase operator introduced in \eqref{log-operator} and \eqref{DO}, respectively. A distinctive feature of our setting is that the reaction term $\mathcal{B}$ is allowed to exhibit critical growth with respect to the underlying Musielak-Orlicz-Sobolev structure. More precisely, our assumptions are compatible with the continuous embedding $W^{1,\mathcal{S}}(\Omega)$ into $L^{\mathcal{S}^{\ast}}(\Omega)$, where $W^{1,\mathcal{S}}(\Omega)$ denotes the Musielak-Orlicz Sobolev space generated by the generalized $N$-function
\begin{align}\label{S}
	\mathcal{S}(x,t)=a(x) t^{p(x)}+b(x) t^{q(x)}\log(\delta +\eta t) \quad \text{for } (x, t) \in \Omega \times (0, \infty),
\end{align}
and $L^{\mathcal{S}^{\ast}}(\Omega)$ is the corresponding Musielak-Orlicz space generated by the associated critical growth function
\begin{align}\label{S*}
	\mathcal{S}^\ast(x,t):= \left( a(x)^\frac{1}{p(x)} t\right)^{p^\ast(x)}  + \left(\left(b(x) \log(\delta+\eta t)\right)^\frac{1}{q(x)} t  \right)^{q^\ast(x)}
\end{align}
for $(x, t) \in \Omega \times (0, \infty)$. Here, $\kappa^*(x):=\frac{N \kappa(x)}{N-\kappa(x)}$ denotes the Sobolev conjugate of $\kappa \in C(\overline{\Omega})$. For precise definitions of these spaces, we refer to Section~\ref{Section-2}.

We impose the following general assumptions:
\begin{enumerate}
	\item[\textnormal{(H$_0$)}]\hypertarget{H0}{}
		\begin{enumerate}
			\item[\textnormal{(i)}]\hypertarget{H0i}{}
				$\Omega \subset \mathbb{R}^N$, $N \geq 2$, is a bounded domain with Lipschitz boundary $\partial \Omega$;
			\item[\textnormal{(ii)}]\hypertarget{H0ii}{}
				$p, q \in C(\overline{\Omega})$ with $1 < p(x)\leq q(x)< N$ for all $x \in \overline{\Omega}$;
			\item[\textnormal{(iii)}]\hypertarget{H0iii}{}
				$a, b\in L^1(\Omega)$, $a,b\geq 0$, $a(x) + b(x) \geq d >0$ for a.a.\,$x \in \Omega$, and $1\leq \delta\leq e$, $\eta\geq 0$, $\delta+\eta>1$.
		\end{enumerate}
\end{enumerate}
In order to obtain optimal compact and continuous embeddings for $W^{1,\mathcal{S}}(\Omega)$, in addition to assumption \textnormal{(\hyperlink{H0}{H$_0$})}, we further assume the following:
\begin{enumerate}
	\item[\textnormal{(H$_1$)}]\hypertarget{H1}{}
	\begin{enumerate}
		\item[\textnormal{(i)}]\hypertarget{H1i}{}
			$p, q \in C^{0,1}(\overline{\Omega})$ and $a, b \in C^{0,1}(\overline{\Omega})$;
		\item[\textnormal{(ii)}]\hypertarget{H1ii}{}
			$\frac{q(x)}{p(x)} < 1 + \frac{1}{N}$ for all $x\in \overline{\Omega}$.
	\end{enumerate}
\end{enumerate}
First, we are interested in a priori bounds for problem \eqref{D} under the following subcritical growth conditions:
\begin{enumerate}
	\item[\textnormal{(H$_2$)}]\hypertarget{H2}{}
		The functions $\mathcal{A}\colon\Omega\times\R\times\R^N\to \R^N$ and $\mathcal{B}\colon\Omega \times \R\times \R^N\to \R$ are Carath\'eodory functions such that
		\begin{enumerate}[leftmargin=1cm]
			\item[\textnormal{(i)}]\hypertarget{D1i}{}
				$\displaystyle
				\begin{aligned}[t]
					&|\mathcal{A}(x,t,\xi)|\\
					&\leq \alpha_1 \Big[a(x)^{\frac{N-1}{N-p(x)}}|t|^{\frac{p^*(x)}{p'(x)}}+ \left(b(x)\log\big(\delta+\eta|t|\big)\right)^{\frac{N-1}{N-q(x)}}|t|^{\frac{q^*(x)}{q'(x)}}\\
					&\qquad\quad+a(x)|\xi|^{p(x)-1} + b(x)|\xi|^{q(x)-1}\log\big(\delta+\eta|\xi|\big)+1\Big]
				\end{aligned}
				$
			\item[\textnormal{(ii)}]\hypertarget{H2ii}{}
				$\displaystyle
				\begin{aligned}[t]
					&\mathcal{A}(x,t,\xi)\cdot \xi\\
					&\geq \alpha_2 \big[a(x)|\xi|^{p(x)}+ b(x)|\xi|^{q(x)}\log\big(\delta+\eta|\xi|\big)\big]\\ &\quad-\alpha_3\Big[a(x)^{\frac{r(x)}{p(x)}}|t|^{r(x)}+\left(b(x)\log\big(\delta+\eta|t|\big)\right)^{\frac{s(x)}{q(x)}}|t|^{s(x)}+1\Big],
				\end{aligned}
				$
			\item[\textnormal{(iii)}]\hypertarget{H2iii}{}
				$\displaystyle
				\begin{aligned}[t]
					&|\mathcal{B}(x,t,\xi)|\\
					&\leq \beta \Big[a(x)^{\frac{r(x)}{p(x)}}|t|^{r(x)-1}+\left(b(x)\log\big(\delta+\eta|t|\big)\right)^{\frac{s(x)}{q(x)}}|t|^{s(x)-1}+a(x)^{\frac{1}{p(x)}+\frac{1}{r'(x)}}|\xi|^{\frac{p(x)}{r'(x)}}\\ &\qquad\quad+\left(b(x)\log\big(\delta+\eta|t|\big)\right)^{\frac{1}{q(x)}}\left(b(x)\log\big(\delta+\eta|\xi|\big)\right)^{\frac{1}{s'(x)}}|\xi|^{\frac{q(x)}{s'(x)}}+1 \Big],
				\end{aligned}
				$
			\end{enumerate}
			for a.a.\,$x\in\Omega$ and for all $(t,\xi) \in \R\times\R^N$, where $\alpha_1,\alpha_2,\alpha_3$, $\beta$ are positive constants,  $r,s\in C(\overline{\Omega})$ satisfy $q(x)<r(x)<p^*(x)$ and $q(x)<s(x)<q^*(x)$ for all $x\in\overline{\Omega}$, and  $\kappa'(\cdot):=\frac{\kappa(\cdot)}{\kappa(\cdot)-1}$ for $\kappa(\cdot)>1$. Furthermore, when $\delta=1$, it also holds that $b(x)\leq C a(x)$ for a.a.\,$x\in\Omega$ for some constant $C>0$.
\end{enumerate}
We say that $ u\in W_0^{1,\mathcal{S}}(\Omega) $ is a weak solution of problem \eqref {D}  if
\begin{align}\label{def_sol_D}
	\int_{\Omega}\mathcal{A}(x,u,\nabla u)\cdot \nabla \varphi \,\mathrm{d} x
	= \int_{\Omega}\mathcal{B}(x,u, \nabla u)\varphi \,\mathrm{d} x
\end{align}
for all $\ph \in W_0^{1,\mathcal{S}}(\Omega)$, where $\mathcal{S}$ is given in \eqref{S}. Based on Proposition \ref{emb-main} along with hypotheses \textnormal{(\hyperlink{H0}{H$_0$})}, \textnormal{(\hyperlink{H1}{H$_1$})} and \textnormal{(\hyperlink{H2}{H$_2$})}, we know that all terms in the weak formulation \eqref{def_sol_D} are well defined.

We have the following first result.

\begin{theorem}\label{D.a-priori}
	Let hypotheses \textnormal{(\hyperlink{H0}{H$_0$})}, \textnormal{(\hyperlink{H1}{H$_1$})}, and \textnormal{(\hyperlink{H2}{H$_2$})} be satisfied. Then, every weak solution $u\in \Wpzero{\mathcal{S}}$ of problem \eqref{D} belongs to $L^\infty(\Omega)$ and satisfies the a priori estimate
	\begin{align}\label{D-bound}
		\|u\|_{L^\infty(\Omega)} \leq C \max \left\{\|u\|_{\Psi,\Omega}^{\tau_1},\|u\|_{\Psi,\Omega}^{\tau_2} \right\},
	\end{align}
	where $C,\tau_1,\tau_2$ are positive constants independent of $u$, and $\|\cdot\|_{\Psi,\Omega}$ denotes the Luxemburg norm of the Musielak-Orlicz space generated by
	\begin{align}\label{Psi}
		\Psi(x,t):= \left(a(x)^{\frac{1}{p(x)}} t\right)^{r(x)}  + \left(\left(b(x)\log(\delta+\eta t)\right)^\frac{1}{q(x)} t  \right)^{s(x)}
		\quad\text{for } (x,t)\in \close\times [0,\infty).
	\end{align}
\end{theorem}

Next, we extend Theorem \ref{D.a-priori} to the critical case. We impose the following assumptions:
\begin{enumerate}
	\item[\textnormal{(H$_3$)}]\hypertarget{H3}{}
		The functions $\mathcal{A}\colon\Omega\times\R\times\R^N\to \R^N$ and $\mathcal{B}\colon\Omega \times \R\times \R^N\to \R$ are Carath\'eodory functions such that
	\begin{enumerate}[leftmargin=1cm]
		\item[\textnormal{(i)}]\hypertarget{H3i}{}
			$\begin{aligned}[t]
				&|\mathcal{A}(x,t,\xi)|\\
				&\leq \alpha_1 \Big[a(x)^{\frac{N-1}{N-p(x)}}|t|^{\frac{p^*(x)}{p'(x)}}+ \left(b(x)\log\big(\delta+\eta|t|\big)\right)^{\frac{N-1}{N-q(x)}}|t|^{\frac{q^*(x)}{q'(x)}}\\
				&\qquad\quad+a(x)|\xi|^{p(x)-1}+ b(x)|\xi|^{q(x)-1}\log\big(\delta+\eta|\xi|\big)+1\Big],
			\end{aligned}$
		\item[\textnormal{(ii)}]\hypertarget{H3ii}{}
			$\begin{aligned}[t]
				&\mathcal{A}(x,t,\xi)\cdot \xi\\
				&\geq \alpha_2 \big[a(x)|\xi|^{p(x)}+ b(x)|\xi|^{q(x)}\log\big(\delta+\eta|\xi|\big)\big]\\
				&\quad-\alpha_3\bigg[a(x)^{\frac{p^*(x)}{p(x)}}|t|^{p^*(x)}+\left(b(x)\log\big(\delta+\eta|t|\big)\right)^{\frac{q^*(x)}{q(x)}}|t|^{q^*(x)}+1\bigg],
			\end{aligned}$
		\item[\textnormal{(iii)}]\hypertarget{H3iii}{}
			$\begin{aligned}[t]
				&|\mathcal{B}(x,t,\xi)|\\
				&\leq \beta \Big[a(x)^{\frac{p^*(x)}{p(x)}}|t|^{p^*(x)-1}+\left(b(x)\log\big(\delta+\eta|t|\big)\right)^{\frac{q^*(x)}{q(x)}}|t|^{q^*(x)-1}+a(x)^{\frac{N+1}{N}}|\xi|^{\frac{p(x)}{(p^*)'(x)}}\\
				&\qquad\quad+\left(b(x)\log\big(\delta+\eta|t|\big)\right)^{\frac{1}{q(x)}}\left(b(x)\log\big(\delta+\eta|\xi|\big)\right)^{\frac{1}{(q^*)'(x)}}|\xi|^{\frac{q(x)}{(q^*)'(x)}}+1 \Big],
			\end{aligned}$
	\end{enumerate}
	for a.a.\,$x\in\Omega$ and for all $(t,\xi) \in \R\times\R^N$, where $\alpha_1,\alpha_2,\alpha_3$ and $\beta$ are positive constants. Furthermore, when $\delta=1$, it also holds that $b(x)\leq C a(x)$ for a.a.\,$x\in\Omega$ for some constant $C>0$.
\end{enumerate}
The definition of a weak solution of problem \eqref{D} under hypothesis \textnormal{(\hyperlink{H3}{H$_3$})} is the same as in \eqref{def_sol_D}. Again, by Proposition \ref{emb-main} together with hypotheses \textnormal{(\hyperlink{H0}{H$_0$})}, \textnormal{(\hyperlink{H1}{H$_1$})} and \textnormal{(\hyperlink{H3}{H$_3$})}, this definition is well defined.

We have the following result.

\begin{theorem} \label{CD.boundedness}
	Let hypotheses \textnormal{(\hyperlink{H0}{H$_0$})}, \textnormal{(\hyperlink{H1}{H$_1$})}, and \textnormal{(\hyperlink{H3}{H$_3$})} be satisfied. Then, every weak solution $u\in \Wpzero{\mathcal{S}}$ of problem \eqref{D} belongs to $L^\infty (\Omega)$.
\end{theorem}

The proofs of Theorems \ref{D.a-priori} and \ref{CD.boundedness} rely heavily on the celebrated De Giorgi--Nash--Moser theory, which provides powerful iterative techniques based on truncation arguments for deriving a priori estimates for solutions of partial differential equations. The foundations of this theory were laid in the seminal works of De Giorgi \cite{De-Giorgi-1957}, Nash \cite{Nash-1958}, and Moser \cite{Moser-1960}. Since then, these methods have become indispensable tools in regularity theory and have led to fundamental results such as local and global boundedness, Harnack and weak Harnack inequalities, as well as H\"{o}lder continuity of weak solutions. For comprehensive accounts of this theory and its applications, we refer to the classical monographs of Gilbarg--Trudinger \cite{Gilbarg-Trudinger-1983}, Lady{\v{z}}enskaja--Ural{\cprime}ceva \cite{Ladyzenskaja-Uralceva-1968},  Lady{\v{z}}enskaja--Solonnikov--Ural{\cprime}ceva \cite{Ladyzenskaja-Solonnikov-Uralceva-1968} and Lieberman \cite{Lieberman-1996}. The boundedness results established in Theorems \ref{D.a-priori} and \ref{CD.boundedness} generalize a number of previous boundedness and regularity results, including those of Borer--Gasi\'{n}ski--Stapenhorst--Winkert \cite{Borer-Gasinski-Stapenhorst-Winkert-2025}, Ho--Kim \cite{Ho-Kim-2019}, Ho--Sim \cite{Ho-Sim-2017}, Ho--Kim--Winkert--Zhang \cite{Ho-Kim-Winkert-Zhang-2022}, Gasi\'{n}ski--Winkert \cite{Gasinski-Winkert-2020a, Gasinski-Winkert-2021}, Kim--Kim--Oh--Zeng \cite{Kim-Kim-Oh-Zeng-2022}, Marino--Winkert \cite{Marino-Winkert-2020, Marino-Winkert-2019}, Winkert \cite{Winkert-2010}, Winkert--Zacher \cite{Winkert-Zacher-2012,Winkert-Zacher-2015}. In a related direction, we also mention the boundedness results obtained by Barletta--Cianchi--Marino \cite{Barletta-Cianchi-Marino-2022} in the framework of Orlicz spaces.

In the last part of the paper, we are interested in the H\"{o}lder continuity of weak solutions to the elliptic equation
\begin{align}\label{E-H}
	-\operatorname{div} \mathcal{A}(x, u,\nabla u) = \mathcal{B}(x, u, \nabla u) \quad \text{in } \Omega.
\end{align}
We say that a function $f \colon \overline{\Omega} \to \mathbb{R}$ is log-H\"{o}lder continuous on $\overline{\Omega}$ if there exists a constant $C_{\log}(f) > 0$ such that
\begin{align}\label{def-log-continuity}
	-|f(x) - f(y)|\log |x - y| \leq C_{\log}(f) \quad \text{for all } x, y \in \overline{\Omega} \text{ with } 0<|x - y| \leq \frac{1}{2}.
\end{align}
We denote the set of all such functions by $C^{0, \frac{1}{|\log t|}}(\overline{\Omega})$. We impose the following assumptions, recalling the definition of  $\mathcal{S}$ given in \eqref{S}:
\begin{enumerate}
	\item[\textnormal{(H$_4$)}]\hypertarget{H4}{}
	\begin{enumerate}[leftmargin=1cm]
		\item[\textnormal{(i)}]\hypertarget{H4i}{}
			$p, q \in C^{0, \frac{1}{|\log t|}}(\overline{\Omega})$;
		\item[\textnormal{(ii)}]\hypertarget{H4ii}{}
			$0 \leq a(\cdot)\in L^\infty(\Omega)$ and $0 \leq b(\cdot) \in C^{0, \alpha}(\overline{\Omega}) \text{ for some } \alpha \in (0,1]$.
%		\item[\textnormal{(iii)}]\hypertarget{H4iii}{}
%			$\frac{q(x)}{p(x)} < 1 + \frac{1}{N}$ for all $x\in \overline{\Omega}$.
	\end{enumerate}
\end{enumerate}

\begin{enumerate}
	\item[\textnormal{(H$_5$)}]\hypertarget{H5}{}
	The functions $\mathcal{A}\colon\Omega\times\R\times\R^N\to \R^N$ and $\mathcal{B}\colon\Omega \times \R\times \R^N\to \R$ are Carath\'eodory functions such that
	\begin{enumerate}[leftmargin=1cm]
		\item[\textnormal{(i)}]\hypertarget{H5i}{}
			$|\mathcal{A}(x,t,  \xi)| \leq a_1(|t|)\left(\mathbf{s}(x,|\xi|) + 1\right)$;
		\item[\textnormal{(ii)}]\hypertarget{H5ii}{}
			$\mathcal{A}(x, t, \xi)\cdot\xi \geq a_2(|t|)\mathcal{S}(x, |\xi|)-a_3$;
		\item[\textnormal{(iii)}]\hypertarget{H5iii}{}
			$|\mathcal{B}(x, t, \xi)| \leq b_1(|t|)\left(\mathcal{S}(x, |\xi|) + b_2\right)$
	\end{enumerate}
	for a.a.\,$x \in \Omega$ and for $(t,\xi) \in \mathbb{R} \times \mathbb{R}^N$, where $a_3,b_2$ are positive constants, $a_1,b_1\colon [0,\infty)\to (0,\infty)$ are nondecreasing continuous functions, $a_2\colon [0,\infty)\to (0,\infty)$ is a nonincreasing continuous function, and
	\begin{align*}
		\qquad\mathbf{s}(x, t):=\partial_t\mathcal{S}(x,t)= p(x)a(x)t^{p(x)-1} + q(x)b(x) t^{q(x)-1}\log (\delta+\eta t)+\frac{\eta b(x) t^{q(x)} }{\delta+\eta t}
	\end{align*}
	for $(x,t )\in \Omega\times [0,\infty)$.
\end{enumerate}
We say that $u \in W^{1,\mathcal{S}}(\Omega)$ is a weak solution of problem \eqref{E-H} if
\begin{align}\label{E-V}
	\int_\Omega \mathcal{A}(x, u,\nabla u) \cdot\nabla \varphi\,\mathrm{d} x = \int_\Omega \mathcal{B}(x, u, \nabla u) \varphi\,\mathrm{d} x
\end{align}
for every $\varphi \in W^{1,\mathcal{S}}(\Omega)\cap W_0^{1,1}(\Omega)$.

Our main results on H\"{o}lder continuity are as follows.

\begin{theorem}\label{Th.H}
	Let hypotheses \textnormal{(\hyperlink{H0}{H$_0$})},  \textnormal{(\hyperlink{H4}{H$_4$})}, and \textnormal{(\hyperlink{H5}{H$_5$})} be satisfied. Assume further that
	\begin{align}\label{TA3}
		C_*:=\sup_{(x,t)\in \Omega\times[1,\infty)}t^{q(x) - p(x)-\alpha}\log(\delta+\eta t)<+\infty.
	\end{align}
	Let $u \in W^{1,\mathcal{S}}(\Omega)$  be any bounded weak solution of problem \eqref{E-H} with $u\in C^{0,\beta_0}(\partial\Omega)$ for some $\beta_0\in (0,1]$. Then, $u$  belongs to $C^{0,\beta}(\overline{\Omega})$ for some  $\beta\in (0,\beta_0]$.
\end{theorem}

\begin{theorem}\label{Th.H'}
	Let hypotheses \textnormal{(\hyperlink{H0}{H$_0$})}, \textnormal{(\hyperlink{H1}{H$_1$})}, and  \textnormal{(\hyperlink{H3}{H$_3$})} be satisfied. Then, every weak solution $u\in \Wpzero{\mathcal{S}}$ of problem \eqref{D} belongs to  $C^{0,\beta}(\overline{\Omega})$ for some  $\beta\in (0,1]$.
\end{theorem}

It is worth pointing out that if $1<\delta\leq e$ and $\eta=0$, then condition \eqref{TA3} is equivalent to
\begin{align*}
	q(x)\leq p(x)+\alpha\quad \text{for all } x\in\overline{\Omega},
\end{align*}
which is precisely the balance condition appearing in the H\"{o}lder continuity result of Ri--Kwon \cite{Ri-Kwon-2025} for the classical variable exponent double phase setting with a right-hand side of growth not exceeding the $p^*(\cdot)$-critical level. If $\eta>0$, then \eqref{TA3} is satisfied if
\begin{align*}
	q(x)< p(x)+\alpha\quad \text{for all } x\in\overline{\Omega}.
\end{align*}

We next derive three direct consequences of Theorem~\ref{Th.H'} for relevant classes of double phase operators. For $1\leq\delta\leq e$ and $\eta\geq0$ with $\delta+\eta>1$, set
\begin{align*}
	\mathcal{H}_{\delta,\eta}(x,t) :=t^{p(x)}+\mu(x)t^{q(x)}\log(\delta+\eta t) \quad\text{for }(x,t)\in\Omega\times[0,\infty).
\end{align*}

\begin{corollary}[Generalized variable exponent logarithmic double phase problem]\label{corollary-general-log-double-phase}
	Let $\Omega\subset\R^N$, $N\geq2$, be a bounded domain with Lipschitz boundary $\partial\Omega$, $p,q\in C^{0,1}(\overline{\Omega})$, and $0\leq\mu\in C^{0,1}(\overline{\Omega})$. Assume that
	\begin{align*}
		1<p(x)\leq q(x)<N \quad\text{and}\quad \frac{q(x)}{p(x)}<1+\frac{1}{N} \quad\text{for all }x\in\overline{\Omega}.
	\end{align*}
	Let $1\leq\delta\leq e$ and $\eta\geq0$ satisfy $\delta+\eta>1$.
	Suppose that $f\colon\Omega\times\R\times\R^N\to\R$ is a Carath\'{e}odory
	function such that
	\begin{align*}
		|f(x,t,\xi)|
		\leq C\Bigg[& |t|^{p^*(x)-1} +\big(\mu(x)\log(\delta+\eta|t|)\big)^{\frac{q^*(x)}{q(x)}} |t|^{q^*(x)-1}\\ &+|\xi|^{\frac{p(x)}{(p^*)'(x)}} +\big(\mu(x)\log(\delta+\eta|t|)\big)^{\frac1{q(x)}} \big(\mu(x)\log(\delta+\eta|\xi|)\big)^{\frac1{(q^*)'(x)}} |\xi|^{\frac{q(x)}{(q^*)'(x)}}+1\Bigg]
	\end{align*}
	for a.a.\,$x\in\Omega$ and for all $(t,\xi)\in\R\times\R^N$, where $C>0$.
	Then every weak solution $u\in W_0^{1,\mathcal{H}_{\delta,\eta}}(\Omega)$ of
	\begin{align*}
		&-\operatorname{div}\left( |\nabla u|^{p(x)-2}\nabla u +\mu(x)\left( \log(\delta+\eta|\nabla u|) +\dfrac{\eta|\nabla u|}{q(x)(\delta+\eta|\nabla u|)} \right) |\nabla u|^{q(x)-2}\nabla u \right)\\
		&=f(x,u,\nabla u) \quad \text{in }\Omega,
	\end{align*}
	with $u=0$ on $\partial\Omega$, belongs to $C^{0,\beta}(\overline{\Omega})$ for some $\beta\in(0,1]$.
\end{corollary}

In particular, by choosing $\delta=e$ and $\eta=0$, we obtain the following result for the variable exponent double phase operator.

\begin{corollary}[Variable exponent double phase problem]
	Let $\Omega$, $p$, $q$, and $\mu$ be as in Corollary \ref{corollary-general-log-double-phase}. Suppose that $f\colon\Omega\times\R\times\R^N\to\R$ is a Carath\'{e}odory function satisfying
	\begin{align*}
		|f(x,t,\xi)| \leq C\Big[& |t|^{p^*(x)-1} +\mu(x)^{\frac{q^*(x)}{q(x)}}|t|^{q^*(x)-1} +|\xi|^{\frac{p(x)}{(p^*)'(x)}}\\ &+\mu(x)^{\frac1{q(x)}+\frac1{(q^*)'(x)}} |\xi|^{\frac{q(x)}{(q^*)'(x)}}+1\Big]
	\end{align*}
	for a.a.\,$x\in\Omega$ and for all $(t,\xi)\in\R\times\R^N$, where $C>0$. Then every weak solution $u\in W_0^{1,\mathcal{H}}(\Omega)$, where
	\begin{align*}
		\mathcal{H}(x,t):=t^{p(x)}+\mu(x)t^{q(x)},
	\end{align*}
	of
	\begin{alignat*}{2}
		-\operatorname{div}\left( |\nabla u|^{p(x)-2}\nabla u +\mu(x)|\nabla u|^{q(x)-2}\nabla u \right) &=f(x,u,\nabla u)
		\quad &&\text{in }\Omega,\\
		u&=0&&\text{on }\partial\Omega,
	\end{alignat*}
	belongs to $C^{0,\beta}(\overline{\Omega})$ for some $\beta\in(0,1]$.
\end{corollary}

Choosing $\delta=e$ and $\eta=1$ in Corollary \ref{corollary-general-log-double-phase} gives the corresponding result for the variable exponent logarithmic double phase operator recently introduced by Arora--Crespo-Blanco--Winkert \cite{Arora-Crespo-Blanco-Winkert-2025}.

\begin{corollary}[Variable exponent logarithmic double phase problem]
	Let $\Omega$, $p$, $q$, and $\mu$ be as in Corollary \ref{corollary-general-log-double-phase}. Suppose that $f\colon\Omega\times\R\times\R^N\to\R$ is a Carath\'eodory function satisfying
	\begin{align*}
		|f(x,t,\xi)| \leq C\Big[& |t|^{p^*(x)-1} +\big(\mu(x)\log(e+|t|)\big)^{\frac{q^*(x)}{q(x)}} |t|^{q^*(x)-1}\\ &+|\xi|^{\frac{p(x)}{(p^*)'(x)}} +\big(\mu(x)\log(e+|t|)\big)^{\frac1{q(x)}} \big(\mu(x)\log(e+|\xi|)\big)^{\frac1{(q^*)'(x)}} |\xi|^{\frac{q(x)}{(q^*)'(x)}}+1\Big]
	\end{align*}
	for a.a.\,$x\in\Omega$ and all $(t,\xi)\in\R\times\R^N$, where $C>0$. Then every weak solution $u\in W_0^{1,\mathcal{H}_{\log}}(\Omega)$, where
	\begin{align*}
		\mathcal{H}_{\log}(x,t) :=t^{p(x)}+\mu(x)t^{q(x)}\log(e+t),
	\end{align*}
	of
	\begin{align*}
		&-\operatorname{div}\Bigg( |\nabla u|^{p(x)-2}\nabla u +\mu(x)\left( \log(e+|\nabla u|) +\dfrac{|\nabla u|}{q(x)(e+|\nabla u|)}\right) |\nabla u|^{q(x)-2}\nabla u \Bigg)\\
		&=f(x,u,\nabla u) \quad\text{in }\Omega,
	\end{align*}
	with $u=0$ on $\partial\Omega$,
	belongs to $C^{0,\beta}(\overline{\Omega})$ for some $\beta\in(0,1]$.
\end{corollary}

We also point out the recent work of Missaoui--Bahrouni--Ounaies \cite{Missaoui-Bahrouni-Ounaies-2026}, who established H\"{o}lder continuity results for a broad class of nonlinear elliptic equations. However, the assumptions imposed in \cite{Missaoui-Bahrouni-Ounaies-2026} do not encompass the double phase framework. Finally, the paper by Crespo-Blanco--V\'elez-Santiago--Winkert \cite{Crespo-Blanco-Velez-Santiago-Winkert-2024} also deals with global H\"older regularity for elliptic and parabolic double phase problems, but with special right-hand sides. The approach in that paper is completely different from ours and it also does not cover the logarithmic double phase operator.

We emphasize that the contribution of the present paper is not simply the passage from constant to variable exponents. Rather, we develop boundedness and global H\"older regularity for general weak solutions within a generalized logarithmic double phase framework with variable exponents. Our setting also allows critical growth with respect to the associated Musielak-Orlicz Sobolev structure and covers both nondegenerate and degenerate logarithmic regimes. To the best of our knowledge, this combination is not covered by the existing regularity results discussed above.

The paper is organized as follows. In Section \ref{Section-2}, we collect some preliminary material on Musielak-Orlicz and Musielak-Orlicz Sobolev spaces and present a continuous embedding result that serves as a key tool in the subsequent analysis. Section \ref{Section-3} is concerned with a priori estimates for subcritical logarithmic double phase problems, while Section \ref{Section-4} addresses the boundedness of solutions under critical growth assumptions. Finally, in Section \ref{Holder}, we establish H\"{o}lder regularity of bounded weak solutions to problem \eqref{E-H} and of weak solutions to problem \eqref{D}.

%********************************************************************
\section{Preliminaries}\label{Section-2}
%********************************************************************

In this section, we present the properties of the underlying function spaces and tools that are needed in the sequel. Most of the results are taken from Arora--Crespo-Blanco--Winkert \cite{Arora-Crespo-Blanco-Winkert-2025-b,Arora-Crespo-Blanco-Winkert-2025}, Chlebicka--Gwiazda--\'{S}wierczewska-Gwiazda--Wr\'{o}blewska-Kami\'{n}ska \cite{Chlebicka-Gwiazda-Swierczewska-Gwiazda-Wroblewska-Kaminska-2021}, Diening--Harjulehto\allowbreak--H\"{a}st\"{o}--R\u{u}\v{z}i\v{c}ka \cite{Diening-Harjulehto-Hasto-Ruzicka-2011}, Fan \cite{Fan-2012}, Fan--Zhao \cite{Fan-Zhao-2001}, Harjulehto--H\"{a}st\"{o} \cite{Harjulehto-Hasto-2019}, Ho--Winkert \cite{Ho-Winkert-2023}, Kov{\'a}{\v{c}}ik--R{\'a}kosn{\'{\i}}k \cite{Kovacik-Rakosnik-1991}, Musielak \cite{Musielak-1983}, and Papageorgiou--Winkert \cite{Papageorgiou-Winkert-2024}.

Let $\Omega$ be a bounded domain in $\R^N$, $N\geq 2$, with Lipschitz boundary $\partial\Omega$. Denote by $M(\Omega)$  the set of all measurable functions $u\colon\Omega\to\R$.

\begin{definition}
	Let $\varphi \colon \Omega \times (0,+\infty) \to \R$. We say that
	\begin{enumerate}
		\item[\textnormal{(i)}]
			$\varphi$ is almost increasing in the second variable if there exists $C_a \geq 1$ such that $\varphi(x,s) \leq C_a \varphi(x,t)$ for all $0 < s < t$ and for a.a.\,$x \in \Omega$;
		\item[\textnormal{(ii)}]
			$\varphi$ is almost decreasing in the second variable if there exists $C_a \geq 1$ such that $C_a \varphi(x,s) \geq \varphi(x,t)$ for all $0 < s < t$ and for a.a.\,$x \in \Omega$.
	\end{enumerate}
\end{definition}

\begin{definition}
	Let $\varphi \colon \Omega \times (0,+\infty) \to \R$ and $p,q>0$. We say that $\varphi$ satisfies
	\begin{enumerate}[leftmargin=2cm]
		\item[\textnormal{(aInc)}$_p$]
			if $t^{-p}\varphi(x,t)$ is almost increasing in $t$ for a.a.\,$x\in\Omega$;
		\item[\textnormal{(aDec)}$_q$]
			if $t^{-q}\varphi(x,t)$ is almost decreasing in $t$ for a.a.\,$x\in\Omega$.
	\end{enumerate}
\end{definition}

Next, we give the definition of a generalized strong $\Phi$-function.

\begin{definition}
	A function $\varphi \colon \Omega \times [0,+\infty) \to [0,+\infty]$ is called a generalized  strong $\Phi$-function if it is measurable in the first variable and increasing in the second variable. Moreover, for a.a.\,$x \in \Omega$, the function $\varphi(x,\cdot)$ is continuous on $[0,\infty]$ and convex. In addition, it  satisfies
	\begin{align*}
		\varphi(x,0)=0,\quad \lim_{t\to 0^+} \varphi(x,t) = 0,\quad \text{and}\quad \lim_{t \to +\infty} \varphi(x,t) = +\infty.
	\end{align*}
	The set of all generalized strong $\Phi$-functions is denoted by $\Phi(\Omega)$.
\end{definition}

We use the adjective ``strong'' in accordance with the terminology of \cite{Arora-Crespo-Blanco-Winkert-2025-b,Arora-Crespo-Blanco-Winkert-2025,Harjulehto-Hasto-2019}, in order to distinguish this class from generalized weak $\Phi$-functions.

For each $\varphi\in \Phi(\Omega)$, we define
\begin{align*}
	L^\varphi(\Omega) = \{ u \in M(\Omega)\colon \varrho_{\varphi,\Omega} (\lambda u) < \infty \text{ for some } \lambda > 0 \},
\end{align*}
equipped with the associated Luxemburg norm
\begin{align*}
	\|u\|_{\varphi,\Omega} = \inf \left\{ \lambda > 0 \colon \varrho_{\varphi,\Omega} \left( \frac{u}{\lambda} \right)  \leq 1 \right\},
\end{align*}
where
\begin{align*}
	\varrho_{\varphi,\Omega} (u) = \int_\Omega \varphi(x,|u(x)|) \,\mathrm{d} x.
\end{align*}

It was shown by Harjulehto--H\"{a}st\"{o} \cite{Harjulehto-Hasto-2019} that $\left(L^\varphi(\Omega),\|\cdot\|_{\varphi,\Omega}\right)$ is a Banach space. We also define Musielak-Orlicz Sobolev spaces associated with $\varphi$ as follows. The space $W^{1,\varphi} (\Omega)$ is defined as
\begin{align*}
	W^{1,\varphi} (\Omega) = \left\{ u \in L^{\varphi}(\Omega) \colon |\nabla u| \in L^{\varphi}(\Omega) \right\},
\end{align*}
equipped with the Luxemburg norm
\begin{align*}
	\|u\|_{1,\varphi,\Omega} = \inf \left\{ \lambda > 0 \colon  \varrho_{1,\varphi,\Omega} \left( \frac{u}{\lambda} \right)  \leq 1 \right\},
\end{align*}
where
\begin{align*}
	\varrho_{1,\varphi,\Omega} (u) = \int_\Omega \left[\varphi(x,|u(x)|)+\varphi(x,|\nabla u(x)|)\right] \,\mathrm{d} x.
\end{align*}
We will also work with the space
\begin{align*}
	W^{1,\varphi}_0 (\Omega) = \overline{C_0^\infty (\Omega)}^{\|\cdot\|_{1,\varphi,\Omega}},
\end{align*}
where $C_0^\infty (\Omega)$ denotes the space of infinitely differentiable functions with compact support in $\Omega$.

These spaces $W^{1,\varphi} (\Omega)$ and $W^{1,\varphi}_0 (\Omega)$ are indeed Sobolev spaces, see Harjulehto--H\"{a}st\"{o} \cite{Harjulehto-Hasto-2019} for a proof.

\begin{proposition}%\label{P.MOS-S}
	Let $\varphi\in \Phi(\Omega)$ be such that $L^{\varphi}(\Omega) \subseteq L^1_{\operatorname{loc}} (\Omega)$.  Then both spaces $W^{1,\varphi} (\Omega)$ and $W_0^{1,\varphi} (\Omega)$ are Banach spaces. If $\varphi$ satisfies \textnormal{(aDec)$_q$} for some $q > 1$, then they are separable. If $\varphi$ satisfies \textnormal{(aInc)$_p$} and \textnormal{(aDec)$_q$} for some $p$, $q > 1$, then they possess an equivalent uniformly convex norm and  hence are reflexive.
\end{proposition}

Throughout the paper, we write $\varrho_{\varphi}(\cdot)$, $\varrho_{1,\varphi}(\cdot)$, $\|\cdot\|_{\varphi}$, and $\|\cdot\|_{1,\varphi}$ instead of $\varrho_{\varphi,\Omega}(\cdot)$, $\varrho_{1,\varphi,\Omega}(\cdot)$, $\|\cdot\|_{\varphi,\Omega}$ and $\|\cdot\|_{1,\varphi,\Omega}$, respectively. Furthermore, for $r \in C(\overline{\Omega})$, we set $r^- = \min_{x \in \overline{\Omega}} r(x)$, $r^+ = \max_{x \in \overline{\Omega}} r(x)$, and define
\begin{align*}
	C_+ (\Omega) = \{ r \in C(\overline{\Omega})\colon  1 < r^- \}.
\end{align*}

The next proposition establishes the relationship between the modulars and the corresponding norms, see Arora--Crespo-Blanco--Winkert \cite{Arora-Crespo-Blanco-Winkert-2025}.

\begin{proposition}\label{P.nor-mod-A}
	Let $\varphi\in \Phi(\Omega)$ satisfy \textnormal{(aInc)}$_p$ and \textnormal{(aDec)}$_q$, with $1 \leq p \leq q < \infty$. Then,
	\begin{align*}
		\frac{1}{C} \min \left\lbrace \|u\|_{\varphi}^p ,  \|u\|_{\varphi}^q \right\rbrace
		\leq \varrho_{\varphi} (u)
		\leq C \max \left\lbrace \|u\|_{\varphi}^p ,  \|u\|_{\varphi}^q \right\rbrace, \quad \text{for all }u \in L^\varphi(\Omega)
	\end{align*}
	and
	\begin{align*}
		\frac{1}{C} \min \left\lbrace \|u\|_{1,\varphi}^p ,  \|u\|_{1,\varphi}^q \right\rbrace
		\leq \varrho_{1,\varphi} (u)
		\leq C \max \left\lbrace \|u\|_{1,\varphi}^p ,  \|u\|_{1,\varphi}^q \right\rbrace,\quad \text{for all }u \in W^{1,\varphi} (\Omega),
	\end{align*}
	where $C$ is the maximum of the constants appearing in \textnormal{(aInc)}$_p$ and \textnormal{(aDec)}$_q$.
\end{proposition}

Define the function $\mathcal{S}\colon \Omega\times [0,\infty)\to [0,\infty)$ as in \eqref{S}. We have the following result, see Arora--Crespo-Blanco--Winkert \cite{Arora-Crespo-Blanco-Winkert-2025-b}.

\begin{proposition}%\label{aCon}
	Assume that \textnormal{(\hyperlink{H0}{H$_0$})}  holds. Then $\mathcal{S}\in \Phi(\Omega)$, $L^{\mathcal{S}}(\Omega) \subseteq L^1_{\operatorname{loc}} (\Omega)$, and $\mathcal{S}$ satisfies \textnormal{(aInc)$_{p^-}$} and \textnormal{(aDec)$_{q^++1}$}. Consequently, the spaces $L^{\mathcal{S}}(\Omega)$, $W^{1, \mathcal{S}}(\Omega)$ and $W_0^{1, \mathcal{S}}(\Omega)$ are reflexive Banach spaces.
\end{proposition}

Furthermore, the relations between the norms and modulars on $L^{\mathcal{S}}(\Omega)$ and $W^{1, \mathcal{S}}(\Omega)$ are stated in the following proposition. The result is taken from Arora--Crespo-Blanco--Winkert \cite{Arora-Crespo-Blanco-Winkert-2025-b}.

\begin{proposition}\label{P.nor-mod-L}
	Assume that \textnormal{(\hyperlink{H0}{H$_0$})} holds, and let $(X,\varrho_X,\|\cdot\|_X)$ be either $\left(L^{\mathcal{S}}(\Omega),\varrho_{\mathcal{S}},\|\cdot\|_{\mathcal{S}}\right)$ or $\left(W^{1, \mathcal{S}}(\Omega),\varrho_{1,\mathcal{S}},\|\cdot\|_{1, \mathcal{S}}\right)$.  Let $u \in X$, then the following assertions hold:
	\begin{enumerate}
		\item[\textnormal{(i)}]
			If $u \neq 0$, then $\|u\|_{X} = \eta$ if and only if $\varrho_{X} \left(\frac{u}{\eta}\right)=1$.
		\item[\textnormal{(ii)}]
			$\varrho_{X}(u) <1$ $($or $=1$ or $>1)$ $\Leftrightarrow \|u\|_{X}<1$ $($or $=1$ or $>1)$.
		\item[\textnormal{(iii)}]
			If $\|u\|_{X} < 1$, then $\|u\|^{q^++1}_{X} \leq \varrho_{X}(u) \leq \|u\|^{p^-}_{X}$.
		\item[\textnormal{(iv)}]
			If $\|u\|_{X} > 1$, then $\|u\|^{p^-}_{X} \leq \varrho_{X}(u) \leq  \|u\|^{q^++1}_{X}$.
		\item[\textnormal{(v)}]
			$\|u\|_{X} \to 0$ if and only if $\varrho_{X}(u) \to 0$.
		\item[\textnormal{(vi)}]
			$\|u\|_{X} \to \infty$ if and only if $\varrho_{X}(u) \to \infty$.
	\end{enumerate}
\end{proposition}

From the definition of the norm $\|\cdot\|_{1,\mathcal{S}}$ and Proposition~\ref{P.nor-mod-L}(i), we easily obtain
\begin{align*}%\label{EN}
	2^{-1}\|u\|_{1,\mathcal{S}}\leq \|u\|_{\mathcal{S}}+\|\nabla u\|_{\mathcal{S}}\leq 2^{\frac{1}{p^-}}\|u\|_{1,\mathcal{S}},\quad \text{for all } u\in W^{1,\mathcal{S}} (\Omega).
\end{align*}
Hence, $W^{1,\mathcal{S}}(\Omega)$ can be equipped with the equivalent norm $\|u\|_\mathcal{S}+\|\nabla u\|_\mathcal{S}$.

For the estimates in the following sections, we require a sharper estimate of the modular in terms of the norm. To this end, we first establish the following result.

\begin{proposition}\label{P.aCon}
	Assume that \textnormal{(\hyperlink{H0}{H$_0$})}  holds. If $\delta=1$, additionally assume that $a\in L^\infty(\Omega)$ and $b(x)\leq Ca(x)$ for a.a.\,$x\in\Omega$ with some constant $C>0$.  Let $\Psi$ be defined by \eqref{Psi}, where $r,s\in C_+ (\Omega)$. Then, $\Psi$ satisfies \textnormal{(aInc)$_{m^-}$} and \textnormal{(aDec)$_{n^++\eps}$} for every $\eps>0$, where $m(x):=\min\{r(x),s(x)\}$ and  $n(x):=\max\{r(x),s(x)\}$ for all $x\in\overline{\Omega}$. In particular, it holds that
	 \begin{align*}
	 	\frac{1}{C_\eps} \min \left\lbrace \|u\|_{\Psi}^{m^-} ,  \|u\|_{\Psi}^{n^++\eps} \right\rbrace
	 	\leq \varrho_{\Psi} (u)
	 	\leq C_\eps \max \left\lbrace \|u\|_{\Psi}^{m^-} ,  \|u\|_{\Psi}^{n^++\eps} \right\rbrace \quad \text{for all }u \in L^\Psi(\Omega),
	\end{align*}
	where $C_\eps$ is the constant appearing in \textnormal{(aDec)}$_{n^++\eps}$.
\end{proposition}

\begin{proof}
	Let \textnormal{(\hyperlink{H0}{H$_0$})} hold. Obviously, $\Psi$ satisfies \textnormal{(aInc)$_{m^-}$} with $C_a=1$. We next verify that $\Psi$ satisfies \textnormal{(aDec)$_{n^++\eps}$} for any given $\eps>0$, that is, we show that  there exists a constant $C_\varepsilon\geq 1$ such that
	\begin{equation}\label{f}
		f(x,t)\leq C_\eps f(x,\tau)\quad \text{for a.a.\,}x\in\Omega\text{ and for all } 0<\tau<t<\infty,
	\end{equation}
	where
	\begin{align*}
		f(x,t):=t^{-(n^++\eps)}\Psi(x,t).
	\end{align*}
	We distinguish between the following two cases.

	(i) \underline{Case $\delta>1$}. For any $\sigma>0$, define
	\begin{align*}
		g_\sigma(t):=\frac{\log(\delta+\eta t)}{t^\sigma}\quad\text{for } t\in (0,\infty).
	\end{align*}
	We first show that there exists a constant $C_\sigma\geq 1$ such that
	\begin{equation}\label{g}
		g_\sigma(t)\leq C_\sigma g_\sigma(\tau)\quad \text{for all } 0<\tau<t<\infty.
	\end{equation}
	Indeed,
	\begin{equation*}
		g_\sigma'(t)=\frac{t^{-\sigma-1}}{\delta+\eta t}\left[\eta t(1-\sigma\log(\delta+\eta t))-\sigma\delta\log(\delta+\eta t)\right].
	\end{equation*}
	Thus, $g_\sigma'(t)<0$ for $0<t\ll 1$ or $1\ll t$. From this and the continuity of $g_\sigma'$ on $(0,\infty)$, we find $t_{1,\sigma}$ and $t_{2,\sigma}$ with $0<t_{1,\sigma}<t_{2,\sigma}<\infty$ such that $g_\sigma'(t)<0$ for all $t\in [0,t_{1,\sigma}]\cup [t_{2,\sigma},\infty).$ Then, by setting $C_\sigma=\frac{\max_{\xi\in [t_{1,\sigma}, t_{2,\sigma}]}g_\sigma(\xi)}{\min_{\xi\in [t_{1,\sigma}, t_{2,\sigma}]}g_\sigma(\xi)}$, we easily obtain \eqref{g}.

	Finally, to see \eqref{f}, by taking $\sigma$ with $0<\sigma<\left(\frac{\eps q}{s}\right)^-$ and rewriting
	\begin{equation*}
		f(x,t)=a(x)^\frac{r(x)}{p(x)} t^{r(x)-n^+-\eps}  + \left(b(x)g_\sigma(t)\right)^{\frac{s(x)}{q(x)}} t^{s(x)-n^++\sigma\frac{s(x)}{q(x)}-\eps},
	\end{equation*}
	we easily derive \eqref{f} from \eqref{g}.

	(ii) \underline{Case $\delta=1$, $a\in L^\infty(\Omega)$, and $b(x)\leq Ca(x)$ for a.a.\,$x\in\Omega$ with some constant $C$}. In this case, it is clear that
	\begin{equation*}
		a(x)^{\frac{r(x)}{p(x)}}\geq a_0>0\quad \text{for a.a.\,} x\in\Omega,
	\end{equation*}
	with some positive constant $a_0$. We have
	\begin{align*}
		&t\partial_t\Psi(x,t)- (n^++\varepsilon)\Psi(x,t)\\
		&\leq -\eps a_0t^{r(x)}+\left(\frac{s(x)}{q(x)}\frac{\eta t}{1+\eta t}\frac{1}{\log(1+\eta t)}-\eps\right)\left(b(x)\log(1+\eta t)\right)^\frac{s(x)}{q(x)} t^{s(x)}.
	\end{align*}
	Thus, we find $t_{1,\varepsilon}, t_{2,\varepsilon}$ with $0<t_{1,\varepsilon}< 1<t_{2,\varepsilon}$ such that
	\begin{align*}
		t\partial_t\Psi(x,t)- (n^++\varepsilon)\Psi(x,t)\leq 0\quad \text{for a.a.\,}x\in\Omega\text{ and for all } t\in [0,t_{1,\varepsilon}]\cup [t_{2,\varepsilon},\infty).
	\end{align*}
	It follows that
	\begin{align*}
		\partial_t f(x,t)=t^{-(n^++\eps+1)}\left[t\partial_t\Psi(x,t)- (n^++\varepsilon)\Psi(x,t)\right]\leq 0
	\end{align*}
	for a.a.\,$x\in\Omega$ and for all $t\in [0,t_{1,\varepsilon}]\cup [t_{2,\varepsilon},\infty).$ Then, by using the estimate
	\begin{align*}
		f(x,t)\leq \frac{\max_{\xi\in [t_{1,\eps}, t_{2,\eps}]}f(x,\xi)}{\min_{\xi\in [t_{1,\eps}, t_{2,\eps}]}f(x,\xi)}f(x,\tau)\quad\text{for a.a.\,}x\in\Omega \text{ and for all } 0<\tau<t,
	\end{align*}
	we easily derive \eqref{f} with
	\begin{align*}
		C_\eps=\frac{t_{1,\eps}^{-(n^++\eps-m^-)}\left(1+\log(1+\eta t_{2,\eps})\right)^{\left(\frac{s}{q}\right)^+}\left\|a^{\frac{r}{p}}+b^{\frac{s}{q}}\right\|_{L^{\infty}(\Omega)}}{a_0t_{2,\eps}^{m^--n^+-\eps}}.
	\end{align*}
	This completes the proof.
\end{proof}

The next proposition follows immediately from Propositions \ref{P.nor-mod-A} and \ref{P.aCon}.

\begin{proposition}\label{P.nor-mod-S}
	Assume that \textnormal{(\hyperlink{H0}{H$_0$})} holds. If $\delta=1$, additionally assume that $a\in L^\infty(\Omega)$ and $b(x)\leq Ca(x)$ for a.a.\,$x\in\Omega$ with some constant $C>0$. Let $(X,\varrho_X,\|\cdot\|_X)$ be either $\left(L^{\mathcal{S}}(\Omega),\varrho_{\mathcal{S}},\|\cdot\|_{\mathcal{S}}\right)$ or $\left(W^{1, \mathcal{S}}(\Omega),\varrho_{1,\mathcal{S}},\|\cdot\|_{1, \mathcal{S}}\right)$.  Then,  for any given $\varepsilon>0$, there exists a constant $C_\varepsilon\geq 1$ such that
	\begin{align*}
		C_\varepsilon^{-1}\min\left\{\|u\|_{X}^{p^-}, \|u\|_{X}^{q^++\varepsilon}\right\} \leq \varrho_{X}(u) \leq C_\varepsilon\max\left\{\|u\|_{X}^{p^-}, \|u\|_{X}^{q^++\varepsilon}\right\}\quad \text{for all }u\in X.
	\end{align*}
\end{proposition}

Following the arguments in the proofs of Arora--Crespo-Blanco--Winkert \cite[Proposition 3.8]{Arora-Crespo-Blanco-Winkert-2025-b} and Ho--Winkert \cite[Proposition 3.7]{Ho-Winkert-2023}, we obtain the following embedding result.

\begin{proposition}\label{emb-main}
	Let \textnormal{(\hyperlink{H0}{H$_0$})}  and \textnormal{(\hyperlink{H1}{H$_1$})} be satisfied. Let $\Psi$ be defined as in \eqref{Psi} with $r,s\in C(\overline{\Omega})$ satisfying $1<r(x)\leq p^*(x)$ and $1<s(x)\leq q^*(x)$ for all $x\in\overline{\Omega}$. Then we have the continuous embedding
	\begin{align}\label{Prop-S-E}
		W^{1,\mathcal{S}}(\Omega)
		\hookrightarrow  L^{\Psi}(\Omega).
	\end{align}
	Furthermore, if $r(x)<p^*(x)$ and $s(x)< q^*(x)$ for all $x\in\overline{\Omega}$, then the embedding in \eqref{Prop-S-E} is compact.
\end{proposition}

Regarding Proposition \ref{emb-main}, we also mention the work by Cianchi--Diening \cite{Cianchi-Diening-2024} about compact embeddings in Musielak-Orlicz spaces and the papers by Cianchi \cite{Cianchi-1996,Cianchi-2004} about optimal embeddings in Orlicz spaces, see also Donaldson--Trudinger \cite{Donaldson-Trudinger-1971}.

For a measurable subset $E$ of $\mathbb{R}^N$, we denote by $|E|$ the $N$-dimensional Lebesgue measure of $E$. For a measurable function $f \colon E \to \mathbb{R}$ we write
\begin{align*}
	\sup_E f := \operatorname{esssup}_E f, \quad \inf_E f := \operatorname{essinf}_E f, \quad\text{and}\quad \operatorname{osc}_E f := \sup_E f - \inf_E f.
\end{align*}

For a ball $B_{\rho}(x_0) \subset \mathbb{R}^N$ centered at $x_0 \in \mathbb{R}^N$ with radius $\rho$, we set
\begin{align*}
	\Omega_{\rho}(x_0) := \Omega \cap B_{\rho}(x_0), \quad (\partial \Omega)_{\rho}(x_0) := \partial \Omega \cap B_{\rho}(x_0), \quad \omega_N := |B_1(x_0)|.
\end{align*}
When $x_0$ is clear from the context, we omit it from the above notation.

Let $\Gamma$ be either a portion of $\partial \Omega$ or the whole boundary $\partial \Omega$. We denote by $C_0^1(\overline{\Omega} \setminus \Gamma)$ the set of functions in $C^1(\overline{\Omega})$ with compact support in $\overline{\Omega} \setminus \Gamma$, and by $W_0^{1,1}(\overline{\Omega} \setminus \Gamma)$ the closure of $C_0^1(\overline{\Omega} \setminus \Gamma)$ in $W^{1,1}(\Omega)$. In particular, if $\Gamma = \partial \Omega$, then we define $W_0^{1,1}(\Omega) := W_0^{1,1}(\overline{\Omega} \setminus \Gamma)$.

Let $u\in W^{1,1}(\Omega)$. We say that $u\leq 0$ on $\Gamma$ if $\max\{u,0\}\in W_0^{1,1}(\overline{\Omega} \setminus \Gamma)$. We say that $u\geq 0$ on $\Gamma$ if $-u\leq 0$ on $\Gamma$. For $u,v \in W^{1,1}(\Omega)$, we say that $u\leq v$ on $\Gamma$ if $u-v\leq 0$ on $\Gamma$. We define
\begin{equation*}
	\begin{aligned}
		\sup_\Gamma u &:= \inf\left\{k\in\mathbb{R}\colon  u\leq k \text{ on } \Gamma\right\},&\inf_\Gamma u &:=-\sup_\Gamma(-u),\\
		\sup_\Gamma |u| &:= \max \{- \inf_\Gamma u, \sup_\Gamma u\}, & \operatorname{osc}_{\Gamma} u &:= \sup_\Gamma u - \inf_\Gamma u.
	\end{aligned}
\end{equation*}
Note that if $\Omega$ is a bounded Lipschitz domain in $\mathbb{R}^N$ and $\Gamma=\partial \Omega$, then the above definitions for $u \in W^{1,1}(\Omega)$ coincide with the corresponding definitions for the trace of $u$ on $\partial \Omega$.

In the remainder of this section, we collect several useful results that will be key ingredients in our arguments for establishing the boundedness and H\"{o}lder continuity of solutions via De Giorgi iteration. The next lemma concerns the geometric convergence of sequences of nonnegative numbers. A proof can be found in Ho--Sim \cite[Lemma 4.3]{Ho-Sim-2015}.

\begin{lemma}\label{leRecur}
	Let $\{Z_n\}, n=0,1,2,\ldots,$ be a sequence of nonnegative numbers, satisfying the recursion inequality
	\begin{align*}
		Z_{n+1} \leq K b^n \left (Z_n^{1+\mu_1}+ Z_n^{1+\mu_2} \right ) , \quad n=0,1,2, \ldots,
	\end{align*}
	for some $b>1$, $K>0$ and $\mu_2\geq \mu_1>0$. If
		\begin{align*}
		Z_0 \leq \min  \left((2K)^{-\frac{1}{\mu_1}} b^{-\frac{1}{\mu_1^2}}, (2K)^{-\frac{1}{\mu_2}}b^{-\frac{1}{\mu_1 \mu_2}-\frac{\mu_2-\mu_1}{\mu_2^2}}\right),
	\end{align*}
	then $Z_n \to 0$ as $n \to \infty$.
\end{lemma}

The next three lemmas are taken from the book by Lady{\v{z}}enskaja--Ural{\cprime}ceva \cite{Ladyzenskaja-Uralceva-1968}, see Lemmas 3.3, 3.5, and 4.8 in Chapter 2.

\begin{lemma}\label{L2.12}
	Let $u \in W^{1,m}(\Omega)$, $m\geq 1$, and let $B_{\rho}(x_0)$ be a ball centered at $x_0\in\partial\Omega$ with radius $\rho$. Then, for $k \geq \sup_{(\partial \Omega)_{\rho}(x_0)} u$, the function
	\begin{align*}
		\hat{u}^{(k)}(x) =
		\begin{cases}
			\max\left\{u(x),k\right\}, & x \in \Omega_\rho(x_0),\\
			k,& x \in B_{\rho}(x_0)\setminus\Omega_\rho(x_0),
		\end{cases}
	\end{align*}
	belongs to $W^{1,m}(B_{\rho}(x_0))$.
\end{lemma}

\begin{lemma}\label{L2.6}
	Let $B_{\rho}$ be a ball of radius $\rho$ in $\mathbb{R}^N$. Then, for every $u \in W^{1,1}(B_{\rho})$ and all numbers $k$ and $l$ satisfying $k < l$, the following inequality holds:
	\begin{align*}
		(l - k) |A_{l,\rho}|^{1 - \frac{1}{N}} \leq\frac{\beta\rho^N}{|B_{\rho} \setminus A_{k,\rho}|}  \int_{A_{k,\rho} \setminus A_{l,\rho}} |\nabla u| \, \mathrm{d} x,
	\end{align*}
	where $A_{t,\rho} := \{ x \in B_{\rho} \colon u(x) > t \}$ for $t\in\{l,k\}$ and $\beta > 1$ is a constant depending only on $N$.
\end{lemma}

\begin{lemma}\label{L2.7}
	Let $u$ be a measurable and bounded in a ball $B_{\rho_0}$ or in a portion of it, $\Omega_{\rho_0}:=B_{\rho_0}\cap\Omega$.
	Consider the balls $B_{\rho}$ and $B_{b\rho}$, where $b>1$ is a fixed constant, concentric with $B_{\rho_0}$. Suppose that for every $\rho \leq b^{-1} \rho_0$, at least one of the following inequalities holds:
	\begin{align*}
		\operatorname{osc}_{\Omega_{\rho}} u \leq c_0 \rho^\varepsilon, \quad \operatorname{osc}_{\Omega_{\rho}} u \leq \theta \operatorname{osc}_{\Omega_{b\rho}} u
	\end{align*}
	for some positive constants $c_0$, $\varepsilon \leq 1$, and $\theta < 1$. Then, for every $\rho \leq \rho_0$,
	\begin{align*}
		\operatorname{osc}_{\Omega_{\rho}} u \leq c \rho_0^{-\alpha} \rho^\alpha,
	\end{align*}
	where
	\begin{align*}
		\alpha = \min\{\varepsilon, -\log_b \theta\}\quad\text{and}\quad c = b^\alpha \max\{c_0\rho_0^\varepsilon, \operatorname{osc}_{\Omega_{\rho_0}} u\}.
	\end{align*}
\end{lemma}

Furthermore, in the following sections, we frequently make use of the following elementary inequalities:
\begin{enumerate}
	\item[\textnormal{(i)}]
		For all $a,b\geq0$, $\eps>0$, and $\eta>1$,
		\begin{align}\label{young}
			ab\leq \frac{1}{\eta}\eps a^\eta+\frac{\eta-1}{\eta}\eps^{-\frac{1}{\eta-1}}b^{\frac{\eta}{\eta-1}}\leq \eps a^\eta+\eps^{-\frac{1}{\eta-1}}b^{\frac{\eta}{\eta-1}}.
		\end{align}
	\item[\textnormal{(ii)}]
		For all $t\geq 0$ and $0<\alpha\leq\beta\leq\gamma$,
		\begin{align}\label{ei}
			t^\beta&\leq t^\alpha+t^\gamma\quad \text{and}\quad t^\beta\leq t^\gamma+1.
		\end{align}
	\item[\textnormal{(iii)}]
		For all $t,\eta\geq 0$ and $\delta\geq 1$,
		\begin{align}\label{ei-1}
			\log(\delta+\eta t)&\geq \min\{1,t\}\log(\delta+\eta).
		\end{align}
	\item[\textnormal{(iv)}]
		For all $t\geq 0$ and $\delta,\gamma\geq 1$,
		\begin{align}\label{ei-2}
			\log(\delta+\gamma t)&\leq \gamma\log(\delta+t).
		\end{align}
	%\item[\textnormal{(v)}]
	%	For all $t\geq 0$, $\sigma> 0$, $\varepsilon>0$, and $\gamma>0$,
	%	\begin{align*}%\label{i-log2}
	%		\log(e+\gamma t)\leq \varepsilon t^\sigma+C(\varepsilon,\sigma).
	%	\end{align*}
\end{enumerate}

Let us now fix our notation. We write $\N_0:=\N \cup \{0\}$ and for a real number $t>1$, we denote by $t':=\frac{t}{t-1}$ the conjugate number of $t$. For a measurable function $v\colon\Omega\to\R$, we set
\begin{align*}
	v_+ := \max \{v,0\}
	\quad\text{and}\quad
	v_- := \max \{-v,0\}.
\end{align*}
By $\|\cdot\|_{X}$, we denote the norm of a normed space $X$.

%********************************************************************
\section{A priori bounds in the subcritical case}\label{Section-3}
%********************************************************************

In this section, we establish a priori bounds for weak solutions of problem \eqref{D} in the subcritical case under assumptions \textnormal{(\hyperlink{H0}{H$_0$})}, \textnormal{(\hyperlink{H1}{H$_1$})}, and \textnormal{(\hyperlink{H2}{H$_2$})}. The proof of Theorem \ref{D.a-priori} is based on ideas used in Ho--Kim \cite{Ho-Kim-2019}, Ho--Winkert \cite{Ho-Winkert-2023}, and Winkert--Zacher \cite{Winkert-Zacher-2012,Winkert-Zacher-2015}.

\begin{proof}[Proof of Theorem \ref{D.a-priori}]
	Let $u$ be a weak solution of problem \eqref{D}.

	{\bf Step 1: Defining the recursion sequence and deriving basic estimates.}

	For each $n\in\N_0$, we define
	\begin{align}\label{def.kn}
		\kappa_n:=\kappa_*\left(2-\frac{1}{2^n}\right)
	\end{align}
	with $\kappa_*>0$ to be specified later, and
	\begin{align}\label{def.vn}
		v_n:=\left(u-\kappa_n\right)_+.
	\end{align}
By using the standard mollifiers we can show that $v_n\in \Wpzero{\mathcal{S}}$, see Proposition 3.8 by Arora--Crespo-Blanco--Winkert \cite{Arora-Crespo-Blanco-Winkert-2025}. Furthermore, we have
	\begin{align*}%\label{S-k_n}
		\kappa_n \nearrow 2\kappa_*
		\quad \text{and}\quad
		\kappa_* \leq \kappa_n <2\kappa_* \quad \text{for all }n\in \mathbb{N}_0.
	\end{align*}
	We define the recursion sequence $\{Z_n\}_{n\in\mathbb{N}_0}$ by
	\begin{align*}%\label{S-Zn.def}
		Z_n:=\int_{A_{\kappa_n} }\Psi(x,v_n)\,\mathrm{d} x\quad \text{for } n\in\N_0,
	\end{align*}
	where
	\begin{align}\label{def.Ak}
		A_\kappa:=\{x\in\Omega\colon u(x)>\kappa\}\quad \text{for } \kappa\in\R.
	\end{align}
	It is clear that
	\begin{align}\label{S-Zn.decreasing}
		A_{\kappa_{n+1}}\subset A_{\kappa_n}
		\quad \text{and}\quad
		Z_{n+1}\leq Z_n \quad\text{for all }n\in \mathbb{N}_0.
	\end{align}
	Using
	\begin{align*}
		u(x)- \kappa_n \geq u(x)\l(1-\frac{\kappa_n}{\kappa_{n+1}}\r) = \frac{u(x)}{2^{n+2}-1} \quad \text{for a.a.\,} x \in A_{\kappa_{n+1}},
	\end{align*}
	we obtain
	\begin{align} \label{S-est.u}
		u(x)\leq (2^{n+2}-1)v_n(x)
		\quad \text{for a.a.\,}x\in A_{\kappa_{n+1}} \text{ and for all } n\in\N_0.
	\end{align}
	Moreover, setting $\lambda_n=\kappa_{n+1}-\kappa_n$ and using \eqref{S-Zn.decreasing}, the estimate $v_{n}\geq \lambda_n$ on $A_{\kappa_{n+1}}$ and \eqref{ei-1}, we have
	\begin{align*}
		Z_n & \geq \int_{A_{\kappa_{n+1}}} \Bigg[\lambda_n^{r(x)} a(x)^\frac{r(x)}{p(x)}\left(\frac{v_{n}}{\lambda_n}\right)^{r(x)}\\
		&\qquad \qquad\qquad+\lambda_n^{s(x)}\left(b(x)\log\left(\delta+\lambda_n\eta \left(\frac{v_{n}}{\lambda_n}\right)\right)\right)^{\frac{s(x)}{q(x)}}\left(\frac{v_{n}}{\lambda_n}\right)^{s(x)}\Bigg]\,\mathrm{d} x \\
		& \geq \int_{A_{\kappa_{n+1}}} \left[\lambda_n^{r(x)} a(x)^\frac{r(x)}{p(x)}+\lambda_n^{s(x)}\left(b(x)\log\left(\delta+\lambda_n\eta\right)\right)^{\frac{s(x)}{q(x)}}\right]\,\mathrm{d} x \\
		& \geq \int_{A_{\kappa_{n+1}}} \left[\lambda_n^{r(x)} a(x)^\frac{r(x)}{p(x)}+\min\left\{\lambda_n^{s(x)},\lambda_n^{s(x)+\frac{s(x)}{q(x)}}\right\}\left(b(x)\log\left(\delta+\eta\right)\right)^{\frac{s(x)}{q(x)}}\right]\,\mathrm{d} x \\
		& \geq \int_{A_{\kappa_{n+1}}} \min\left\{\lambda_n^{r(x)},\lambda_n^{s(x)},\lambda_n^{s(x)+\frac{s(x)}{q(x)}}\right\}\left[a(x)^\frac{r(x)}{p(x)}+\left(b(x)\log\left(\delta+\eta\right)\right)^{\frac{s(x)}{q(x)}}\right]\,\mathrm{d} x\\
		& \geq \min\left\{\lambda_n^{a_1},\lambda_n^{a_2}\right\}\int_{A_{\kappa_{n+1}}} \left[a(x)^\frac{r(x)}{p(x)}+\left(b(x)\log\left(\delta+\eta\right)\right)^{\frac{s(x)}{q(x)}}\right]\,\mathrm{d} x,
	\end{align*}
		where
	\begin{align*}
		a_1:=\min\left\{r^-, s^-\right\}\leq a_2:=\max\left\{r^+, \left(\frac{s(q+1)}{q}\right)^+\right\}.
	\end{align*}
	From this and the estimate
	\begin{equation}\label{d0}
		\begin{aligned}
			&a(x)^\frac{r(x)}{p(x)}+\left(b(x)\log\left(\delta+\eta\right)\right)^{\frac{s(x)}{q(x)}}\\
			&\geq \min\left\{1,\left(\frac{d}{2}\right)^{\left(\frac{r}{p}\right)^+},\left(\frac{d}{2}\log\left(\delta+\eta\right)\right)^{\left(\frac{s}{q}\right)^+}\right\}=:d_0>0\quad \text{for a.a.\,}x\in\Omega,
		\end{aligned}
	\end{equation}
	we arrive at
	\begin{align*}
		Z_n\geq d_0\min\left\{\lambda_n^{a_1},\lambda_n^{a_2}\right\}|A_{\kappa_{n+1}}|.
	\end{align*}
	Thus,
	\begin{align} \label{S-|A_{k_{n+1}}|}
		|A_{\kappa_{n+1}}| \leq \left(\kappa_{*}^{-a_1}+\kappa_{*}^{-a_2}\right)\frac{2^{(n+1)a_2}}{d_0} Z_n\leq 2\left(1+\kappa_{*}^{-a_2}\right)\frac{2^{(n+1)a_2}}{d_0} Z_n\quad \text{for all } n\in\N_0.
	\end{align}
	Applying \eqref{ei}, we deduce from the assumptions on the exponents that
	\begin{align*}
		&\int_{A_{\kappa_{n+1}}} \left[a(x)v_{n+1}^{p(x)}+b(x)\log\left(\delta+\eta v_{n+1}\right)v_{n+1}^{q(x)}\right]\,\mathrm{d} x\\
		& \leq \int_{A_{\kappa_{n+1}}}\left[a(x)^\frac{r(x)}{p(x)}v_{n+1}^{r(x)} +\left(b(x)\log(\delta+\eta v_{n+1})\right)^{\frac{s(x)}{q(x)}}v_{n+1}^{s(x)}+2\right]\,\mathrm{d} x.
	\end{align*}
	Combining this with \eqref{S-Zn.decreasing} and \eqref{S-|A_{k_{n+1}}|} yields
	\begin{align}\label{S-D.u}
		\int_{A_{\kappa_{n+1}}}\mathcal{S}(x,v_{n+1})\,\mathrm{d} x\leq 5\left(1+\kappa_{*}^{-a_2}\right)\frac{2^{(n+1)a_2}}{d_0} Z_n\quad \text{for all }n\in\N_0.
	\end{align}

	Next, we estimate the truncated energies in terms of $Z_n$. To this end, we test \eqref{def_sol_D} with $\varphi =v_{n+1}$. This gives
	\begin{align}\label{S-D.var.Eq}
		\int_{A_{\kappa_{n+1}}}
		\mathcal{A}(x,u,\nabla u) \cdot \nabla u \,\mathrm{d} x =\int_{A_{\kappa_{n+1}} }\mathcal{B}(x,u,\nabla u)v_{n+1}\,\mathrm{d} x.
	\end{align}
	Note that on $A_{\kappa_{n+1}}$ we have $u\geq v_{n+1} >0$, and by \eqref{d0},
	\begin{align*}
		u\leq \frac{1}{d_0}\left[a(x)^{\frac{r(x)}{p(x)}}u^{r(x)}+\left(b(x)\log\big(\delta+\eta u\big)\right)^{\frac{s(x)}{q(x)}}u^{s(x)}\right]+1.
	\end{align*}
	Using these estimates together with inequality \eqref{young}, we derive from \textnormal{(\hyperlink{H2}{H$_2$})}\textnormal{(\hyperlink{H2ii}{ii})} and \textnormal{(\hyperlink{H2}{H$_2$})}\textnormal{(\hyperlink{H2iii}{iii})} that
	\begin{align*}
		&\int_{A_{\kappa_{n+1}}} \mathcal{A}(x,u,\nabla u)\cdot \nabla u\,\mathrm{d} x\\ &\geq \alpha_2\int_{A_{\kappa_{n+1}} }\left[a(x)|\nabla u|^{p(x)}+b(x)|\nabla u|^{q(x)}\log\left(\delta+\eta|\nabla u|\right)\right]\,\mathrm{d} x\\
		& \quad-a_3\int_{A_{\kappa_{n+1}} }\bigg[a(x)^{\frac{r(x)}{p(x)}}u^{r(x)}+\left(b(x)\log\big(\delta+\eta u\big)\right)^{\frac{s(x)}{q(x)}}u^{s(x)}+1\bigg]
	\end{align*}
	and
	\begin{align*}
		& \int_{A_{\kappa_{n+1}}} \mathcal{B}(x,u,\nabla 	u)v_{n+1}\,\mathrm{d} x \\
		& \leq \beta\int_{A_{\kappa_{n+1}} }\bigg[a(x)^{\frac{r(x)}{p(x)}}u^{r(x)-1}+\left(b(x)\log\big(\delta+\eta u\big)\right)^{\frac{s(x)}{q(x)}}u^{s(x)-1}\\
		&\qquad\qquad\qquad+a(x)^{\frac{1}{p(x)}+\frac{1}{r'(x)}}|\nabla u|^{\frac{p(x)}{r'(x)}} \\
		&\qquad\qquad\qquad+\left(b(x)\log\left(\delta+\eta u\right)\right)^{\frac{1}{q(x)}}\left(b(x)\log\big(\delta+\eta|\nabla u|\big)\right)^{\frac{1}{s'(x)}}|\nabla u|^{\frac{q(x)}{s'(x)}}+1\bigg]u\,\mathrm{d} x    \\
		& \leq \frac{\alpha_2}{2}\int_{A_{\kappa_{n+1}} }\left[a(x)|\nabla u|^{p(x)}+b(x)|\nabla u|^{q(x)}\log\big(\delta+\eta|\nabla u|\big)\right]\,\mathrm{d} x   \\
		& \qquad\qquad\qquad\quad+ C_1\int_{A_{\kappa_{n+1}} }\left[a(x)^{\frac{r(x)}{p(x)}}u^{r(x)} +\left(b(x)\log\big(\delta+\eta u\big)\right)^{\frac{s(x)}{q(x)}}u^{s(x)}+1\right]\,\mathrm{d} x.
	\end{align*}
	Here, and throughout the rest of the proof, $C_i$ ($i\in\N$) denotes a positive constant that is independent of $u$, $n$, and $\kappa_{*}$. Plugging the last two estimates into \eqref{S-D.var.Eq}, and then using \eqref{S-est.u} and \eqref{ei-2},  we obtain
	\begin{align*}
		& \int_{A_{\kappa_{n+1}} }\left[a(x)|\nabla u|^{p(x)}+b(x)|\nabla u|^{q(x)}\log\left(\delta+\eta|\nabla u|\right)\right]\,\mathrm{d} x \\
		& \leq C_2\int_{A_{\kappa_{n+1}} }\left[a(x)^{\frac{r(x)}{p(x)}}u^{r(x)}+\left(b(x)\log\big(\delta+\eta u\big)\right)^{\frac{s(x)}{q(x)}}u^{s(x)}+1\right]\,\mathrm{d} x  \\
		& \leq C_2\int_{A_{\kappa_{n+1}}}\Bigg[a(x)^{\frac{r(x)}{p(x)}}\left((2^{n+2}-1)v_n\right)^{r(x)}\\
		&\qquad\qquad\qquad\quad+\left(b(x)\log\big(\delta+\eta (2^{n+2}-1)v_n\big)\right)^{\frac{s(x)}{q(x)}} \left((2^{n+2}-1)v_n\right)^{s(x)}\Bigg]\,\mathrm{d} x +C_2|A_{\kappa_{n+1}}| \\
		& \leq C_32^{2(r^++s^+)n}\int_{A_{\kappa_{n}}}\left[a(x)^\frac{r(x)}{p(x)}v_n^{r(x)}+\left(b(x)\log(\delta+\eta v_n)\right)^{\frac{s(x)}{q(x)}}v_n^{s(x)}\right]\,\mathrm{d} x+C_2|A_{\kappa_{n+1}}|.
	\end{align*}
	This and \eqref{S-|A_{k_{n+1}}|} yield \begin{align}\label{S-D.grad}
		\int_{A_{\kappa_{n+1}}}\mathcal{S}(x,|\nabla u|)\,\mathrm{d} x\leq  C_4(1+\kappa_{*}^{-a_2})2^{2(r^++s^+)n}Z_n
		\quad\text{for all } n\in\N_0.
	\end{align}

	\textbf{Step 2: Estimating $Z_{n+1}$ in terms of $Z_n$.}

	Let $\{B_i\}_{i=1}^m$ be a finite open covering of $\overline{\Omega}$, where $B_i$ ($i\in \{1,\cdots,m\}$) are open balls of radius $R$ in $\R^N$ such that $\Omega_i:=B_i\cap\Omega$ ($i\in \{1,\cdots,m\}$) are Lipschitz domains, see, for example, the book of Carbone--De Arcangelis \cite[Proposition 2.5.4]{Carbone-DeArcangelis-2002}. We may choose $R$ sufficiently small such that
	\begin{align}\label{S-D.Omega_i}
		|\Omega_{i}|<1\quad \text{for all } i\in\{1,\cdots,m\},
	\end{align}
	and
	\begin{align}\label{S-D.loc.exp}
		q_i^+<r^{-}_{i}\leq r_i^+<\left(p^*\right)^-_i\quad \text{and} \quad q_i^+<s_i^-\leq s_i^+<\left(q^*\right)^-_i\quad \text{for all } i\in\{1,\cdots,m\},
	\end{align}
	where, for a function $f\in C\left(\overline{\Omega}\right)$ and $i\in\{1,\cdots,m\}$, we denote
	\begin{align*}
		f_i^+:=\max_{x\in\overline{\Omega}_i} f(x)
		\quad \text{and} \quad f^{-}_{i}:=\min_{x\in\overline{\Omega}_i} f(x).
	\end{align*}
	Let $n\in\N_0$. For each $i\in\{1,\cdots,m\}$, $\hat{\alpha}>0$, and $\hat{\beta}>0$, we define
	\begin{align*}
		T_{n,i}(\hat{\alpha},\hat{\beta}):=\int_{\Omega_{i}}\left[a(x)^{\frac{\hat{\alpha}}{p(x)}}v_{n+1}^{\hat{\alpha}}+\left(b(x)\log(\delta+\eta v_{n+1})\right)^{\frac{\hat{\beta}}{q(x)}}v_{n+1}^{\hat{\beta}}\right]\,\mathrm{d} x
	\end{align*}
	and
	\begin{align*}
		m_i=\min\{r_i^-,s_i^-\},\ \ n_i=\max\{r_i^+,s_i^+\}.
	\end{align*}
	We have
	\begin{align*}
		Z_{n+1} & =\int_{\Omega}\left[a(x)^\frac{r(x)}{p(x)}v_{n+1}^{r(x)}+\left(b(x)\log(\delta+\eta v_{n+1})\right)^{\frac{s(x)}{q(x)}}v_{n+1}^{s(x)}\right]\,\mathrm{d} x \\
		& \leq \sum_{i=1}^m\int_{\Omega_{i}}\left[a(x)^\frac{r(x)}{p(x)}v_{n+1}^{r(x)}+\left(b(x)\log(\delta+\eta v_{n+1})\right)^{\frac{s(x)}{q(x)}}v_{n+1}^{s(x)}\right]\,\mathrm{d} x.
	\end{align*}
	From this and \eqref{ei} it follows that
	\begin{align}\label{S-decompose1}
		Z_{n+1}\leq \sum_{i=1}^m\big[T_{n,i}(r_i^-,s_i^-)+T_{n,i}(r_i^+,s_i^+)\big].
	\end{align}
	From \eqref{S-D.loc.exp}, we can fix $\eps$ so small  that
	\begin{align}\label{S-D-eps}
		0<\eps<\min_{1\leq i\leq m} \min\{\left(p^*\right)^-_i-r_i^+, \left(q^*\right)^-_i-s_i^+,m_i-q_i^+\}.
	\end{align}

	Let $i\in\{1,\cdots,m\}$ and let $\star\in\{+,-\}$. By H\"{o}lder's inequality, we have
	\begin{align*}
		\int_{A_{\kappa_{n+1}}\cap \Omega_i}a(x)^{\frac{r_i^\star}{p(x)}}v_{n+1}^{r_i^\star}\,\mathrm{d} x
		\leq \left(\int_{\Omega_{i}}a(x)^{\frac{r_i^\star+\eps}{p(x)}}v_{n+1}^{r_i^\star+\eps}\,\mathrm{d} x\right)^{\frac{r_i^\star}{r_i^\star+\eps}}|A_{\kappa_{n+1}}\cap \Omega_i|^{\frac{\eps}{r_i^\star+\eps}}
	\end{align*}
	and
	\begin{align*}
		&\int_{A_{\kappa_{n+1}}\cap \Omega_i} \left(b(x)\log(\delta+\eta v_{n+1})\right)^{\frac{s_i^\star}{q(x)}}v_{n+1}^{s_i^\star}\,\mathrm{d} x\\
		& \leq \left(\int_{\Omega_{i}}\left(b(x)\log(\delta+\eta v_{n+1})\right)^{\frac{s_i^\star+\eps}{q(x)}}v_{n+1}^{s_i^\star+\eps}\,\mathrm{d} x\right)^{\frac{s_i^\star}{s_i^\star+\eps}}|A_{\kappa_{n+1}}\cap \Omega_i|^{\frac{\eps}{s_i^\star+\eps}}.
	\end{align*}
	From the last two estimates and \eqref{S-D.Omega_i}, it follows that
	\begin{equation}\label{S-de-E1}
		\begin{aligned}
			& T_{n,i}(r_i^\star,s_i^\star)\\
			&=\int_{A_{\kappa_{n+1}}\cap \Omega_i}\left[a(x)^{\frac{r_i^\star}{p(x)}}v_{n+1}^{r_i^\star}+\left(b(x)\log(\delta+\eta v_{n+1})\right)^{\frac{s_i^\star}{q(x)}}v_{n+1}^{s_i^\star}\right]\,\mathrm{d} x \\
			& \leq 2|A_{\kappa_{n+1}}\cap \Omega_i|^{\frac{\eps}{r^++s^++\eps}}\left[\varrho_{\Psi_{\star},\Omega_i}(v_{n+1})^{\frac{m_i}{m_i+\eps}}+\varrho_{\Psi_{\star},\Omega_i}(v_{n+1})^{\frac{n_i}{n_i+\eps}}\right],
		\end{aligned}
	\end{equation}
	where
	\begin{equation*}
		\Psi_{\star}(x,t):=a(x)^{\frac{r_i^\star+\eps}{p(x)}}t^{r_i^\star+\eps}+\left(b(x)\log(\delta+\eta t)\right)^{\frac{s_i^\star+\eps}{q(x)}}t^{s_i^\star+\eps}.
	\end{equation*}
	Invoking Propositions~\ref{P.aCon} and \ref{emb-main} together with \eqref{ei}, we have
	\begin{equation}\label{S-de-E2}
		\begin{aligned}
		\varrho_{\Psi_{\star},\Omega_i}(v_{n+1})^{\frac{m_i}{m_i+\eps}}+\varrho_{\Psi_{\star},\Omega_i}(v_{n+1})^{\frac{n_i}{n_i+\eps}}&\leq  C_5\left(\|v_{n+1}\|_{L^{\Psi_{\star}}(\Omega_i)}^{m_i}+\|v_{n+1}\|_{L^{\Psi_{\star}}(\Omega_i)}^{\frac{n_i(n_i+2\eps)}{n_i+\eps}}\right)\\
		&\leq C_6\left(\|v_{n+1}\|_{W^{1,\mathcal{S}}  (\Omega_i)}^{m_i}+\|v_{n+1}\|_{W^{1,\mathcal{S}} (\Omega_i)}^{\frac{n_i(n_i+2\eps)}{n_i+\eps}}\right).
	\end{aligned}
	\end{equation}
	We set
		\begin{align*}
			R_{n,i}:=\int_{\Omega_i}\big[\mathcal{S}\left(x,|\nabla v_{n+1}|\right)+\mathcal{S}\left(x, v_{n+1}\right)\big]\,\mathrm{d} x.
		\end{align*}
	Applying Proposition \ref{P.nor-mod-S} to the case $\Omega=\Omega_i$ and using \eqref{ei} again, we obtain
	\begin{equation}\label{S-de-E2'}
					\|v_{n+1}\|_{W^{1,\mathcal{S}} (\Omega_i)}^{m_i}+\|v_{n+1}\|_{W^{1,\mathcal{S}} (\Omega_i)}^{\frac{n_i(n_i+2\eps)}{n_i+\eps}}\leq C_{7}\left(R_{n,i}^{\frac{m_i}{q_i^++\eps}}+R_{n,i}^{\frac{n_i(n_i+2\eps)}{p_i^-(n_i+\eps)}}\right).
	\end{equation}
	It follows from \eqref{S-D-eps}, \eqref{S-de-E1}, \eqref{S-de-E2}, and \eqref{S-de-E2'} that
		\begin{align}\label{S-de-E5}
		T_{n,i}(r_i^\star,s_i^\star)\leq C_{8}|A_{\kappa_{n+1}}|^{\frac{\eps}{r^++s^++\eps}}\left(R_{n}^{1+\gamma_1}+R_{n}^{1+\gamma_2}\right),
	\end{align}
	where
	\begin{align*}
		R_{n}:=\int_{\Omega}\big[\mathcal{S}\left(x,|\nabla v_{n+1}|\right)+\mathcal{S}\left(x, v_{n+1}\right)\big]\,\mathrm{d} x
	\end{align*}
	and
	\begin{align*}
		0<\gamma_1:=\min_{1\leq i\leq m} \frac{m_i}{q_i^++\eps}-1<\gamma_2:=\max_{1\leq i\leq m} \frac{n_i(n_i+2\eps)}{p_i^-(n_i+\eps)}-1.
	\end{align*}
	Note that \eqref{S-D.u} and \eqref{S-D.grad} imply
	\begin{align*}
		R_n\leq C_{9}\left(1+\kappa_{*}^{-a_2}\right)2^{2(r^++s^+)n}Z_n\quad \text{for all } n\in\N_0.
	\end{align*}
	Consequently,
	\begin{align}\label{S-de-E7}
		R_{n}^{1+\gamma_1}+R_{n}^{1+\gamma_2}\leq C_{10}\left(1+\kappa_{*}^{-a_2(1+\gamma_2)}\right)2^{2(r^++s^+)(1+\gamma_2)n}\left(Z_{n}^{1+\gamma_1}+Z_{n}^{1+\gamma_2}\right).
	\end{align}
	On the other hand, \eqref{S-|A_{k_{n+1}}|} yields
	\begin{align}\label{S-de-E8}
		|A_{\kappa_{n+1}}|^{\frac{\eps}{r^++s^++\eps}}\leq C_{11}\left(\kappa_{*}^{-\frac{\eps a_1}{r^++s^++\eps}}+\kappa_{*}^{-\frac{\eps a_2}{r^++s^++\eps}}\right)2^{\frac{\eps a_2}{r^++s^++\eps}n} Z_n^{\frac{\eps}{r^++s^++\eps}}.
	\end{align}
	Plugging \eqref{S-de-E7} and \eqref{S-de-E8} into \eqref{S-de-E5}, we obtain
	\begin{align}\label{S-de-E9}
		T_{n,i}(r_i^\star,s_i^\star)\leq
		C_{12}\round{\kappa_{*}^{-b_1}+\kappa_{*}^{-b_2}}b^n\round{Z_{n}^{1+\delta_1}+Z_{n}^{1+\delta_2}},
	\end{align}
	where
	\begin{align*}%\label{S-mu}
		0 & <b_1:=\frac{\eps a_1}{r^++s^++\eps}<b_2:=a_2(1+\gamma_2)+\frac{\eps a_2}{r^++s^++\eps}, \\
		1 & <b:=2^{2(r^++s^+)(1+\gamma_2)+\frac{\eps a_2}{r^++s^++\eps}}
	\end{align*}
	and
	\begin{align*}%\label{S-delta}
		0<\delta_1:=\gamma_1+\frac{\eps}{r^++s^++\eps}\leq \delta_2:=\gamma_2+\frac{\eps}{r^++s^++\eps}.
	\end{align*}
	Finally, we make use of \eqref{S-de-E9} to derive from \eqref{S-decompose1} that
	\begin{align}\label{S-D-ReIneq}
		Z_{n+1}
		\leq
		C_{13}\round{\kappa_{*}^{-b_1}+\kappa_{*}^{-b_2}}b^n\round{Z_{n}^{1+\delta_1}+Z_{n}^{1+\delta_2}}\quad \text{for all } n\in\mathbb{N}_0.
	\end{align}

	\textbf{Step 3: A priori bounds.}

	It is clear that \eqref{D-bound} holds in the case $u=0$. Let us now consider the case $u\ne 0$, which is equivalent to $\int_{\Omega}\Psi(x,|u|)\,\mathrm{d} x>0$. In this case, we employ an argument similar to that used by Ho--Kim \cite[Proof of Theorem 4.2]{Ho-Kim-2019} to derive \eqref{D-bound} from \eqref{S-D-ReIneq}. Indeed,  by Lemma \ref{leRecur}, we have
	\begin{align}\label{apply.le.1}
		Z_n\to 0\quad \text{as } n\to \infty
	\end{align}
	provided that
	\begin{align}\label{apply.le.2}
		Z_0 \leq \min\left\{\left(2C_{13}\round{\kappa_{*}^{-b_1}+\kappa_{*}^{-b_2}}\right)^{-\frac{1}{\delta_1}}b^{-\frac{1}{\delta_1^2}}, \left(2C_{13}\round{\kappa_{*}^{-b_1}+\kappa_{*}^{-b_2}}\right)^{-\frac{1}{\delta_2}}b^{-\frac{1}{\delta_1\delta_2}-\frac{\delta_2-\delta_1}{\delta_2^2}}\right\}.
	\end{align}
	We note that
	\begin{align}\label{apply.le.2'new}
		Z_0= \int_{\Omega}\Psi(x,(u-\kappa_{*})_+)\,\mathrm{d} x\le
		\int_{\Omega}\Psi(x,|u|)\,\mathrm{d} x,
	\end{align}
	and that
	\begin{align}\label{apply.le.3}
		\begin{cases}
			\int_{\Omega}\Psi(x,|u|)\,\mathrm{d} x \leq (2C_{13})^{-\frac{1}{\delta_1}}(\kappa_{*}^{-b_1}+\kappa_{*}^{-b_2})^{-\frac{1}{\delta_1}}b^{-\frac{1}{\delta_1^2}}, \\
			\int_{\Omega}\Psi(x,|u|)\,\mathrm{d} x \leq
			(2C_{13})^{-\frac{1}{\delta_2}}(\kappa_{*}^{-b_1}+\kappa_{*}^{-b_2})^{-\frac{1}{\delta_2}}b^{-\frac{1}{\delta_1\delta_2}-\frac{\delta_2-\delta_1}{\delta_2^2}}
		\end{cases}
	\end{align}
	is equivalent to
	\begin{align}\label{apply.le.3'}
		\begin{cases}
			\kappa_{*}^{-b_1}+\kappa_{*}^{-b_2} \leq (2C_{13})^{-1}b^{-\frac{1}{\delta_1}}\round{\int_{\Omega}\Psi(x,|u|)\,\mathrm{d} x}^{-\delta_1}, \\
			\kappa_{*}^{-b_1}+\kappa_{*}^{-b_2} \leq (2C_{13})^{-1}b^{-\frac{1}{\delta_1}-\frac{\delta_2-\delta_1}{\delta_2}}\round{\int_{\Omega}\Psi(x,|u|)\,\mathrm{d} x}^{-\delta_2}.
		\end{cases}
	\end{align}
	Clearly, \eqref{apply.le.3'} holds provided that
	\begin{align*}
		\begin{cases}
			2\kappa_{*}^{-b_1} \leq (2C_{13})^{-1}b^{-\frac{1}{\delta_1}-\frac{\delta_2-\delta_1}{\delta_2}}\min\left\{\round{\int_{\Omega}\Psi(x,|u|)\,\mathrm{d} x}^{-\delta_1}, \round{\int_{\Omega}\Psi(x,|u|)\,\mathrm{d} x}^{-\delta_2}\right\}, \\
			2\kappa_{*}^{-b_2} \leq (2C_{13})^{-1}b^{-\frac{1}{\delta_1}-\frac{\delta_2-\delta_1}{\delta_2}}\min\left\{\round{\int_{\Omega}\Psi(x,|u|)\,\mathrm{d} x}^{-\delta_1},
			\round{\int_{\Omega}\Psi(x,|u|)\,\mathrm{d} x}^{-\delta_2}\right\},
		\end{cases}
	\end{align*}
	which is equivalent to
	\begin{align}\label{apply.le.4}
		\begin{cases}
			\kappa_{*} \geq (4C_{13})^{\frac{1}{b_1}}b^{\frac{1}{b_1}\round{\frac{1}{\delta_1}+\frac{\delta_2-\delta_1}{\delta_2}}}\max\left\{\round{\int_{\Omega}\Psi(x,|u|)\,\mathrm{d} x}^{\frac{\delta_1}{b_1}}, \round{\int_{\Omega}\Psi(x,|u|)\,\mathrm{d} x}^{\frac{\delta_2}{b_1}}\right\}, \\
			\kappa_{*} \geq (4C_{13})^{\frac{1}{b_2}}b^{\frac{1}{b_2}\round{\frac{1}{\delta_1}+\frac{\delta_2-\delta_1}{\delta_2}}}\max\left\{\round{\int_{\Omega}\Psi(x,|u|)\,\mathrm{d} x}^{\frac{\delta_1}{b_2}},
			\round{\int_{\Omega}\Psi(x,|u|)\,\mathrm{d} x}^{\frac{\delta_2}{b_2}}\right\}.
		\end{cases}
	\end{align}
	Therefore, choosing
	\begin{align*}
		\kappa_{*}&=\max\left\{(4C_{13})^{\frac{1}{b_1}},(4C_{13})^{\frac{1}{b_2}}\right\}b^{\frac{1}{b_1}\round{\frac{1}{\delta_1}+\frac{\delta_2-\delta_1}{\delta_2}}}\\
		&\qquad\times\max\left\{\round{\int_{\Omega}\Psi(x,|u|)\,\mathrm{d} x}^{\frac{\delta_1}{b_2}},
		\round{\int_{\Omega}\Psi(x,|u|)\,\mathrm{d} x}^{\frac{\delta_2}{b_1}}\right\},
	\end{align*}
	we obtain \eqref{apply.le.4}. Hence, \eqref{apply.le.2} follows from \eqref{apply.le.2'new} and \eqref{apply.le.3}. Consequently, \eqref{apply.le.1} holds. Moreover, by Lebesgue's dominated convergence theorem,
	\begin{align*}
		Z_n=\int_{\Omega}\Psi(x,v_n)\,\mathrm{d} x\to \int_{\Omega}\Psi(x,(u-2\kappa_{*})_+)\,\mathrm{d} x
		\quad \text{as }  n\to \infty.
	\end{align*}
	Consequently, we obtain
	\begin{align*}
		\int_{\Omega}\Psi(x,(u-2\kappa_{*})_+)\,\mathrm{d} x=0.
	\end{align*}
	Equivalently,
	\begin{align*}
		\esssup_{x\in\Omega} u(x)\leq 2\kappa_{*}.
	\end{align*}
	Replacing $u$ by $-u$ in the above arguments, we obtain
	\begin{align*}
		\esssup_{x\in\Omega} (-u)(x)\leq 2\kappa_{*}.
	\end{align*}
	Therefore,
	\begin{align}\label{applying.le.5}
		\|u\|_{L^\infty(\Omega)}\le C_{14}\max\left\{\round{\int_{\Omega}\Psi(x,|u|)\,\mathrm{d} x}^{\frac{\delta_1}{b_2}}, \round{\int_{\Omega}\Psi(x,|u|)\,\mathrm{d} x}^{\frac{\delta_2}{b_1}}\right\}.
	\end{align}
Applying Proposition \ref{P.aCon} again, we derive \eqref{D-bound} from \eqref{applying.le.5}. The proof is complete.
\end{proof}

%********************************************************************
\section{Boundedness results in the critical case}\label{Section-4}
%********************************************************************

In this section, we discuss the boundedness of weak solutions to problem \eqref{D} when the nonlinearities exhibit critical growth in the sense of Proposition \ref{emb-main} under the assumptions \textnormal{(\hyperlink{H0}{H$_0$})}, \textnormal{(\hyperlink{H1}{H$_1$})}, and \textnormal{(\hyperlink{H3}{H$_3$})}. The proof of Theorem \ref{CD.boundedness} uses ideas based on the works by Ho--Kim--Winkert--Zhang \cite{Ho-Kim-Winkert-Zhang-2022} and Ho--Winkert \cite{Ho-Winkert-2023}.

\begin{proof}[Proof of Theorem \ref{CD.boundedness}]
	Let $u$ be a weak solution of problem \eqref{D}. We consider a covering of $\overline{\Omega}$ consisting of balls $\{B_i\}_{i=1}^m$ of radius $R$ such that each $\Omega_i:=B_i\cap\Omega$ ($i=1,\cdots,m$) is a Lipschitz domain. By \textnormal{(\hyperlink{H1}{H$_1$})}\textnormal{(\hyperlink{H1ii}{ii})}, we have $q(x)<p^*(x)$ for all $x\in\overline{\Omega}$. Therefore, by taking the radius $R$ smaller if necessary, we obtain
	\begin{align*}
		q_i^+:=\max_{x\in \overline{\Omega}_i} q(x)<{(p^*)}^{-}_{i}:=\min_{x\in \overline{\Omega}_i} p^*(x)\quad \text{for all } i\in\{1,\cdots,m\}.
	\end{align*}
	Hence, we have
	\begin{align}\label{D.loc.exp}
		\varepsilon:=\frac{1}{2}\min_{1\leq i\leq m} \left((p^*)_i^--q_i^+\right)>0.
	\end{align}
	Let $\kappa_{*}\geq 1$ be sufficiently large such that
	\begin{align}\label{k*}
		\int_{A_{\kappa_{*}}}\mathcal{S}(x,|\nabla u|)\,\mathrm{d} x+\int_{A_{\kappa_{*}}}\mathcal{S}(x,|u|)\,\mathrm{d} x+\int_{A_{\kappa_{*}}}\mathcal{S}^*(x,|u|)\,\mathrm{d} x<1,
	\end{align}
	where $A_\kappa$ for  $\kappa\in\R$, is defined in \eqref{def.Ak}, and $\mathcal{S}$ and $\mathcal{S}^*$ are given in \eqref{S} and \eqref{S*}. Then, let $\{\kappa_{n}\}_{n\in\N_0}$ and $\{v_{n}\}_{n\in\N_0}$ be as in \eqref{def.kn} and \eqref{def.vn}, respectively. Moreover, we define
	\begin{align}
		\label{Zn.def}
		Z_n:=\int_{A_{\kappa_{n}}}\mathcal{S}(x,|\nabla u|)\,\mathrm{d} x+\int_{A_{\kappa_{n}} }\mathcal{S}^*(x,v_{n})\,\mathrm{d} x.
	\end{align}
	Similarly to Section \ref{Section-3}, we easily obtain
	\begin{align}\label{C.Zn.dec}
		Z_{n+1}\leq Z_n\quad\text{for all } n\in \mathbb{N}_0
	\end{align}
	and
	\begin{align} \label{C.|A_{k_{n+1}}|}
		|A_{\kappa_{n+1}}| \leq \frac{2^{2(n+1)(q^*)^+}}{d_1\kappa_{*}^{(p^*)^{-}}} Z_n\leq \frac{2^{2(n+1)(q^*)^+}}{d_1} Z_n\quad \text{for all } n\in\N_0,
	\end{align}
	where
	\begin{align}\label{H2new}
		d_1:=\min\left\{1,\left(\frac{d}{2}\right)^{\left(\frac{p^*}{p}\right)^+},\left(\frac{d}{2}\log\left(\delta+\eta\right)\right)^{\left(\frac{q^*}{q}\right)^+}\right\}>0.
	\end{align}
	In the following, we will establish recursion inequalities for $\{Z_n\}_{n\in\N_0}$, which will then imply the desired conclusion by Lemma \ref{leRecur}. As before, throughout the remainder of the proof, the symbols $C_i$ ($i\in\N$) denote positive constants that are independent of $n$ and $\kappa_{*}$.

	\textbf{Claim 1:} There exists a constant $\mu>0$ such that
	\begin{align*}%\label{u^p^*}
		\int_{A_{\kappa_{n+1}} }\mathcal{S}^*(x,v_{n+1})\,\mathrm{d} x\leq  C_1 2^{\frac{2(p^*)^+(q^*)^+}{q^-}n}Z_{n}^{1+\mu}\quad  \text{for all } n\in\N_0.
	\end{align*}
	Indeed, we have
	\begin{align}\label{decompose1}
		\int_{A_{\kappa_{n+1}} }\mathcal{S}^*(x,v_{n+1})\,\mathrm{d} x=\int_{\Omega}\mathcal{S}^*(x,v_{n+1})\,\mathrm{d} x\leq \sum_{i=1}^m\int_{\Omega_{i}}\mathcal{S}^*(x,v_{n+1})\,\mathrm{d} x.
	\end{align}
	Let $i\in\{1,\cdots,m\}$. From \eqref{k*} and the relation between the norm and the modular (see Proposition \ref{P.nor-mod-S}), it follows that
	\begin{align*}
		\int_{\Omega_{i}}\mathcal{S}^*(x,v_{n+1})\,\mathrm{d} x\leq \|v_{n+1}\|_{\mathcal{S}^*,\Omega_{i}}^{(p^*)_i^-}.
	\end{align*}
	Applying Proposition \ref{emb-main} with $\Omega=\Omega_{i}$ yields
	\begin{align*}
		\int_{\Omega_{i}}\mathcal{S}^*(x,v_{n+1})\,\mathrm{d} x\leq C_2\|v_{n+1}\|_{W^{1,\mathcal{S}}(\Omega_{i})}^{(p^*)_i^-}.
	\end{align*}
	Next, applying Proposition \ref{P.nor-mod-S} and recalling \eqref{k*}, we arrive at
	\begin{align*}
		\int_{\Omega_{i}}\mathcal{S}^*(x,v_{n+1})\,\mathrm{d} x
		& \leq C_3\bigg(\int_{\Omega_{i}}\mathcal{S}(x,|\nabla v_{n+1}|)\,\mathrm{d} x+\int_{\Omega_{i}}\mathcal{S}(x,v_{n+1})\,\mathrm{d} x\bigg)^{\frac{(p^*)_i^-}{q_i^++\eps}} \\
		& \leq C_4\left(\int_{A_{\kappa_{n+1}}}\mathcal{S}(x,|\nabla v_{n+1}|)\,\mathrm{d} x+\int_{A_{\kappa_{n+1}}}\mathcal{S}^*(x,v_{n+1})\,\mathrm{d} x+|A_{\kappa_{n+1}}|\right)^{\frac{(p^*)_i^-}{q_i^++\eps}},
	\end{align*}
	where $\eps$ is given in \eqref{D.loc.exp}. Combining this estimate with \eqref{Zn.def}, \eqref{C.Zn.dec}, and \eqref{C.|A_{k_{n+1}}|}, we find that
	\begin{align*}%\label{grad3}
		\int_{\Omega_{i}}\mathcal{S}^*(x,v_{n+1})\,\mathrm{d} x
		 & \leq C_52^{\frac{2n(q^*)^+(p^*)_i^-}{q_i^+}}Z_{n}^{\frac{(p^*)_i^-}{q_i^++\eps}}.
	\end{align*}
	Substituting this estimate into \eqref{decompose1}, we deduce that
	\begin{align*}
		\int_{A_{\kappa_{n+1}}}\mathcal{S}^*(x,v_{n+1})\,\mathrm{d} x\leq C_62^{\frac{2n(p^*)^+(q^*)^+}{q^-}}Z_{n}^{1+\mu},
	\end{align*}
	where
	\begin{align*}
		0<\mu:=\min_{1\leq i\leq m} \frac{(p^*)_i^-}{q_i^++\eps}-1
	\end{align*}
	in view of \eqref{D.loc.exp}. This establishes Claim 1.

	{\bf Claim 2:}
	It holds that
	\begin{align*}%\label{D.grad}
		\int_{A_{\kappa_{n+1}}}\mathcal{S}(x,|\nabla u|)\,\mathrm{d} x\leq  C_7 2^{\l[\frac{2(p^*)^+(q^*)^+}{q^-}+2(q^*)^+\r]n}Z_{n-1}^{1+\mu}\quad \text{for all } n\in\N.
	\end{align*}
	To this end, we test  \eqref{def_sol_D} with $\varphi =v_{n+1} \in W_0^{1,\mathcal{S}}(\Omega)$, which gives
	\begin{align}\label{D.var.Eq}
		\int_{A_{\kappa_{n+1}}}
		\mathcal{A}(x,u,\nabla u) \cdot \nabla u\,\mathrm{d} x =\int_{A_{\kappa_{n+1}} }\mathcal{B}(x,u,\nabla u)v_{n+1}\,\mathrm{d} x.
	\end{align}
	On $A_{\kappa_{n+1}}$, we have $u\geq v_{n+1} >0$, $u> \kappa_{n+1} \geq 1$, and
	\begin{align*}
		\mathcal{S}^*(x,u)\geq d_1u\geq d_1,
	\end{align*}
	where $d_1$ is given in \eqref{H2new}. Using these estimates and \eqref{young}, we infer from \textnormal{(\hyperlink{H3}{H$_3$})}\textnormal{(\hyperlink{H3ii}{ii})} and \textnormal{(\hyperlink{H3}{H$_3$})}\textnormal{(\hyperlink{H3iii}{iii})}  that
	\begin{equation}\label{D-C-E1}
		\begin{aligned}
			&\int_{A_{\kappa_{n+1}}} \mathcal{A}(x,u,\nabla u)\cdot \nabla u\,\mathrm{d} x\\ & \geq \alpha_2\int_{A_{\kappa_{n+1}} }\mathcal{S}(x,|\nabla u|)\,\mathrm{d} x-a_3\int_{A_{\kappa_{n+1}} }\left[\mathcal{S}^*(x,u)+1\right]\,\mathrm{d} x  \\
			& \geq \alpha_2\int_{A_{\kappa_{n+1}} }\mathcal{S}(x,|\nabla u|)\,\mathrm{d} x-\frac{a_3(d_1+1)}{d_1}\int_{A_{\kappa_{n+1}} }\mathcal{S}^*(x,u)\,\mathrm{d} x
		\end{aligned}
	\end{equation}
	and
	\begin{equation}\label{D-C-E2}
		\begin{aligned}
			& \int_{A_{\kappa_{n+1}}} \mathcal{B}(x,u,\nabla 	u)v_{n+1}\,\mathrm{d} x \\
			& \leq \beta\int_{A_{\kappa_{n+1}} }\bigg[a(x)^{\frac{p^*(x)}{p(x)}}u^{p^*(x)-1}\\
			&\qquad\qquad +(b(x)\log(\delta+\eta u))^{\frac{q^*(x)}{q(x)}}u^{q^*(x)-1}+a(x)^{\frac{N+1}{N}}|\nabla u|^{\frac{p(x)}{(p^*)'(x)}} \\
			& \qquad\qquad+\left(b(x)\log(\delta+\eta u)\right)^{\frac{1}{q(x)}}\left(b(x)\log\big(\delta+\eta|\nabla u|\big)\right)^{\frac{1}{(q^*)'(x)}}|\nabla u|^{\frac{q(x)}{(q^*)'(x)}}+1\bigg]u\,\mathrm{d} x  \\
			& \leq \frac{\alpha_2}{2}\int_{A_{\kappa_{n+1}} }\mathcal{S}(x,|\nabla u|)\,\mathrm{d} x+C_8\int_{A_{\kappa_{n+1}} }\mathcal{S}^*(x,u)\,\mathrm{d} x.
		\end{aligned}
	\end{equation}
	Combining \eqref{D-C-E1} and \eqref{D-C-E2} with \eqref{D.var.Eq}, and then using \eqref{S-est.u}, we arrive at
	\begin{align*}
		& \int_{A_{\kappa_{n+1}} }\mathcal{S}(x,|\nabla u|)\,\mathrm{d} x \\
		& \leq C_9\int_{A_{\kappa_{n+1}}}\mathcal{S}^*(x,u)\,\mathrm{d} x \\
		& \leq C_9\int_{A_{\kappa_{n+1}}}\Bigg[a(x)^{\frac{p^*(x)}{p(x)}}\left((2^{n+2}-1)v_n\right)^{p^*(x)}\\
		&\qquad\qquad\qquad\qquad+\left(b(x)\log\big(\delta+\eta (2^{n+2}-1)v_n\big)\right)^{\frac{q^*(x)}{q(x)}}\left((2^{n+2}-1)v_n\right)^{q^*(x)}\Bigg]\,\mathrm{d} x.
	\end{align*}
	Using \eqref{ei-2} once more, the last inequality gives
	\begin{align*}
		\int_{A_{\kappa_{n+1}} }\mathcal{S}(x,|\nabla u|)\,\mathrm{d} x\leq C_{10}2^{2(q^*)^+n}\int_{A_{\kappa_{n}}}\mathcal{S}^*(x,v_n)\,\mathrm{d} x.
	\end{align*}
	Applying Claim 1 with $n-1$ in place of $n$, and absorbing the resulting fixed factor into the constant, proves Claim 2.

	Using Claims 1 and 2 along with  \eqref{C.Zn.dec} gives
	\begin{align}\label{Recur}
		Z_{n+1}\leq C_{11} b^nZ_{n-1}^{1+\mu}\quad \text{for all } n\in\N,
	\end{align}
	where $b:=2^{\l[\frac{2(p^*)^+(q^*)^+}{q^-}+2(q^*)^+\r]}>1$. This implies
	\begin{align*}
		Z_{2(n+1)}\leq C_{11} b^{2n+1}Z_{2n}^{1+\mu}\quad \text{for all } n\in\N_0.
	\end{align*}
	Thus, defining $\widetilde{Z}_n:=Z_{2n}$ and $\widetilde{b}:=b^2$, we obtain
	\begin{align}\label{Recur1}
		\widetilde{Z}_{n+1}\leq bC_{11} \widetilde{b}^n\widetilde{Z}_n^{1+\mu} \quad \text{for all } n\in\N_0.
	\end{align}
	In view of Lemma \ref{leRecur} for the case $\mu_1=\mu_2=\mu$,  it follows from \eqref{Recur1} that
	\begin{align}\label{Recur+1}
		Z_{2n}=\widetilde{Z}_n \to 0 \quad \text {as } n\to \infty
	\end{align}
	provided that
	\begin{align}\label{Z_0}
		\widetilde{Z}_{0}\leq (bC_{11})^{-\frac{1}{\mu}}\ \widetilde{b}^{-\frac{1}{\mu^{2}}}.
	\end{align}
	We further infer from \eqref{Recur} that
	\begin{align*}
		Z_{2(n+1)+1}\leq C_{11} b^{2(n+1)}Z_{2n+1}^{1+\mu}\quad \text{for all } n\in\N_0,
	\end{align*}
	which can be rewritten as
	\begin{align}\label{Recur2}
		\bar{Z}_{n+1}\leq \widetilde{b}C_{11} \widetilde{b}^n\bar{Z}_n^{1+\mu} \quad \text{for all } n\in\N_0,
	\end{align}
	where $\bar{Z}_n:=Z_{2n+1}$ and again $\widetilde{b}:=b^2$.  Applying Lemma \ref{leRecur} once more to \eqref{Recur2}, we conclude that
	\begin{align}\label{Recur+2}
		Z_{2n+1}=\bar{Z}_n \to 0 \quad
		\text {as } n\to \infty
	\end{align}
	provided that
	\begin{align}\label{Z_0'}
		\bar{Z}_{0}\leq (\widetilde{b}C_{11})^{-\frac{1}{\mu}}\ \widetilde{b}^{-\frac{1}{\mu^{2}}}.
	\end{align}
	It is clear that
	\begin{align*}
		\bar{Z}_{0}=Z_1\leq Z_0=\widetilde{Z}_0\leq \int_{A_{\kappa_{*}}}\mathcal{S}(x,|\nabla u|)\,\mathrm{d} x+\int_{A_{\kappa_{*}}}\mathcal{S}^*(x,u)\,\mathrm{d} x.
	\end{align*}
	Therefore, choosing $\kappa_{*}>1$ sufficiently large, we ensure that
	\begin{align*}
		\int_{A_{\kappa_{*}}}\mathcal{S}(x,|\nabla u|)\,\mathrm{d} x+\int_{A_{\kappa_{*}}}\mathcal{S}(x,|u|)\,\mathrm{d} x+\int_{A_{\kappa_{*}}}\mathcal{S}^*(x,u)\,\mathrm{d} x\leq \min\left\{1,(\widetilde{b}C_{11})^{-\frac{1}{\mu}}\ \widetilde{b}^{-\frac{1}{\mu^{2}}}\right\}.
	\end{align*}
	Consequently, \eqref{k*}, \eqref{Z_0} and \eqref{Z_0'} are satisfied. Hence, \eqref{Recur+1} and \eqref{Recur+2} follow, namely,
	\begin{align*}
		Z_n=\int_{A_{\kappa_{n}}}\mathcal{S}(x,|\nabla u|)\,\mathrm{d} x+\int_{A_{\kappa_{n}} }\mathcal{S}^*(x,u-\kappa_{n})\,\mathrm{d} x\to 0\quad  \text{as } n\to\infty.
	\end{align*}
	As a consequence,
	\begin{align*}
		\int_{\Omega}\mathcal{S}^*(x,(u-2\kappa_{*})_+)\,\mathrm{d} x=0.
	\end{align*}
	Therefore, $(u-2\kappa_{*})_{+}=0$ a.e.\,in $\Omega$ and hence
	\begin{align*}
		\esssup_{\Omega} u \leq 2\kappa_*.
	\end{align*}
	Replacing $u$ by $-u$ in the above argument, we likewise obtain
	\begin{align*}
		\esssup_{\Omega} (-u) \leq 2\kappa_*.
	\end{align*}
	Combining the two estimates, we conclude that
	\begin{align*}
		\|u\|_{L^\infty(\Omega)}\leq 2\kappa_{*}.
	\end{align*}
	The proof is complete.
\end{proof}

%********************************************************************
\section{H\"{o}lder continuity}\label{Holder}
%********************************************************************

In this section, we present the proof of Theorem \ref{Th.H} by combining ideas from Lady{\v{z}}enskaja--Ural{\cprime}ceva \cite{Ladyzenskaja-Uralceva-1968}, Fan--Zhao \cite{Fan-Zhao-1999}, and Ri--Kwon \cite{Ri-Kwon-2025}. More precisely, we adopt the strategy of Ri--Kwon \cite{Ri-Kwon-2025} for treating double phase growth, which combines the classical De Giorgi iteration scheme of Lady{\v{z}}enskaja--Ural{\cprime}ceva \cite{Ladyzenskaja-Uralceva-1968} and Fan--Zhao \cite{Fan-Zhao-1999} with the frozen functional technique, a standard tool in the regularity theory of minimizers of double phase functionals.

We first collect some preliminary results that will be used throughout this section. It is readily seen that, if $f \in C^{0, \frac{1}{|\log t|}}(\overline{\Omega})$ (see \eqref{def-log-continuity} for its definition), then there exists a constant $L(f) > 0$ such that
\begin{align}\label{L}
	R^{-\operatorname{osc}_{\Omega_{R}(x_0)} f}  \leq L(f), \quad \text{for all } \ \Omega_{R}(x_0)\ne \emptyset.
\end{align}
Moreover, for $\gamma\in (0,1]$, we denote by $C^{0, \gamma}(\overline{\Omega})$ the space of all $f\in C(\overline{\Omega})$ such that
\begin{align*}
	[f]_{0,\gamma} := \sup_{\substack{x,y \in \Omega \\ x \neq y}}
	\frac{|f(x)-f(y)|}{|x-y|^{\gamma}}<\infty.
\end{align*}

Let \textnormal{(\hyperlink{H0}{H$_0$})} hold. The conjugate function of $\mathcal{S}$, defined in \eqref{S}, is given by
\begin{align*}
	\widetilde{\mathcal{S}}(x,t):=\sup_{r\geq 0}\left\{rt-\mathcal{S}(x,r)\right\}\quad \text{for } (x,t)\in \Omega\times [0,\infty).
\end{align*}
Furthermore, set
\begin{align}\label{h+}
	\mathfrak{s}^+:=\sup_{(x,t)\in \Omega\times [0,\infty)}\frac{t\mathbf{s}(x,t)}{\mathcal{S}(x,t)},\quad \mathfrak{s}^-:=\inf_{(x,t)\in \Omega\times [0,\infty)}\frac{t\mathbf{s}(x,t)}{\mathcal{S}(x,t)},
\end{align}
where $\mathbf{s}(x, t):=\partial_t\mathcal{S}(x,t)$ as given in \textnormal{(\hyperlink{H5}{H$_5$})}, and
\begin{align*}
	\widetilde{\mathfrak{s}}^+:=\frac{\mathfrak{s}^-}{\mathfrak{s}^--1},\quad \widetilde{\mathfrak{s}}^-:=\frac{\mathfrak{s}^+}{\mathfrak{s}^+-1}.
\end{align*}

The next lemma follows directly from the definition of the conjugate function, see Bahrouni--Bahrouni--Missaoui \cite{Bahrouni-Bahrouni-Missaoui-2026}.

\begin{lemma} \label{LConjugate}
	For $r,t,\tau\geq 0$ and $x\in\Omega$, the following assertions hold.
	\begin{enumerate}
		\item[\textnormal{(i)}]
			$rt\leq \mathcal{S}(x,r)+\widetilde{\mathcal{S}}(x,t)$;
		\item[\textnormal{(ii)}]
			$\widetilde{\mathcal{S}}(x,\mathbf{s}(x,t))\leq \left(\mathfrak{s}^+-1\right)\mathcal{S}(x,t)$;
		\item[\textnormal{(iii)}]
			$\min\left\{\tau^{\mathfrak{s}^+},\tau^{\mathfrak{s}^-}\right\}\mathcal{S}(x,t)\leq \mathcal{S}(x,\tau t)\leq \max\left\{\tau^{\mathfrak{s}^+},\tau^{\mathfrak{s}^-}\right\}\mathcal{S}(x,t)$;
		\item[\textnormal{(iv)}]
			$\min\left\{\tau^{\widetilde{\mathfrak{s}}^+},\tau^{\widetilde{\mathfrak{s}}^-}\right\}\widetilde{\mathcal{S}}(x,t)\leq\widetilde{\mathcal{S}}(x,\tau t)\leq \max\left\{\tau^{\widetilde{\mathfrak{s}}^+},\tau^{\widetilde{\mathfrak{s}}^-}\right\}\widetilde{\mathcal{S}}(x,t)$.
	\end{enumerate}
\end{lemma}

In the remainder of this section, we assume that the hypotheses of Theorem \ref{Th.H} are satisfied. Let $u \in W^{1,\mathcal{S}}(\Omega)\cap L^\infty(\Omega)$ be a  weak solution of \eqref{E-H}, and let $M>0$ be such that $\| u\|_{L^{\infty}(\Omega)} \leq M$.  Furthermore, let $y \in \overline{\Omega}$ and $R\in (0,1]$ be such that $|B_R(y)| \leq 1$. For $r>0$, we set
\begin{align*}
	B_r:=B_r(y),\quad \Omega_r:=B_r(y)\cap \Omega, \quad (\partial \Omega)_{r}:= B_{r}(y)\cap \partial \Omega.
\end{align*}
For $i=1,2$, define
\begin{align}\label{Si}
	\mathcal{S}_i(t) := t^{p_i} + b_0 t^{q_i}\log(\delta+\eta t) \quad \text{for } t \geq 0,
\end{align}
where
\begin{align*}
	b_0 := \inf_{\Omega_R} b
\end{align*}
and
\begin{align*}
	f_1:=\inf_{\Omega_R}f,\quad f_2:=\sup_{\Omega_R}f\quad \text{for}\ f\in\{p,q\}.
\end{align*}
It is clear that the function  $t\mapsto\frac{\mathcal{S}_i(t)}{t^{\gamma}}$ with $0\leq \gamma\leq p_1$ is nondecreasing on $(0,\infty)$. Furthermore, $t\mapsto \mathcal{S}_i(t)$ is convex on $(0,\infty)$ since
\begin{align*}
	\mathcal{S}_i''(t)
	&=p_i(p_i-1)t^{p_i-2}
	+b_0q_i(q_i-1)t^{q_i-2}\log(\delta+\eta t)\\
	&\quad+b_0q_it^{q_i-1}\frac{\eta}{\delta+\eta t}
	+b_0\eta\frac{\eta(q_i-1)t^{q_i}+\delta q_it^{q_i-1}}
	{(\delta+\eta t)^2}>0,\quad \text{for all } t\in (0,\infty).
\end{align*}
Moreover, by \eqref{ei-2}, we have
\begin{equation}\label{S-ine}
	\mathcal{S}_i(\tau t)\leq \tau^{q^++1}\mathcal{S}_i(t),\quad\text{for all } t\geq 0 \ \text{and}\ \tau\geq 1.
\end{equation}
Finally, we note that condition \eqref{TA3} implies
\begin{align}\label{pq}
	t^{q(x)}\log(\delta+\eta t)\leq C_* t^{p(x)+\alpha}+\log(\delta+\eta)\quad\text{for all } (x,t)\in \Omega\times [0,\infty).
\end{align}

%********************************************************************
\subsection{Interior H\"older continuity}
%********************************************************************

In this subsection, we establish the H\"older continuity of $u$ in an interior ball $B_R\subset \Omega$.

Our first step is the following Caccioppoli-type inequality.

\begin{proposition}\label{Le-HCac.I}
	Let $0 < \tau < t \leq R$ and let $\omega = \pm u$. Then, for any $k \geq \sup_{B_R}\omega-\sigma M$, where
	\begin{align}\label{H2}
		\sigma := \min\left\{2, \frac{a_2(M)}{3Mb_1(M)} \right\},
	\end{align}
	the following assertions hold:
	\begin{enumerate}
		\item[\textnormal{(i)}]
			If
			\begin{align}\label{H3}
				b_0 \leq 4[b]_{0,\alpha} R^\alpha,
			\end{align}
			then
			\begin{align}\label{H4}
				\int_{A_{k,\tau}} |\nabla \omega|^{p_1} \, \mathrm{d} x \leq \gamma' \left(\frac{R}{t - \tau}\right)^{\alpha} \int_{A_{k,t}} \left(\frac{\omega - k}{t - \tau}\right)^{p_2} \, \mathrm{d} x + \gamma'_1 |A_{k,t}|,
			\end{align}
			where $A_{\kappa,r}:= \{x \in B_r \colon \omega(x) > \kappa \}$ for $r>0$ and $\kappa\in\mathbb{R}$,
			\begin{align*}
				\gamma' = \gamma'\left(\text{\rm data},a_1(M),a_2(M), b_1(M)\right), \quad \gamma'_1 =\gamma_1'\left(\text{\rm data},a_1(M),a_2(M)\right).
			\end{align*}
		\item[\textnormal{(ii)}]
			If
			\begin{align}\label{H5}
				b_0 > 4[b]_{0,\alpha} R^\alpha,
			\end{align}
			then
			\begin{align}\label{H6}
				\int_{A_{k,\tau}} \mathcal{S}_1(|\nabla \omega|) \, \mathrm{d} x \leq \gamma'' \int_{A_{k,t}} \mathcal{S}_2 \left(\frac{\omega - k}{t - \tau} \right) \, \mathrm{d} x + \gamma''_1 |A_{k,t}|,
			\end{align}
			where
			\begin{align*}
				\gamma'' = \gamma''(\text{\rm data},a_1(M),a_2(M)), \quad
				\gamma''_1 = \gamma''_1(\text{\rm data},a_1(M),a_2(M)).
			\end{align*}
	\end{enumerate}
\end{proposition}

\begin{proof}
	Let \( 0 < \tau < t \leq R \) and let \( k \geq 0 \). Let $\mathfrak{s}^+$ be given by \eqref{h+}, and let \( \zeta \in C^1(\mathbb{R}^N) \) be a function satisfying
	\begin{equation}\label{zeta}
		\begin{aligned}
			0 \leq \zeta(x) \leq 1, \quad |\nabla \zeta(x)| \leq \frac{2}{t-\tau} \quad\text{for } x \in \mathbb{R}^N, \\\
			\zeta(x) = 1 \quad\text{for } x \in B_\tau, \quad \operatorname{supp} \zeta \subset B_t.
		\end{aligned}
	\end{equation}
	Taking $\varphi = \zeta^{\mathfrak{s}^+} (\omega-k)_+$ as a test function in \eqref{E-V}, we arrive at
	\begin{equation}\label{H7}
		\begin{aligned}
		&\int_{A_{k,t}} \zeta^{\mathfrak{s}^+} \mathcal{A}(x,u,\nabla u)\cdot \nabla u \, \mathrm{d} x\\
		& \leq  \mathfrak{s}^+ \left|\int_{A_{k,t}} \zeta^{\mathfrak{s}^+-1} \mathcal{A}(x,u,\nabla u)\cdot\nabla \zeta \, (\omega - k) \, \mathrm{d} x\right| + \left|\int_{A_{k,t}} \zeta^{\mathfrak{s}^+} \mathcal{B}(x,u,\nabla u) (\omega - k) \, \mathrm{d} x\right|.
		\end{aligned}
	\end{equation}
	Assumption \textnormal{(\hyperlink{H5}{H$_5$})}\textnormal{(\hyperlink{H5ii}{ii})} yields
	\begin{align}\label{H7-1}
		\int_{A_{k,t}} \zeta^{\mathfrak{s}^+} \mathcal{A}(x,u,\nabla u)\cdot \nabla u \,\mathrm{d} x\geq a_2(M) \int_{A_{k,t}} \zeta^{\mathfrak{s}^+} \mathcal{S}(x,|\nabla \omega|)\, \mathrm{d} x - a_3 |A_{k,t}|.
	\end{align}
	Furthermore, by \textnormal{(\hyperlink{H5}{H$_5$})}\textnormal{(\hyperlink{H5i}{i})}, we obtain
	\begin{equation}\label{H9}
		\begin{aligned}
			&\mathfrak{s}^+ \left| \int_{A_{k,t}} \zeta^{\mathfrak{s}^+ - 1} \mathcal{A}(x, u, \nabla u)\cdot\nabla \zeta\,(\omega - k)\, \mathrm{d} x \right|\\
			&\leq  a_1(M)\mathfrak{s}^+\int_{A_{k,t}} \zeta^{\mathfrak{s}^+ - 1}\left[\mathbf{s}(x, |\nabla u|)+ 1\right]|\nabla \zeta|(\omega-k)\, \mathrm{d} x.
		\end{aligned}
	\end{equation}
	Invoking Lemma \ref{LConjugate}, for $\varepsilon\in (0,1)$, we have
	\begin{align*}
		\zeta^{\mathfrak{s}^+ - 1}\mathbf{s}(x, |\nabla u|)|\nabla \zeta|(\omega-k)
		& \leq  \widetilde{\mathcal{S}}\left(x,\varepsilon\zeta^{\mathfrak{s}^+ - 1}\mathbf{s}(x, |\nabla u|)\right)+\mathcal{S}\left(x,\varepsilon^{-1}|\nabla \zeta|(\omega-k)\right)\\
		& \leq  \varepsilon^{\frac{\mathfrak{s}^+}{\mathfrak{s}^+-1}}\zeta^{\mathfrak{s}^+}\widetilde{\mathcal{S}}\left(x,\mathbf{s}(x,|\nabla u|)\right)+\varepsilon^{-\mathfrak{s}^+}\mathcal{S}\left(x,|\nabla \zeta|(\omega-k)\right) \\
		& \leq  \left(\mathfrak{s}^+-1\right)\varepsilon^{\frac{\mathfrak{s}^+}{\mathfrak{s}^+-1}}\zeta^{\mathfrak{s}^+}\mathcal{S}\left(x,|\nabla u|\right)+\varepsilon^{-\mathfrak{s}^+}\mathcal{S}\left(x,|\nabla \zeta|(\omega-k)\right).
	\end{align*}
	Choosing $\varepsilon\in (0,1)$ such that
	\begin{align*}
		a_1(M)\mathfrak{s}^+\left(\mathfrak{s}^+-1\right)\varepsilon^{\frac{\mathfrak{s}^+}{\mathfrak{s}^+-1}}<\frac{a_2(M)}{3},
	\end{align*}
	we infer from the last estimate and \eqref{zeta} that
	\begin{align}\label{H9-1}
		a_1(M)\mathfrak{s}^+ \zeta^{\mathfrak{s}^+ - 1}|\nabla \zeta|(\omega-k)\mathbf{s}(x, |\nabla u|) \leq \frac{a_2(M)}{3}\zeta^{\mathfrak{s}^+} \mathcal{S}\left(x,|\nabla \omega|\right)+C_1 \mathcal{S}\left(x,\frac{\omega-k}{t-\tau}\right).
	\end{align}
	Here and throughout the rest of this proof, $C_i$ ($i\in\mathbb{N}$) denotes a positive constant depending only on the data, $a_1(M)$ and $a_2(M)$. Using $t\leq 1+\max\{d^{-1},(d\log(\delta+\eta))^{-1}\}\mathcal{S}(x,t)$ for $(x,t)\in\Omega\times [0,\infty)$, we also have
	\begin{align}\label{H9-2}
		\mathfrak{s}^+a_1(M)\zeta^{\mathfrak{s}^+ - 1}|\nabla \zeta|(\omega-k)\leq C_2 \mathcal{S}\left(x,\frac{\omega-k}{t-\tau}\right)+C_3.
	\end{align}
	From \eqref{H9}, \eqref{H9-1}, and \eqref{H9-2}, we arrive at
	\begin{equation}\label{H7-2}
		\begin{aligned}
			&\mathfrak{s}^+ \left| \int_{A_{k,t}} \zeta^{\mathfrak{s}^+ - 1} \mathcal{A}(x, u, \nabla u)\cdot\nabla \zeta\, (\omega - k)\, \mathrm{d} x \right|  \\
			& \leq \frac{a_2(M)}{3} \int_{A_{k,t}} \zeta^{\mathfrak{s}^+} \mathcal{S}(x, |\nabla \omega|)\, \mathrm{d} x + C_4 \int_{A_{k,t}} \mathcal{S}\left(x, \frac{\omega- k}{t - \tau}\right)\, \mathrm{d} x + C_5 |A_{k,t}|.
		\end{aligned}
	\end{equation}
	On the other hand, by the assumption on $k$ and the definition of $\sigma$ in \eqref{H2}, we have
	\begin{align*}
		0 < \omega(x) - k \leq \sigma M \leq \frac{a_2(M)}{3 b_1(M)} \quad \text{for } x \in A_{k,t}.
	\end{align*}
	Together with \textnormal{(\hyperlink{H5}{H$_5$})}\textnormal{(\hyperlink{H5iii}{iii})}, this gives
	\begin{align}\label{H7-3}
		\left| \int_{A_{k,t}} \zeta^{\mathfrak{s}^+} \mathcal{B}(x,u,\nabla u)(\omega - k)\, \mathrm{d} x \right|
		\leq \frac{a_2(M)}{3} \left( \int_{A_{k,t}} \zeta^{\mathfrak{s}^+} \mathcal{S}(x,|\nabla \omega|)\, \mathrm{d} x +  b_2|A_{k,t}|\right).
	\end{align}
	Combining \eqref{H7}, \eqref{H7-1}, \eqref{H7-2}, and \eqref{H7-3}, we derive
	\begin{align}\label{H14}
		\int_{A_{k,\tau}} \mathcal{S}(x, |\nabla \omega|)\, \mathrm{d} x \leq \int_{A_{k,t}} \zeta^{\mathfrak{s}^+} \mathcal{S}(x, |\nabla \omega|)\, \mathrm{d} x \leq C_6 \int_{A_{k,t}}  \mathcal{S}\left(x, \frac{\omega- k}{t - \tau}\right)\, \mathrm{d} x + C_7 |A_{k,t}|.
	\end{align}

	(i) \textbf{Case \eqref{H3}:} By taking $R>0$ sufficiently small, we may assume that $b(x)<\frac{d}{2}$ for all $x\in B_R$. Consequently,
	\begin{align}\label{a.BR}
		a(x)>\frac{d}{2} \quad \text{for all }x\in B_R.
	\end{align}
	By \eqref{H3}, it holds that
	\begin{align*}
		b(x) = b(x) - b_0 + b_0 \leq 6 [b]_{0,\alpha} R^{\alpha} \quad \text{for all } x \in B_R.
	\end{align*}
	Using this and \eqref{pq} along with $0<t-\tau<R \leq 1$ and $\omega-k\leq\sigma M\leq\frac{a_2(M)}{3b_1(M)}$, we find that
	\begin{equation}\label{H15}
		\begin{aligned}
			\mathcal{S}\left(x, \frac{\omega- k}{t - \tau}\right)
			& \leq  \|a\|_{L^\infty(\Omega)}\left[\left(\frac{\omega- k}{t - \tau}\right)^{p_2}+1\right] +  b(x) \left[C_*\left(\frac{\omega- k}{t - \tau}\right)^{p(x)+\alpha}+\log(\delta+\eta)\right]  \\
			& \leq  \|a\|_{L^\infty(\Omega)}\left(\frac{\omega- k}{t - \tau}\right)^{p_2}+  C_* b(x) \left(\frac{\omega- k}{t - \tau}\right)^{p_2+\alpha}+C_8  \\
			& \leq  \|a\|_{L^\infty(\Omega)}\left(\frac{\omega- k}{t - \tau}\right)^{p_2}+ 6 C_*[b]_{0,\alpha} R^{\alpha} \left(\frac{\omega- k}{t - \tau}\right)^{p_2+\alpha} + C_8   \\
			& \leq \left[\|a\|_{L^\infty(\Omega)} + 6 C_*[b]_{0,\alpha}\left(\frac{a_2(M)}{3b_1(M)}\right)^{\alpha}\right]  \left(\frac{R}{t - \tau}\right)^{\alpha}  \left(\frac{\omega- k}{t - \tau}\right)^{p_2} +C_8.
		\end{aligned}
	\end{equation}
	Combining \eqref{H15} with the estimate $|\nabla\omega|^{p_1}\leq \frac{2}{d}\mathcal{S}(x,|\nabla\omega|)+1$, which follows from \eqref{a.BR}, we derive \eqref{H4} from \eqref{H14}.

	\vskip5pt
	(ii) \textbf{Case \eqref{H5}:} In this case,
	\begin{align}\label{a2}
		b(x) = b(x) - b_0 + b_0 \leq 2 [b]_{0,\alpha} R^{\alpha} + b_0 \leq \frac{3}{2} b_0 \quad \text{for all } x \in B_R.
	\end{align}
	Moreover,
	\begin{align}\label{e1}
		b(x)\log(\delta+\eta|\nabla\omega|)\leq \mathcal{S}(x,|\nabla\omega|)+\|b\|_{L^{\infty}(\Omega)}\log(\delta+\eta)\quad \text{for a.a.\,} x\in\Omega.
	\end{align}
	On the other hand, if $|\nabla\omega|\geq 1$, then
	\begin{align*}
		|\nabla\omega|^{p_1}\leq \frac{1}{d}a(x)|\nabla\omega|^{p(x)}+\frac{1}{d}(\log(\delta+\eta))^{-1}b(x)|\nabla\omega|^{q(x)}\log(\delta+\eta|\nabla\omega|).
	\end{align*}
	Thus,
	\begin{align}\label{e2}
		|\nabla\omega|^{p_1}\leq \frac{1}{d}\left(1+(\log(\delta+\eta))^{-1}\right)\mathcal{S}(x,|\nabla\omega|)+1\quad \text{for a.a.\,} x\in\Omega.
	\end{align}
	Using \eqref{e1} and \eqref{e2}, and then applying \eqref{H14} and \eqref{a2}, we arrive at
	\begin{align*}
		&\int_{A_{k,\tau}}   \mathcal{S}_1(|\nabla \omega|)\, \mathrm{d} x\\
		& \leq \int_{A_{k,\tau}}\left[|\nabla \omega|^{p_1}+b(x)\left(|\nabla \omega|^{q(x)}+1\right)\log(\delta+\eta|\nabla\omega|)\right]\, \mathrm{d} x   \\
		& \leq \int_{A_{k,\tau}}\left[C_9\mathcal{S}(x,|\nabla\omega|)+C_{10}\right]\, \mathrm{d} x \\
		& \leq C_9 \int_{A_{k,t}} \left[\|a\|_{L^\infty(\Omega)}\left(\left(\frac{\omega- k}{t - \tau}\right)^{p_2}+1\right) + \frac{3}{2}b_0 \left(\left(\frac{\omega- k}{t - \tau}\right)^{q_2}+1\right)\log\left(\delta+\eta\frac{\omega- k}{t - \tau}\right)\right]\, \mathrm{d} x\\
		& \qquad +C_{10}|A_{k,t}|.
	\end{align*}
	From this and the estimate
	\begin{align*}
		\log\left(\delta+\eta\frac{\omega- k}{t - \tau}\right)\leq \left(\frac{\omega- k}{t - \tau}\right)^{p_2}+C_{11},
	\end{align*}
	we obtain
	\begin{align*}
		\int_{A_{k,\tau}}   \mathcal{S}_1(|\nabla \omega|)\, \mathrm{d} x \leq C_{12} \int_{A_{k,t}} \mathcal{S}_2\left(\frac{\omega- k}{t - \tau}\right)\, \mathrm{d} x+C_{13}|A_{k,t}|.
	\end{align*}
	This proves \eqref{H6} and completes the proof.
\end{proof}

\begin{lemma}\label{Le-HR.I}
	Let $R' \in (0,1)$ satisfy $B_{R'} \subset \Omega$.
		Then there exist constants
	$0 < \theta < 1$ and $C_0 > 0$ such that for any $0 < R \leq R'$,
	\begin{align*}%\label{HL2-A2}
		\operatorname{osc}_{B_R} u \leq C R'^{-\beta'} R^{\beta'},
	\end{align*}
	where
	\begin{align*}%\label{HL2-A3*}
		\beta' = \min\left\{1, -\log_4\theta \right\}, \quad C = 4 \max \left\{ C_0 R', \operatorname{osc}_{B_{R'}}u \right\}.
	\end{align*}
\end{lemma}

\begin{proof}
	Let $\sigma$ be given by \eqref{H2} and define \begin{align*}%\label{Hpsi}
		\psi := \ell^{-1} \operatorname{osc}_{B_R} u\quad \text{with}\quad \ell = \max\left\{2, \frac{2}{\sigma}\right\}.
	\end{align*}
	Clearly, at least one of the functions $\omega = \pm u$ satisfies
	\begin{align}\label{omega}
		\left| \{ x \in B_{R/2} \colon \omega(x) > \sup_{B_R} \omega - \frac{1}{2} \operatorname{osc}_{B_R} u \} \right| \leq \frac{1}{2} \left| B_{R/2} \right|.
	\end{align}
	In what follows, we denote by $\omega$ the function $u$ or $-u$ for which \eqref{omega} holds. We now distinguish between the two cases \eqref{H3} and \eqref{H5}.

	\textbf{Case \eqref{H3}:}
	Our first goal is to show that there exists a number
	\begin{align*}
		\theta' = \theta'(\text{\rm data},a_1(M),a_2(M), b_1(M)) > 0
	\end{align*}
	such that, whenever
	\begin{align}\label{H23}
		| A_{k_0,R/2}| \leq \theta' R^N,
	\end{align}
	either
	\begin{align}\label{H24}
		\sup_{B_{R/4}} \omega \leq k_0 + \frac{N_0}{2},
	\end{align}
	or
	\begin{align}\label{H25}
		N_0 \leq (\gamma'_1)^{\frac{1}{p_1}} R,
	\end{align}
	where $0 < N_0 < \psi$, $k_0 = \sup_{B_R} \omega - N_0$, and $\gamma'_1$ is the constant appearing in \eqref{H4}.

	To establish the claim, for $m \in \N_0$, define
	\begin{align}
		\rho_m &= \left(\frac{1}{4} + \frac{1}{2^{m+2}}\right) R, \quad k_m = k_0 + \frac{N_0}{2} - \frac{N_0}{2^{m+1}},\label{H26}\\
		D_{m+1} &= A_{k_m, \rho_{m+1}} \setminus A_{k_{m+1}, \rho_{m+1}}\quad \text{and}\quad y_m = R^{-N} |A_{k_m, \rho_{m}}|.\label{H27}
	\end{align}
	Since $k_m \geq \sup_{B_R} \omega - \sigma M$, we may apply \eqref{H4} with $k = k_m$, $t = \rho_m$, and $\tau = \rho_{m+1}$. This yields
	\begin{align}\label{H27'}
		\int_{A_{k_m,\rho_{m+1}}} |\nabla \omega|^{p_1} \, \mathrm{d} x \leq \gamma'2^{ (m+3)(\alpha + p^+)} R^{-p_2} \int_{A_{k_m,\rho_m}} (\omega - k_m)^{p_2} \, \mathrm{d} x + \gamma'_1 |A_{k_m,\rho_m}|.
	\end{align}
	Using \eqref{H27} together with the estimate $0 < \omega(x) - k_m \leq N_0$ for $x \in A_{k_m, \rho_m}$, we infer from \eqref{H27'} that
	\begin{equation}\label{H28}
		\begin{aligned}
		 	\int_{D_{m+1}} |\nabla \omega|^{p_1} \, \mathrm{d} x
		 	& \leq \Big[ \gamma'2^{ (m+3)(\alpha + p^+)} R^{-p_2}N_0^{p_2} + \gamma'_1 \Big] |A_{k_m, \rho_m}| \\
			& \leq \Big[ \gamma'2^{ (m+3)(\alpha + p^+)}N_0^{p_2} R^{-p_2}+\gamma'_1 \Big] R^{N}y_m.
		\end{aligned}
	\end{equation}
	Assume now that \eqref{H25} does not hold, that is, $N_0 > (\gamma'_1)^{\frac{1}{p_1}} R$. Then \eqref{H28} implies
	\begin{equation}\label{H29'}
		\begin{aligned}
			\int_{D_{m+1}} |\nabla \omega|^{p_1} \, \mathrm{d} x
			& \leq \Big[ \gamma'2^{ (m+3)(\alpha + p^+)}N_0^{p_2-p_1} R^{p_1-p_2}+1 \Big]N_0^{p_1} R^{N-p_1}y_m \\
			& \leq C_1 2^{m(\alpha + p^+)} N_0^{p_1} R^{N - p_1} y_m,
		\end{aligned}
	\end{equation}
	where we have used \textnormal{(\hyperlink{H4}{H$_4$})}\textnormal{(\hyperlink{H4i}{i})} and the estimate $N_0<\psi\leq \frac{a_2(M)}{3b_1(M)}$. Here and in the following of the proof, $C_i$ ($i\in\mathbb{N}$) denotes a positive constant depending only on the data, $a_1(M)$, $a_2(M)$ and $b_1(M)$. Applying H\"older's inequality and \eqref{H29'}, we obtain
	\begin{align}\label{H29}
		\int_{D_{m+1}} |\nabla \omega| \, \mathrm{d} x \leq \left(\int_{D_{m+1}} |\nabla \omega|^{p_1} \, \mathrm{d} x\right)^{\frac{1}{p_1}}|D_{m+1}|^{1-\frac{1}{p_1}}\leq C_2 2^{\frac{m(\alpha + p^+)}{p_1}} N_0 R^{N-1}y_m.
	\end{align}
	Next, applying Lemma \ref{L2.6} with $k = k_m$, $l = k_{m+1}$, $\rho = \rho_{m+1}$, and $u = \omega$, it follows that
	\begin{align}\label{H30}
		\int_{D_{m+1}} |\nabla \omega| \, \mathrm{d} x \geq 2^{- (m+2) }\beta^{ - 1} N_0 |A_{k_{m+1}, \rho_{m+1}}|^{1 - \frac{1}{N}} \rho_{m+1}^{-N} |B_{\rho_{m+1}} \setminus A_{k_m, \rho_{m+1}}|.
	\end{align}
	Since $\frac{R}{4}<\rho_m \leq \frac{R}{2}$ and $k_0 \leq k_m$ for all $m\in \N_0$, we have
	\begin{align*}
		B_{\rho_{m+1}} \setminus A_{k_m, \rho_{m+1}} \supset B_{R/4} \setminus A_{k_0,R/4}.
	\end{align*}
	Consequently, \eqref{H23} yields
	\begin{align}\label{H31}
		|B_{\rho_{m+1}} \setminus A_{k_m, \rho_{m+1}}| \geq |B_{R/4}| - |A_{k_0,R/4}| \geq (\omega_N 4^{-N} - \theta') R^N.
	\end{align}
	Choosing $\theta' > 0$ with $\theta' \leq \frac{1}{2}\omega_N4^{-N}$, we infer from \eqref{H30} and \eqref{H31} that
	\begin{align*}
		\int_{D_{m+1}} |\nabla \omega| \, \mathrm{d} x \geq N_0 \omega_N \beta^{-1} 2^{- (m+N+3)} R^{N-1} y_{m+1}^{1 - \frac{1}{N}}.
	\end{align*}
	Comparing this estimate with \eqref{H29}, we arrive at
	\begin{align*}
		y_{m+1}^{1 - \frac{1}{N}} \leq C_3 2^{m \left(\frac{\alpha+p^+}{p^-}+1\right)} y_m.
	\end{align*}
	Thus,
	\begin{align}\label{H32}
		y_{m+1} \leq C_4 b^m y_m^{1 + \varepsilon_0},
	\end{align}
	where
	\begin{align*}
		b = 2^\frac{ (\alpha+2p^+)N}{(N-1)p^-}>1\quad\text{and} \quad \varepsilon_0 = \frac{1}{N-1}>0.
	\end{align*}
	Hence, if \eqref{H23} holds with $\theta'$ defined by
	\begin{align}\label{H33}
		\theta' := \min \left\{ 2^{-1}4^{-N} \omega_N, C_4^{-\frac{1}{\varepsilon_0}} b^{-\frac{1}{\varepsilon_0^2}} \right\},
	\end{align}
	then the sequence $\{y_m\}_{m\in\mathbb{N}_0}$ satisfies the recursive inequality \eqref{H32} and
	\begin{align*}
		y_0 = R^{-N} \left| A_{k_0, R/2} \right| \leq \theta' \leq C_4^{-\frac{1}{\varepsilon_0}} b^{-\frac{1}{\varepsilon_0^2}}.
	\end{align*}
	Therefore, Lemma \ref{leRecur} applies and yields $y_m \to 0$ as $m \to \infty$. Thus, \eqref{H24} holds.

	The next step is to prove the existence of an integer $s' = s'(\text{\rm data},a_1(M),a_2(M),b_1(M)) > 2$ such that, for $\theta'$ defined in \eqref{H33}, either
	\begin{align}\label{H34}
		\psi \leq 2^{s'} (\gamma'_1)^{\frac{1}{p_1}} R
	\end{align}
	or \eqref{H23} holds with
	\begin{align}\label{H35}
		k_0 = \sup_{B_R} \omega - 2^{-s'+1} \psi.
	\end{align}

	Suppose that \eqref{H34} does not hold, that is,
	\begin{align}\label{H36}
		\psi > 2^{s'} (\gamma'_1)^{\frac{1}{p_1}} R,
	\end{align}
	where $s' > 2$ is an integer to be specified later. Define
	\begin{align}\label{H37}
		\tilde{k}_i := \sup_{B_R} \omega - 2^{-i} \psi, \quad \tilde{D}_i:= A_{\tilde{k}_i, R/2} \setminus A_{\tilde{k}_{i+1}, R/2} \quad \text{for } i = 0, 1,  \dots, s' - 1.
	\end{align}
	By the definition of $\psi$ and \eqref{omega} we have
	\begin{align}\label{Atik}
		|A_{\tilde{k}_i, R/2}| \leq \frac{1}{2} |B_{R/2}| \quad \text{for } i = 0, 1, \dots, s' - 1.
	\end{align}
	Applying \eqref{H4} with $t = R$, $\tau = R/2$, and $k = \tilde{k}_i$ for $i = 0,1,2,\dots,s' - 2$, we obtain
	\begin{align*}
		\int_{\tilde{D}_i} |\nabla \omega|^{p_1} \, \mathrm{d} x
		& \leq \gamma' \left(\frac{R}{R-\frac{R}{2}}\right)^{\alpha} \int_{A_{\tilde{k}_i,R}} \left(\frac{\omega - \tilde{k}_i}{R-\frac{R}{2}}\right)^{p_2} \, \mathrm{d} x + \gamma'_1 |A_{\tilde{k}_i,R}| \\
		& \leq \gamma'2^{\alpha+p^+}R^{-p_2}\left(2^{-i}\psi\right)^{p_2}|A_{\tilde{k}_i,R}|+ \gamma'_1 |A_{\tilde{k}_i,R}|.
	\end{align*}
	Together with \eqref{H36}, \textnormal{(\hyperlink{H4}{H$_4$})}\textnormal{(\hyperlink{H4i}{i})}, and the estimate $2^{-i}\psi\leq \sigma M\leq \frac{a_2(M)}{3b_1(M)}$, this gives
	\begin{align}\label{H39}
		\int_{\tilde{D}_i} |\nabla \omega|^{p_1} \, \mathrm{d} x \leq \left(2^{\alpha + p^+ } \gamma' L(p)\left(\frac{a_2(M)}{3b_1(M)}\right)^{p_2 - p_1}+1\right) \omega_N(2^{-i} \psi)^{p_1} R^{N - p_1},
	\end{align}
	where $L(\cdot)$ is defined in \eqref{L}.

	On the other hand, applying Lemma \ref{L2.6} with $\rho = R/2$, $l = \tilde{k}_{i+1}$, $k = \tilde{k}_i$ for $i = 0, 1, 2, \dots, s' - 2$, and $u = \omega$ along with \eqref{Atik}, we find that
	\begin{align*}
		\int_{\tilde{D}_i} |\nabla \omega| \, \mathrm{d} x
		&\geq 2^{-i - 1}\psi \beta^{-1}  (R/2)^{-N} |A_{\tilde{k}_{i+1}, R/2}|^{1 - \frac{1}{N}} \left(|B_{R/2}| - \frac{1}{2} |B_{R/2}|\right)\\
		&= \omega_N \beta^{-1} 2^{-i - 2} \psi |A_{\tilde{k}_{i+1}, R/2}|^{1 - \frac{1}{N}}.
	\end{align*}
	Using H\"{o}lder's inequality and \eqref{H39}, we get, for $i = 0, 1, 2, \dots, s' - 2$,
	\begin{align*}
		|A_{\tilde{k}_{s'-1}, R/2}|^{1 - \frac{1}{N}}
		&\leq |A_{\tilde{k}_{i+1}, R/2}|^{1 - \frac{1}{N}}\\
		& \leq \psi^{-1}\omega_N^{-1} \beta 2^{i+2}\left(\int_{\tilde{D}_i}|\nabla\omega|^{p_1}\right)^{\frac{1}{p_1}} |\tilde{D}_i|^{1-\frac{1}{p_1}} \\
		& \leq 4\beta \omega_N^{\frac{1}{p_1}-1} \left(2^{\alpha + p^+ } \gamma' L(p)\left(\frac{a_2(M)}{3b_1(M)}\right)^{p_2 - p_1}+1\right)^{\frac{1}{p_1}} |\tilde{D}_i|^{\frac{p_1-1}{p_1}} R^{\frac{N - p_1}{p_1}}.
	\end{align*}
	This implies
	\begin{align}\label{H40}
		|A_{\tilde{k}_{s' - 1}, R/2}|^{\frac{(N - 1)p_1}{N(p_1 - 1)}} \leq C_5 |\tilde{D}_i| R^{\frac{N - p_1}{p_1 - 1}}.
	\end{align}
	Summing \eqref{H40}	with respect to $i$ from $0$ to $s'-2$ and using
	\begin{align*}
		\sum_{i=0}^{s'-2}|\tilde{D}_i|\leq |B_{R/2}|,
	\end{align*}
	we arrive at
	\begin{align*}
		|A_{\tilde{k}_{s' - 1}, R/2}|^{\frac{(N - 1)p_1}{N(p_1 - 1)}} \leq \frac{C_5}{s'-1} |B_{R/2}| R^{\frac{N - p_1}{p_1 - 1}}=\frac{C_6}{s'-1} R^{\frac{p_1(N-1)}{p_1 - 1}}.
	\end{align*}
	Therefore,
	\begin{align}\label{H41}
		|A_{\tilde{k}_{s' - 1}, R/2}| \leq \left(\frac{C_6}{s'-1}\right)^{\frac{{N(p_1 - 1)}}{(N - 1)p_1}} R^N.
	\end{align}
	Choosing $s'>2$ sufficiently large so that
	\begin{align*}
		\frac{C_6}{s'-1}<1 \quad\text{and}\quad \left(\frac{C_6}{s'-1}\right)^{\frac{{N(p^- - 1)}}{(N - 1)p^-}} <\theta',
	\end{align*}
	we infer that
	\begin{align*}
		|A_{\tilde{k}_{s' - 1}, R/2}| \leq \theta' R^N.
	\end{align*}
	Thus, \eqref{H23} holds with $k_0$ as in \eqref{H35}. Therefore, as proved above, if \eqref{H34} does not hold, then at least one of \eqref{H24} and \eqref{H25}, with $N_0 = 2^{-s' + 1} \psi$, holds. If \eqref{H24} holds, then
	\begin{align*}
		\sup_{B_{R/4}} \omega \leq \sup_{B_R} \omega - 2^{-s'} \psi = \sup_{B_R} \omega - 2^{-s' }\ell^{ - 1} \operatorname{osc}_{B_R} u.
	\end{align*}
	Moreover,
	\begin{align*}
		\sup_{B_R} \omega - \sup_{B_{R/4}} \omega \leq \operatorname{osc}_{B_R} u - \operatorname{osc}_{B_{R/4}} u.
	\end{align*}
	Hence,
	\begin{align}\label{H42}
		\operatorname{osc}_{B_{R/4}} u \leq (1 - 2^{-s' }\ell^{ - 1}) \operatorname{osc}_{B_R} u.
	\end{align}
	If \eqref{H25} holds with $N_0 = 2^{-s' + 1} \psi$, then
	\begin{align}\label{H43}
		\operatorname{osc}_{B_R} u \leq \ell 2^{s' - 1} (\gamma'_1)^{\frac{1}{p_1}} R.
	\end{align}
	If \eqref{H34} holds, then
	\begin{align*}%\label{H43'}
		\operatorname{osc}_{B_R} u \leq 2^{s'} (\gamma'_1)^{\frac{1}{p_1}} R.
	\end{align*}
	In any case, one of the following alternatives holds:
	\begin{align*}
		\operatorname{osc}_{B_R} u \leq \ell 2^{s'} (\gamma'_1)^{\frac{1}{p_1}} R\quad \text{or}\quad\operatorname{osc}_{B_{R/4}} u \leq (1 - 2^{-s' }\ell^{ - 1}) \operatorname{osc}_{B_R} u.
	\end{align*}

	\textbf{Case \eqref{H5}:} The proof of this case follows the same strategy as that of \eqref{H3}, so we focus only on the differences. We first show that there exists a constant
	\begin{align*}
		\theta'' = \theta''(\text{\rm data}, a_1(M), a_2(M),b_1(M)) > 0
	\end{align*}
	such that, if
	\begin{align}\label{H44}
		\big| A_{k_0, R/2} \big| \leq \theta'' R^N,
	\end{align}
	then at least one of the following two alternatives holds:
	\begin{align}\label{H45}
		\sup_{B_{R/4}} \omega \leq k_0 + \frac{N_0}{2}
	\end{align}
	or
	\begin{align}\label{H46} N_0 \leq (\gamma_1'')^{\frac{1}{p_1}} R,
	\end{align}
	where $0 < N_0 < \psi$, $k_0 = \sup_{B_R} \omega-N_0$, and $\gamma_1''$ is the constant appearing in \eqref{H6}.

	We make use of functions $\mathcal{S}_i$ ($i=1,2$) defined in \eqref{Si}. By monotonicity, convexity, and \eqref{S-ine}, we have
	\begin{align*}
		\frac{\mathcal{S}_1(N_0 R^{-1})}{(N_0 R^{-1})^{p_1}} \leq \frac{\mathcal{S}_1(|\nabla \omega| + N_0 R^{-1})}{(|\nabla \omega| + N_0 R^{-1})^{p_1}}\leq \frac{2^{q^+}\left[\mathcal{S}_1(|\nabla \omega|) + \mathcal{S}_1(N_0 R^{-1})\right]}{(|\nabla \omega| + N_0 R^{-1})^{p_1}}.
	\end{align*}
	It follows that
	\begin{align*}
		|\nabla \omega|^{p_1} \leq \frac{(N_0 R^{-1})^{p_1}}{\mathcal{S}_1(N_0 R^{-1})}2^{q^+}\left[\mathcal{S}_1(|\nabla \omega|) + \mathcal{S}_1(N_0 R^{-1})\right],
	\end{align*}
	and therefore
	\begin{align}\label{H47}
		|\nabla \omega|^{p_1} \leq 2^{q^+} \left[ \frac{(N_0 R^{-1})^{p_1}}{\mathcal{S}_1(N_0 R^{-1})}  \mathcal{S}_1(|\nabla \omega|) + (N_0 R^{-1})^{p_1} \right].
	\end{align}
	Let $m\in\mathbb{N}_0$ and use the same notation as in \eqref{H26} and \eqref{H27}. Then \eqref{H47} gives
	\begin{equation}\label{H48}
		\begin{aligned}
			& \int_{A_{k_m,\rho_{m+1}}} |\nabla \omega|^{p_1} \, \mathrm{d} x\\
			&\leq 2^{q^+} \left[ \frac{(N_0 R^{-1})^{p_1}}{\mathcal{S}_1(N_0 R^{-1})} \int_{A_{k_m,\rho_{m+1}}} \mathcal{S}_1(|\nabla \omega|) \, \mathrm{d} x  + (N_0 R^{-1})^{p_1} |A_{k_m,\rho_{m+1}}| \right].
		\end{aligned}
	\end{equation}
	Applying \eqref{H6} with $k = k_m$, $t = \rho_m$, $\tau = \rho_{m+1}$ and then using \eqref{S-ine}, we have
	\begin{equation}\label{H49}
		\begin{aligned}
		 	\int_{A_{k_m,\rho_{m+1}}} \mathcal{S}_1(|\nabla \omega|) \, \mathrm{d} x
		 	& \leq \gamma'' \int_{A_{k_m,\rho_m}} \mathcal{S}_2 \left(\frac{\omega - k_m}{\rho_m - \rho_{m+1}} \right) \, \mathrm{d} x + \gamma''_1 |A_{k_m,\rho_m}| \\
			& \leq \gamma''2^{(m+3)\left(q^++1\right)} \int_{A_{k_m,\rho_m}} \mathcal{S}_2 \left(N_0R^{-1} \right) \, \mathrm{d} x + \gamma''_1 |A_{k_m,\rho_m}|,
		\end{aligned}
	\end{equation}
	where we have used $\rho_m - \rho_{m+1} = \frac{R}{2^{m  + 3}}$ and $0 < \omega(x) - k_m \leq N_0$ for $x \in A_{k_m, \rho_m}$. Note that
	\begin{align*}
		\mathcal{S}_2 \left(N_0R^{-1} \right)=N_0^{p_2-p_1}R^{p_1-p_2}\left(N_0R^{-1}\right)^{p_1}+b_0N_0^{q_2-q_1}R^{q_1-q_2}\left(N_0R^{-1}\right)^{q_1}\log\left(\delta+\eta N_0R^{-1}\right).
	\end{align*}
	Thus, using \textnormal{(\hyperlink{H4}{H$_4$})}\textnormal{(\hyperlink{H4i}{i})} together with the estimate $N_0 \leq  \frac{a_2(M)}{3 b_1(M)}$, we derive
	\begin{align}\label{H50}
		\mathcal{S}_2 \left(N_0R^{-1} \right)\leq \max \left\{ \left( \frac{a_2(M)}{3 b_1(M)} \right)^{p_2 - p_1}, \left( \frac{a_2(M)}{3 b_1(M)} \right)^{q_2 - q_1} \right\} \max \{ L(p), L(q) \} \mathcal{S}_1(N_0 R^{-1}).
	\end{align}
	Combining \eqref{H49} with \eqref{H50}, we arrive at
	\begin{align*}
		 \int_{A_{k_m,\rho_{m+1}}} \mathcal{S}_1(|\nabla \omega|) \, \mathrm{d} x \leq \left[C_72^{m\left(q^++1\right)} \mathcal{S}_1 \left(N_0R^{-1} \right)  + \gamma''_1\right] |A_{k_m,\rho_m}|.
	\end{align*}
	Substituting this estimate into \eqref{H48}, we get
	\begin{align}\label{H54}
		\int_{D_{m + 1}} |\nabla \omega|^{p_1} \, \mathrm{d} x \leq 2^{q^+} \left[  C_82^{m\left(q^++1\right)}(N_0 R^{-1})^{p_1} +  \gamma''_1  \right]|A_{k_m,\rho_m}|.
	\end{align}
	If \eqref{H46} does not hold, then \eqref{H54} gives
	\begin{align*}
		\int_{D_{m+1}} |\nabla \omega|^{p_1} \, \mathrm{d} x
		& \leq C_{9}2^{(q^++1)m} N_0^{p_1} R^{N - p_1} y_m.
	\end{align*}
	Arguing as in the derivation of \eqref{H32}, with $\theta''\in\left(0, 2^{-1}4^{-N} \omega_N\right)$, we arrive at
	\begin{align*}%\label{H32'}
		y_{m+1} \leq C_{10} \bar{b}^m y_m^{1 + \varepsilon_0},
	\end{align*}
	where
	\begin{align*}
		\bar{b} = 2^\frac{ (q^++1+p^-)N}{(N-1)p^-}>1 \quad \text{and} \quad  \varepsilon_0 = \frac{1}{N-1}>0.
	\end{align*}
	Hence, if \eqref{H44} holds with $\theta''$ defined by
	\begin{align*}%\label{theta''}
		\theta'' := \min \left\{ 2^{-1}4^{-N} \omega_N, C_{10}^{-\frac{1}{\varepsilon_0}} \bar{b}^{-\frac{1}{\varepsilon_0^2}} \right\},
	\end{align*}
	then another application of Lemma \ref{leRecur} yields \eqref{H45}.

	Next, we show that, for $\theta''$ defined above, there exists an integer
	\begin{align*}
		s'' = s''(\text{\rm data},a_1(M),a_2(M),b_1(M)) > 2
	\end{align*}
	such that either
	\begin{align}\label{H54'}
		\psi \leq 2^{s''} (\gamma_1'')^{\frac{1}{p_1}} R,
	\end{align}
	or \eqref{H44} holds with
	\begin{align}\label{H55}
		k_0 = \sup_{B_R} \omega - 2^{- s'' + 1} \psi.
	\end{align}

	Let $s''>2$ be an integer to be specified later. For $i = 0, 1, 2, \dots, s'' - 1$, we define $\tilde{k}_i$ and $\tilde{D}_i$ as in \eqref{H37}. Repeating the argument leading to \eqref{H48}, with $N_0R^{-1}$ replaced by $2^{-i}\psi$, we obtain
	\begin{align}\label{H56}
		\int_{A_{\tilde{k}_i,R/2}} |\nabla \omega|^{p_1} \, \mathrm{d} x \leq 2^{q^+} \left[ \frac{(2^{-i}\psi R^{-1})^{p_1}}{\mathcal{S}_1(2^{-i}\psi R^{-1})} \int_{A_{\tilde{k}_i,R/2}} \mathcal{S}_1(|\nabla \omega|) \, \mathrm{d} x + (2^{-i}\psi R^{-1})^{p_1} |A_{\tilde{k}_i,R/2}|\right].
	\end{align}
	Applying \eqref{H6} with $k = \tilde{k}_i$ for $i = 0,1,2,\dots,s'' -2$, $t = R$, and $\tau = R/2$, we have
	\begin{equation}\label{H64}
		\begin{aligned}
			\int_{A_{\tilde{k}_i,R/2}} \mathcal{S}_1(|\nabla \omega|) \, \mathrm{d} x
			& \leq \gamma'' \int_{A_{\tilde{k}_i,R}} \mathcal{S}_2(2 R^{-1} (\omega - \tilde{k}_i)) \, \mathrm{d} x + \gamma_1'' |A_{\tilde{k}_i,R}| \\
			& \leq 2^{q^++1}\gamma'' \mathcal{S}_2(2^{-i}\psi R^{-1})|A_{\tilde{k}_i,R}| + \gamma_1'' |A_{\tilde{k}_i,R}|,
		\end{aligned}
	\end{equation}
	where we have used that $0<\omega(x)-\tilde{k}_i\leq 2^{-i}\psi$ for $x\in A_{\tilde{k}_i,R}$. Furthermore, by \textnormal{(\hyperlink{H4}{H$_4$})}\textnormal{(\hyperlink{H4i}{i})} and the estimate $2^{-i}\psi\leq \sigma M\leq\frac{a_2(M)}{3b_1(M)}$, we infer that
	\begin{align*}
		\mathcal{S}_2(2^{-i}\psi R^{-1}) &\leq \max\left\{ (2^{-i}\psi)^{p_2 - p_1}, (2^{-i}\psi)^{q_2 - q_1} \right\} \max\{ L(p), L(q) \} \mathcal{S}_1(2^{-i}\psi)\\
		&\leq \max\left[\left(\frac{a_2(M)}{3b_1(M)}\right)^{q^--p^+}+1\right]\max\{ L(p), L(q) \} \mathcal{S}_1(2^{-i}\psi)
	\end{align*}
	for $i = 0,1,2,\dots,s'' - 2$. Substituting this estimate into \eqref{H64}, we arrive at
	\begin{align}\label{H57}
		\int_{A_{\tilde{k}_i,R/2}} \mathcal{S}_1(|\nabla \omega|) \, \mathrm{d} x \leq C_{11} \left(\mathcal{S}_1(2^{-i}\psi R^{-1})  + \gamma_1''\right) |A_{\tilde{k}_i,R}|.
	\end{align}
	Combining \eqref{H56} and \eqref{H57}, we obtain
	\begin{align*}
		\int_{A_{\tilde{k}_i, R/2}} |\nabla \omega|^{p_1} \, \mathrm{d} x \leq 2^{q^+} \left[ C_{12}(2^{-i}\psi)^{p_1} R^{-p_1} |A_{\tilde{k}_i, R}| + \gamma_{1}'' |A_{\tilde{k}_i, R}|\right].
	\end{align*}
	If \eqref{H54'} does not hold, then the preceding estimate yields
	\begin{align}\label{H58}
		\int_{A_{\tilde{k}_i, R/2}} |\nabla \omega|^{p_1} \, \mathrm{d} x \leq C_{13}(2^{-i}\psi)^{p_1} R^{N - p_1}\quad \text{for }i = 0,1,2,\dots,s'' - 2.
	\end{align}
	On the other hand, applying Lemma \ref{L2.6} with $u = \omega$, $\rho = R/2$, $k = \tilde{k}_i$, and $l = \tilde{k}_{i+1}$ for $i = 0,1,2,\dots,s'' - 2$, we arrive at
	\begin{align}\label{H59}
		|A_{\tilde{k}_{i+1}, R/2}|^{1 - \frac{1}{N}} \leq 4 \beta \omega_{N}^{-1} (2^{-i}\psi)^{-1} \int_{\tilde{D}_i} |\nabla \omega| \, \mathrm{d} x,
	\end{align}
	where we have used the estimate $|A_{\tilde{k}_i, R/2}| \leq \frac{1}{2} |B_{R/2}|$ arguing exactly as in the derivation of \eqref{H31}.

	Repeating the arguments that led to \eqref{H41} and invoking \eqref{H58} and \eqref{H59}, we conclude that there exists an integer $s'' > 2$ such that \eqref{H44} is satisfied with $k_0$ given by \eqref{H55}. As shown above, at least one of \eqref{H45} and \eqref{H46} is valid with $N_0=2^{- s''+1}\psi$. Proceeding exactly as in the derivations of \eqref{H42} and \eqref{H43}, we conclude that either
	\begin{align*}
		\operatorname{osc}_{B_{R/4}} u \leq (1 - 2^{-s'' }\ell^{ - 1}) \operatorname{osc}_{B_R} u\quad \text{or}\quad\operatorname{osc}_{B_R} u \leq \ell 2^{s''-1} (\gamma''_1)^{\frac{1}{p_1}} R,
	\end{align*}
corresponding to the validity of \eqref{H45} or \eqref{H46}, respectively.

If \eqref{H54'} holds, then we obtain
$$\operatorname{osc}_{B_R} u \leq \ell 2^{s''} (\gamma''_1)^{\frac{1}{p_1}} R.$$

	We have therefore established that for every $0 < R \leq R'$, at least one of the following alternatives holds:
	\begin{align*}
		\operatorname{osc}_{B_{R/4}} u \leq \theta \operatorname{osc}_{B_R} u, \quad \operatorname{osc}_{B_R} u \leq c_0 R,
	\end{align*}
	where we set
	\begin{align*}
		\theta := \max\{1 - \ell^{-1} 2^{-s'}, 1 - \ell^{-1} 2^{-s''} \}, \quad c_0 := \max\left\{\ell 2^{s'} (\gamma_1'+1)^{\frac{1}{p^-}}, \ell 2^{s''} (\gamma_1''+1)^{\frac{1}{p^-}}\right\}.
	\end{align*}
	An application of Lemma \ref{L2.7} with $\rho_0 = R'$, $\rho = \frac{R}{4}$, $\varepsilon=1$, and $b = 4$ completes the proof.
\end{proof}

%********************************************************************
\subsection{Boundary H\"older continuity}
%********************************************************************

In this subsection, we establish the H\"{o}lder continuity of $u$ in $\Omega_R$, where $y\in \partial\Omega$ and
$0<R<\operatorname{diam}(\Omega)$. Since the argument closely parallels the proof of the interior H\"{o}lder continuity developed in the previous subsection, we only outline the main steps and indicate the necessary modifications.

Let $y \in \partial\Omega$  and $0<R<\operatorname{diam}(\Omega)$.  We have the following Caccioppoli-type inequality.

\begin{lemma}\label{Le-HCac.B}
	Let $0 < \tau < t \leq R$ and let $\omega = \pm u$. Then, for any
	\begin{align}\label{BH.k}
		k \geq \max\left\{\sup_{\Omega_R}\omega-\sigma M,\sup_{(\partial\Omega)_R}\omega\right\},
	\end{align}
	where $\sigma $ is given by \eqref{H2}, the following assertions hold:
	\begin{enumerate}
		\item[\textnormal{(i)}]
			If
			\begin{align}\label{H3*}
				b_0 \leq 4[b]_{0,\alpha} R^\alpha,
			\end{align}
			then
			\begin{align}\label{H4*}
				\int_{A_{k,\tau}} |\nabla \omega|^{p_1} \, \mathrm{d} x \leq \gamma' \left(\frac{R}{t - \tau}\right)^{\alpha} \int_{A_{k,t}} \left(\frac{\omega - k}{t - \tau}\right)^{p_2} \, \mathrm{d} x + \gamma'_1 |A_{k,t}|,
			\end{align}
			where
			\begin{align*}
				\gamma' = \gamma'\left(\text{\rm data}, a_1(M),a_2(M),b_1(M)\right), \quad \gamma'_1 =\gamma_1'\left(\text{\rm data}, a_1(M),a_2(M)\right).
			\end{align*}
		\item[\textnormal{(ii)}]
			 If
			\begin{align}\label{H5*}
				b_0 > 4[b]_{0,\alpha} R^\alpha,
			\end{align}
			then
			\begin{align}\label{H6*}
				\int_{A_{k,\tau}} \mathcal{S}_1(|\nabla \omega|) \, \mathrm{d} x \leq \gamma'' \int_{A_{k,t}} \mathcal{S}_2 \left(\frac{\omega - k}{t - \tau} \right) \, \mathrm{d} x + \gamma''_1 |A_{k,t}|,
			\end{align}
			where
			\begin{align*}
				\gamma'' = \gamma''(\text{\rm data}, a_1(M),a_2(M)), \quad
				\gamma''_1 = \gamma''_1(\text{\rm data}, a_1(M),a_2(M)).
			\end{align*}
	\end{enumerate}
\end{lemma}

\begin{proof}
	The proof follows the same lines as that of Proposition \ref{Le-HCac.I}, noting $\varphi = \zeta^{\mathfrak{s}^+} (\omega-k)_+\in W^{1,\mathcal{S}}(\Omega)\cap W_0^{1,1}(\Omega)$ thanks to \eqref{BH.k}.
\end{proof}

\begin{lemma}\label{Le-HR.B}
	Let $R' \in (0,1)$ satisfy $R'<\operatorname{diam}(\Omega)$. Assume furthermore that
	\begin{align}\label{ocs-b}
		\operatorname{osc}_{(\partial\Omega)_R}u<\alpha_0R^{\beta_0}\quad \text{for all }0<R<R',
	\end{align}
	where $\beta_0\in (0,1]$ and $\alpha_0>0$ are given constants. Then there exist constants $\beta'\in (0,\beta_0]$ and $C > 0$ such that
	\begin{align*}%\label{HL2-A2*}
		\operatorname{osc}_{\Omega_R} u \leq C R'^{-\beta'} R^{\beta'}
	\end{align*}
	for any $0 < R \leq R'$.
\end{lemma}

\begin{proof}
	Let $\sigma$ be as in \eqref{H2} and define
	\begin{align}\label{Hpsi*}
		\psi := \ell^{-1} \operatorname{osc}_{\Omega_R} u\quad\text{with}\quad \ell = \max\left\{4, \frac{2}{\sigma}\right\}.
	\end{align}
	Since $\Omega$ is a Lipschitz domain, it possesses the fat complement property. More precisely, there exists a constant $\kappa=\kappa(\Omega)\geq 1$ such that
	\begin{align}\label{Fat}
		|B_r(z)|\leq \kappa|B_r(z)\setminus\Omega|\quad \text{for all } z\in\partial\Omega\text{ and } r\in (0,\operatorname{diam}(\Omega)).
	\end{align}
	We distinguish between the cases \eqref{H3*} and \eqref{H5*}.

	\textbf{Case \eqref{H3*}:} We first show that there exists a constant
	\begin{align*}
		\theta' = \theta'(\text{\rm data},a_1(M),a_2(M),b_1(M)) > 0
	\end{align*}
	such that if
	\begin{align}\label{H23*}
		| A_{k_0,R/2}| \leq \theta' R^N,
	\end{align}
	then either

	\begin{align}\label{H24*}
		\sup_{\Omega_{R/4}} \omega \leq k_0 + \frac{N_0}{2},
	\end{align}
	or
	\begin{align}\label{H25*}
		N_0 \leq (\gamma'_1)^{\frac{1}{p_1}} R,
	\end{align}
	where $0 < N_0 < \psi$, $k_0 = \max\{\sup_{\Omega_R} \omega, \sup_{(\partial\Omega)_R}\omega+\psi\} - N_0$, and $\gamma'_1$ is the constant appearing in \eqref{H4*}.

	To this end, for $m \in \N_0$, define
	\begin{equation}\label{B.seq}
		\begin{aligned}
			\rho_m &= \left(\frac{1}{4} + \frac{1}{2^{m+2}}\right) R, \quad k_m = k_0 + \frac{N_0}{2} - \frac{N_0}{2^{m+1}},\\%\label{H26*}\\
		D_{m+1} &= A_{k_m, \rho_{m+1}} \setminus A_{k_{m+1}, \rho_{m+1}} \quad \text{and}\quad  y_m = R^{-N} |A_{k_m, \rho_{m}}|.
		\end{aligned}
	\end{equation}
	It is clear that $k_m \geq \max\{\sup_{\Omega_R} \omega - \sigma M, \sup_{(\partial\Omega)_R}\omega\}$. Proceeding as in the proof of Lemma \ref{Le-HR.I}, we apply \eqref{H4*} with $k = k_m$, $t = \rho_m$, and $\tau = \rho_{m+1}$. This yields
	\begin{align}\label{H28*}
		\int_{D_{m+1}} |\nabla \omega|^{p_1} \, \mathrm{d} x\leq \Big[ \gamma'2^{ (m+3)(\alpha + p^+)}N_0^{p_2} R^{-p_2}+\gamma'_1 \Big] R^{N}y_m.
	\end{align}
	Assume now that \eqref{H25*} does not hold, that is, $N_0 \geq (\gamma'_1)^{\frac{1}{p_1}} R$. Then, combining \eqref{H28*} with H\"{o}lder's inequality, we obtain
	\begin{align}\label{H29*}
		\int_{D_{m+1}} |\nabla \omega| \, \mathrm{d} x \leq \left(\int_{D_{m+1}} |\nabla \omega|^{p_1} \, \mathrm{d} x\right)^{\frac{1}{p_1}}|D_{m+1}|^{1-\frac{1}{p_1}}\leq C 2^{\frac{m(\alpha + p^+)}{p_1}} N_0 R^{N-1}y_m.
	\end{align}
	Here and throughout the remainder of the proof, $C$ denotes a positive constant depending only on the data and on $a_1(M)$, $a_2(M)$ and $b_1(M)$. Its value may vary from line to line.

	On the other hand, since $k_m>\sup_{(\partial\Omega)_R}\omega\geq \sup_{(\partial\Omega)_{\rho_m}}\omega$, the function
	\begin{align*}
		\hat{\omega}^{(k_m)}(x) =
		\begin{cases}
			\max\left\{\omega(x),k_m\right\}, & x \in \Omega_{\rho_m},  \\
			k_m, & x \in B_{\rho_m}\setminus\Omega_{\rho_m},
		\end{cases}
	\end{align*}
	belongs to $W^{1,1}(B_{\rho_m})$ in view of Lemma~\ref{L2.12}. Therefore, we may apply Lemma \ref{L2.6} with $k = k_m$, $l = k_{m+1}$, $\rho = \rho_{m+1}$, and $u = \hat{\omega}^{(k_m)}$, which gives
	\begin{align*}
		(k_{m+1} - k_m) |\hat{A}_{k_{m+1},\rho_{m+1}}|^{1 - \frac{1}{N}} \leq\frac{\beta\rho_{m+1}^N}{|B_{\rho_{m+1}} \setminus \hat{A}_{k_m,\rho_{m+1}}|}  \int_{\hat{A}_{k_{m},\rho_{m+1}} \setminus \hat{A}_{k_{m+1},\rho_{m+1}}} |\nabla \hat{\omega}^{(k_m)}| \, \mathrm{d} x,
	\end{align*}
	where $\hat{A}_{t,\rho_{m+1}} := \{ x \in B_{\rho_{m+1}} \colon \hat{\omega}^{(k_m)}(x) > t \}$ for $t\in\R$ and $\beta > 1$ is a constant depending only on $N$. Moreover, it is straightforward to verify that
	\begin{align*}
		&A_{k_{m+1},\rho_{m+1}}\subset\hat{A}_{k_{m+1},\rho_{m+1}},\\
		&B_{\rho_{m+1}} \setminus \Omega_{\rho_{m+1}}\subset B_{\rho_{m+1}} \setminus \hat{A}_{k_{m},\rho_{m+1}},\\
		&\hat{A}_{k_{m},\rho_{m+1}} \setminus \hat{A}_{k_{m+1},\rho_{m+1}}\subset D_{m+1}.
	\end{align*}
	Taking these inclusions into account together with \eqref{Fat}, we infer from the previous inequality that
	\begin{align}\label{H30*}
		\int_{D_{m+1}} |\nabla \omega| \, \mathrm{d} x \geq \kappa^{-1} \beta^{ - 1} N_0 2^{- (N+2) }\omega_N2^{-m}R^{N-1}y_{m+1}^{1-\frac{1}{N}}.
	\end{align}
	Combining this estimate with \eqref{H29*}, we arrive at
	\begin{align*}
		y_{m+1}^{1 - \frac{1}{N}} \leq C 2^{m \left(\frac{\alpha+p^+}{p^-}+1\right)} y_m.
	\end{align*}
	Consequently,
	\begin{align}\label{H32*}
		y_{m+1} \leq C b^m y_m^{1 + \varepsilon_0},
	\end{align}
	where
	\begin{align*}
		b = 2^\frac{ (\alpha+2p^+)N}{(N-1)p^-}>1\quad\text{and} \quad  \varepsilon_0 = \frac{1}{N-1}>0.
	\end{align*}
	Therefore, if \eqref{H23*} holds with $\theta'$ given by
	\begin{align}\label{H33*}
		\theta' := C^{-\frac{1}{\varepsilon_0}} b^{-\frac{1}{\varepsilon_0^2}},
	\end{align}
	then the sequence $\{y_m\}_{m\in\mathbb{N}_0}$ satisfies the recursive inequality \eqref{H32*} with
	\begin{align*}
		y_0= R^{-N} \left| A_{k_0, R/2} \right| \leq \theta' =C^{-\frac{1}{\varepsilon_0}} b^{-\frac{1}{\varepsilon_0^2}}.
	\end{align*}
	Applying Lemma \ref{leRecur}, we conclude that $y_m \to 0$ as $m \to \infty$, and thus, \eqref{H24*} holds.

	Next, we prove that, for $\theta'$ given in \eqref{H33*}, there is an integer $s' = s'(\text{\rm data}, a_2(M),b_1(M)) > 2$ such that either
	\begin{align}\label{H34*}
		\psi \leq 2^{s'} (\gamma'_1)^{\frac{1}{p_1}} R,
	\end{align}
	or \eqref{H23*} holds with
	\begin{align}\label{H35*}
		k_0 = \max\left\{\sup_{\Omega_R} \omega, \sup_{(\partial\Omega)_R} \omega+\psi\right\} - 2^{-s'+1} \psi.
	\end{align}
	Suppose that \eqref{H34*} does not hold, namely,
	\begin{align*}%\label{H36*}
		\psi > 2^{s'} (\gamma'_1)^{\frac{1}{p_1}} R,
	\end{align*}
	where $s' > 2$ is an integer to be specified later. Define
	\begin{align}\label{H37*}
		\tilde{k}_i := \max\left\{\sup_{\Omega_R} \omega, \sup_{(\partial\Omega)_R} \omega+\psi\right\} - 2^{-i} \psi, \quad \tilde{D}_i:= A_{\tilde{k}_i, R/2} \setminus A_{\tilde{k}_{i+1}, R/2}
	\end{align}
	for $i = 0, 1,  \dots, s' - 1$. Clearly, $\tilde{k}_i\geq \max\left\{\sup_{\Omega_R}\omega-\sigma M,\sup_{(\partial\Omega)_R}\omega\right\}$ for $i = 0,1,2,\dots,s' - 1$. Proceeding as in the proof of Lemma \ref{Le-HR.I}, we apply \eqref{H4*} with $t = R$, $\tau = R/2$, and $k = \tilde{k}_i$ in order to obtain
	\begin{align*}%\label{H39*}
		\int_{\tilde{D}_i} |\nabla \omega|^{p_1} \, \mathrm{d} x \leq \left(2^{\alpha + p^+ } \gamma' L(p)\left(\frac{a_2(M)}{3b_1(M)}\right)^{p_2 - p_1}+1\right) \omega_N(2^{-i} \psi)^{p_1} R^{N - p_1}.
	\end{align*}
	On the other hand, since $\tilde{k}_i>\sup_{(\partial\Omega)_R}\omega\geq \sup_{(\partial\Omega)_{R/2}}\omega$, the function
	\begin{align*}
		\hat{\omega}^{(\tilde{k}_i)}(x) =
		\begin{cases}
			\max\left\{\omega(x),\tilde{k}_i\right\}, & x \in \Omega_{R/2}, \\
			\tilde{k}_i, & x \in B_{R/2}\setminus\Omega_{R/2},
		\end{cases}
	\end{align*}
	belongs to $W^{1,1}(B_{R/2})$ in view of Lemma~\ref{L2.12}. Thus, applying Lemma \ref{L2.6} with $\rho = R/2$, $l = \tilde{k}_{i+1}$, $k = \tilde{k}_i$ for $i = 0, 1, 2, \dots, s' - 2$, and $u = \hat{\omega}^{(\tilde{k}_i)}$, we obtain
	\begin{align}\label{H39*'}
		(\tilde{k}_{i+1} - \tilde{k}_i) |\hat{A}_{\tilde{k}_{i+1},R/2}|^{1 - \frac{1}{N}} \leq\frac{\beta(R/2)^N}{|B_{R/2} \setminus \hat{A}_{k_i,R/2}|}  \int_{\hat{A}_{\tilde{k}_i,R/2} \setminus \hat{A}_{\tilde{k}_{i+1},R/2}} |\nabla \hat{\omega}^{(\tilde{k}_i)}| \, \mathrm{d} x,
	\end{align}
	where $\hat{A}_{t,R/2} := \{ x \in B_{R/2} \colon \hat{\omega}^{(\tilde{k}_i)}(x) > t \}$ for $t\in\R$ and $\beta > 1$ depends only on $N$. Furthermore,
	\begin{align*}
		A_{\tilde{k}_{i+1},R/2}\subset\hat{A}_{\tilde{k}_{i+1},R/2},\quad B_{R/2} \setminus \Omega_{R/2}\subset B_{R/2} \setminus \hat{A}_{\tilde{k}_{i},R/2},\quad \hat{A}_{\tilde{k}_i,R/2} \setminus \hat{A}_{\tilde{k}_{i+1},R/2}\subset \tilde{D}_i.
	\end{align*}
	Proceeding exactly as in the derivation of \eqref{H30*}, we infer from \eqref{H39*'} that
	\begin{align}\label{5.122}
		\int_{\tilde{D}_i} |\nabla \omega| \, \mathrm{d} x \geq  \kappa^{-1}\omega_N \beta^{-1} 2^{-i - 1} \psi |A_{\tilde{k}_{i+1}, R/2}|^{1 - \frac{1}{N}},\quad \text{for} \ \ i = 0, 1, 2, \dots, s' - 2.
	\end{align}
	Then, arguing as in the proof of Lemma \ref{Le-HR.I}, we arrive at
	\begin{align}\label{H41*}
		|A_{\tilde{k}_{s' - 1}, R/2}| \leq \left(\frac{C}{s'-1}\right)^{\frac{{N(p_1 - 1)}}{(N - 1)p_1}} R^N.
	\end{align}
	Choosing $s'>2$ sufficiently large so that
	\begin{align*}
		\frac{C}{s'-1}<1 \quad\text{and}\quad \left(\frac{C}{s'-1}\right)^{\frac{{N(p^- - 1)}}{(N - 1)p^-}} <\theta',
	\end{align*}
	it follows that
	\begin{align*}
		|A_{\tilde{k}_{s' - 1}, R/2}| \leq \theta' R^N.
	\end{align*}
	Thus, \eqref{H23*} holds with $k_0$ as in \eqref{H35*}. Therefore, as shown above, if \eqref{H34*} does not hold, then at least one of \eqref{H24*} and \eqref{H25*}, with $N_0 = 2^{-s' + 1} \psi$, holds. Let us first consider the case that \eqref{H24*} holds. If $ \operatorname{osc}_{\Omega_R}u>4\alpha_0R^{\beta_0}$, then assumption \eqref{ocs-b} implies that
	\begin{align*}
		\operatorname{osc}_{(\partial\Omega)_R}u<\frac{1}{4}\operatorname{osc}_{\Omega_R}u.
	\end{align*}
	Consequently, at least one of the two functions $\omega=\pm u$ satisfies
	\begin{align}\label{pmu}
		\frac{1}{4}\operatorname{osc}_{\Omega_R}u\leq \sup_{\Omega_R}\omega-\sup_{(\partial\Omega)_R}\omega.
	\end{align}
	In the following, we denote by $\omega$  the function $u$ or $-u$ that satisfies \eqref{pmu}. In view of \eqref{pmu} and \eqref{Hpsi*}, we have
	\begin{align*}
		\sup_{(\partial\Omega)_R}\omega+\psi\leq \sup_{\Omega_R} \omega.
	\end{align*}
	Using this and arguing as in the derivation of \eqref{H42}, we arrive at
	\begin{align*}
		\operatorname{osc}_{\Omega_{R/4}} u \leq (1 - 2^{-s' }\ell^{ - 1}) \operatorname{osc}_{\Omega_R} u.
	\end{align*}
	Therefore, if \eqref{H24*} holds, then at least one of the following alternatives is valid:
	\begin{align}\label{H42*}
		\operatorname{osc}_{\Omega_R}u\leq 4\alpha_0R^{\beta_0} \quad \text{or} \quad \operatorname{osc}_{\Omega_{R/4}} u \leq (1 - 2^{-s' }\ell^{ - 1}) \operatorname{osc}_{\Omega_R} u.
	\end{align}
	On the other hand, if \eqref{H25*} holds with $N_0 = 2^{-s' + 1} \psi$, then
	\begin{align*}%\label{H43*}
		\operatorname{osc}_{\Omega_R} u \leq \ell 2^{s' - 1} (\gamma'_1)^{\frac{1}{p_1}} R\leq \ell 2^{s' - 1} (\gamma'_1)^{\frac{1}{p_1}} R^{\beta_0}.
	\end{align*}
	Finally, if \eqref{H34*} holds, then
	\begin{align*}%\label{H43**}
		\operatorname{osc}_{\Omega_R} u \leq \ell 2^{s'} (\gamma'_1)^{\frac{1}{p_1}} R\leq \ell 2^{s'} (\gamma'_1)^{\frac{1}{p_1}} R^{\beta_0}.
	\end{align*}
	In conclusion, for the case of \eqref{H3*}, at least one of the following alternatives is valid:
	\begin{align*}
		\operatorname{osc}_{\Omega_R} u \leq \bar{c}_0 R^{\beta_0}\quad \text{or}\quad \operatorname{osc}_{\Omega_{R/4}} u \leq \bar{\theta} \operatorname{osc}_{\Omega_R} u,
	\end{align*}
	where $\bar{c}_0=\max\left\{4\alpha_0, \ell 2^{s'} (\gamma'_1)^{\frac{1}{p_1}}\right\}$ and $\bar{\theta}=1 - 2^{-s' }\ell^{ - 1}$.

	\vskip5pt
	\textbf{Case \eqref{H5*}:} The proof proceeds in much the same way as in the treatment of case \eqref{H5} in the proof of Lemma \ref{Le-HR.I}. We therefore focus only on the modifications required in the boundary setting. We first show that there exists a constant
	\begin{align*}
		\theta'' = \theta''(\text{\rm data}, a_1(M),a_2(M),b_1(M)) > 0
	\end{align*}
	such that, whenever
	\begin{align}\label{H44*}
		\big| A_{k_0, R/2} \big| \leq \theta'' R^N,
	\end{align}
	at least one of the following alternatives holds:
	\begin{align}\label{H45*}
		\sup_{\Omega_{R/4}} \omega \leq k_0 + \frac{N_0}{2}
	\end{align}
	or
	\begin{align}\label{H46*}
		N_0 \leq (\gamma_1'')^{\frac{1}{p_1}} R,
	\end{align}
	where $0 < N_0 < \psi$, $k_0 = \max\{\sup_{\Omega_R} \omega, \sup_{(\partial\Omega)_R}\omega+\psi\} - N_0$, and $\gamma_1''$ is as in \eqref{H6*}.

	Let $m\in\mathbb{N}_0$ and use the same notation as in \eqref{B.seq}. As in the proof of Lemma~\ref{Le-HR.I}, we obtain
	\begin{align}\label{H54*}
		\int_{D_{m + 1}} |\nabla \omega|^{p_1} \, \mathrm{d} x \leq 2^{q^+} \left[  C2^{m\left(q^++1\right)}(N_0 R^{-1})^{p_1} +  \gamma''_1  \right]|A_{k_m,\rho_m}|.
	\end{align}
	Here, again, $C$ denotes a positive constant depending only on the data and on $a_1(M)$, $a_2(M)$ and $b_1(M)$.

	If \eqref{H46*} does not hold, then \eqref{H54*} yields
	\begin{align*}
		\int_{D_{m+1}} |\nabla \omega|^{p_1} \, \mathrm{d} x
		& \leq C2^{(q^++1)m} N_0^{p_1} R^{N - p_1} y_m.
	\end{align*}
	By H\"older's inequality we deduce from the preceding inequality that
	\begin{align}\label{B97}
		\int_{D_{m+1}} |\nabla \omega| \, \mathrm{d} x
		& \leq C2^{\frac{(q^++1)m}{p_1}} N_0 R^{N-1} y_m.
	\end{align}
	Next, observe that $k_m>\sup_{(\partial\Omega)_R}\omega\geq \sup_{(\partial\Omega)_{\rho_m}}\omega$. Therefore,
	\begin{align*}
		\hat{\omega}^{(k_m)}(x) =
		\begin{cases}
			\max\{\omega(x),k_m\}, & x \in \Omega_{\rho_m},\\
			k_m, & x \in B_{\rho_m}\setminus\Omega_{\rho_m},
		\end{cases}
	\end{align*}
	belongs to $W^{1,1}(B_{\rho_m})$ in view of Lemma~\ref{L2.12}. Hence, Lemma \ref{L2.6} can be applied with $k = k_m$, $l = k_{m+1}$, $\rho = \rho_{m+1}$, and $u = \hat{\omega}^{(k_m)}$. Arguing as in the derivation of \eqref{H30*}, we conclude that
	\begin{align*}
		\int_{D_{m+1}} |\nabla \omega| \, \mathrm{d} x \geq \kappa^{-1} \beta^{ - 1} N_0 2^{- (N+2) }\omega_N2^{-m}R^{N-1}y_{m+1}^{1-\frac{1}{N}}.
	\end{align*}
	Bringing this together with \eqref{B97} gives
	\begin{align*}
		y_{m+1} \leq C \bar{b}^m y_m^{1 + \varepsilon_0},
	\end{align*}
	where
	\begin{align*}
		\bar{b} = 2^\frac{ (q^++1+p^-)N}{(N-1)p^-}\quad\text{and} \quad  \varepsilon_0 = \frac{1}{N-1}.
	\end{align*}
	Consequently, if \eqref{H44*} holds with $\theta''$ given by
	\begin{align*}
		\theta'' := C^{-\frac{1}{\varepsilon_0}} \bar{b}^{-\frac{1}{\varepsilon_0^2}},
	\end{align*}
	then, in view of Lemma \ref{leRecur}, we derive again \eqref{H45*}.

	Next, we show that, for $\theta''$ defined above, there exists an integer
	\begin{align*}
		s'' = s''(\text{\rm data}, a_2(M),b_1(M)) > 2
	\end{align*}
	such that either
	\begin{align}\label{H54'*}
		\psi \leq 2^{s''} (\gamma_1'')^{\frac{1}{p_1}} R,
	\end{align}
	or \eqref{H44*} holds with
	\begin{align}\label{H55*}
		k_0 = \max\left\{\sup_{\Omega_R} \omega, \sup_{(\partial\Omega)_R} \omega+\psi\right\}- 2^{- s'' + 1} \psi.
	\end{align}

	Let $s''>2$ be an integer to be specified later. For $i = 0, 1, 2, \dots, s'' - 1$, we define $\tilde{k}_i$ and $\tilde{D}_i$ as in \eqref{H37*}. As in the proof of Lemma \ref{Le-HR.I}, we obtain
	\begin{align*}
		\int_{A_{\tilde{k}_i, R/2}} |\nabla \omega|^{p_1} \, \mathrm{d} x \leq 2^{q^+} \left[ C(2^{-i}\psi)^{p_1} R^{-p_1} |A_{\tilde{k}_i,R}| + \gamma_{1}'' |A_{\tilde{k}_i, R}|\right].
	\end{align*}
	If \eqref{H54'*} does not hold, then the preceding estimate yields
	\begin{align}\label{H58*}
		\int_{A_{\tilde{k}_i, R/2}} |\nabla \omega|^{p_1} \, \mathrm{d} x \leq C(2^{-i}\psi)^{p_1} R^{N - p_1}\quad \text{for }i = 0,1,2,\dots,s'' - 2.
	\end{align}
	On the other hand, repeating the argument that led to \eqref{5.122}, we arrive at
	\begin{align}\label{H59*}
		\int_{\tilde{D}_i} |\nabla \omega| \, \mathrm{d} x \geq  \kappa^{-1}\omega_N \beta^{-1} 2^{-i - 1} \psi |A_{\tilde{k}_{i+1}, R/2}|^{1 - \frac{1}{N}}.
	\end{align}
	Repeating the arguments that led to \eqref{H41*} and invoking \eqref{H58*} and \eqref{H59*}, we conclude that there exists an integer $s'' > 2$ such that \eqref{H44*} is satisfied with $k_0$ given by \eqref{H55*}.
	Then, as shown above, at least one of \eqref{H45*} and \eqref{H46*} is valid with $N_0=2^{- s''+1}\psi$. If \eqref{H45*} holds, then by proceeding as in the derivation of \eqref{H42*}, we obtain
	$$\operatorname{osc}_{\Omega_R}u\leq 4\alpha_0R^{\beta_0} \quad \text{or} \quad \operatorname{osc}_{\Omega_{R/4}} u \leq (1 - 2^{-s'' }\ell^{ - 1}) \operatorname{osc}_{\Omega_R} u.$$
On the other hand, if \eqref{H46*} holds with $N_0=2^{- s''+1}\psi$, then we obtain	\begin{align*}
		\operatorname{osc}_{\Omega_R} u \leq  \ell 2^{s'' - 1} (\gamma''_1)^{\frac{1}{p_1}} R^{\beta_0}.
	\end{align*}
	Finally, if \eqref{H54'*} holds, then
	\begin{align*}%\label{H43**}
		\operatorname{osc}_{\Omega_R} u \leq \ell 2^{s''} (\gamma''_1)^{\frac{1}{p_1}} R^{\beta_0}.
	\end{align*}
	In summary, we infer that for every  $0 < R \leq R'$, at least one of the following alternatives holds:
	\begin{align*}
		\operatorname{osc}_{\Omega_R} u \leq \tilde{c}_0 R^{\beta_0}\quad \text{or}\quad \operatorname{osc}_{\Omega_{R/4}} u \leq \tilde{\theta} \operatorname{osc}_{\Omega_R} u,
	\end{align*}
	where
	\begin{align*}
		\tilde{c}_0=\max\left\{4\alpha_0,\ell 2^{s'} (\gamma_1'+1)^{\frac{1}{p^-}}, \ell 2^{s''} (\gamma_1''+1)^{\frac{1}{p^-}}\right\}, \quad \tilde{\theta} := \max\left\{1 - \ell^{-1} 2^{-s'},1 - \ell^{-1} 2^{-s''}\right\}.
	\end{align*}
	An application of Lemma \ref{L2.7} with $\rho_0 = R'$, $\rho = \frac{R}{4}$, $\varepsilon=\beta_0$ and $b = 4$ completes the proof of the lemma.
\end{proof}

%********************************************************************
\subsection{Proof of global H\"older continuity}
%******************************************************************

\begin{proof}[Proof of Theorem~\ref{Th.H}]
	By the compactness of $\overline{\Omega}$ and Lemmas \ref{Le-HR.I} and \ref{Le-HR.B}, we find $R_0>0$, $C>0$ and $\alpha\in (0,1]$ such that
	\begin{align*}
		\operatorname{osc}_{B_R(x_0)\cap\overline{\Omega}}u\leq CR^{\alpha}
	\end{align*}
	for every ball $B_R(x_0)$ satisfying $B_R(x_0)\cap\Omega\ne \emptyset$.
	Let $x,z\in \overline{\Omega}$, $x\neq z$. We distinguish two cases.

	\textbf{Case 1:} $|x-z|<R_0$.

	Set $R:=|x-z|$. Then, $B_R(x)\cap \Omega\neq \emptyset$. Therefore,
	\begin{align*}
		|u(x)-u(z)|\le \operatorname{osc}_{B_R(x)\cap \overline{\Omega}} u\le C R^\alpha = C |x-z|^\alpha.
	\end{align*}

	\textbf{Case 2:} $|x-z|\ge R_0$.

	Let $M=\|u\|_{L^\infty(\Omega)}$.  Since $|x-z|\ge R_0$, it follows that
	\begin{align*}
		|u(x)-u(z)|\leq 2M
		\le \frac{2M}{R_0^\alpha}|x-z|^\alpha.
	\end{align*}

	Combining the estimates obtained in the two cases, we arrive at
	\begin{align*}
		|u(x)-u(z)| \le C_* |x-z|^\alpha \quad \text{for all }x,z\in \overline{\Omega},
	\end{align*}
	where
	\begin{align*}
		C_*:=\max\left\{C,\frac{2M}{R_0^\alpha}\right\}.
	\end{align*}
	Consequently, $u\in C^{0,\alpha}(\overline{\Omega})$. The proof is complete.
\end{proof}

\begin{proof}[Proof of Theorem~\ref{Th.H'}]
	Let $u\in \Wpzero{\mathcal{S}}$ be a weak solution of problem \eqref{D}. Then, $u\in L^\infty(\Omega)$ in view of Theorem~\ref{CD.boundedness}. Note that, by \textnormal{(\hyperlink{H1}{H$_1$})}\textnormal{(\hyperlink{H1i}{i})}, hypothesis
	\textnormal{(\hyperlink{H4}{H$_4$})} is satisfied with $\alpha=\frac{p^+}{N}\in(0,1)$. Moreover, \textnormal{(\hyperlink{H1}{H$_1$})}(ii) yields
	\begin{align*}
		q(x)-p(x)<\frac{p(x)}{N}\leq\frac{p^+}{N} \quad\text{for all }x\in\overline{\Omega}.
	\end{align*}
	Hence, by compactness of $\overline{\Omega}$, there exists $\varepsilon>0$ such that
	\begin{align*}
		q(x)-p(x)-\frac{p^+}{N}\leq-\varepsilon \quad\text{for all }x\in\overline{\Omega}.
	\end{align*}
	Consequently, \eqref{TA3} holds with $\alpha=\frac{p^+}{N}$. Therefore, all arguments in the proofs of Proposition \ref{Le-HCac.I} and Lemmas \ref{Le-HR.I}--\ref{Le-HR.B} remain valid for $u$. Moreover, by a standard approximation argument, one can show that $\varphi = \zeta^{\mathfrak{s}^+} (\omega-k)_+$ in the proofs of Proposition \ref{Le-HCac.I} and Lemma \ref{Le-HCac.B} belongs to $\Wpzero{\mathcal{S}}$. Therefore, it is indeed an admissible test function in \eqref{def_sol_D}.
\end{proof}

We conclude the paper with the following remark concerning the definition of weak solutions to problem \eqref{E-H}.

\begin{remark}
	It is well known that, in the case $\mathcal{S}(x,t)=t^{p(x)}$ with $p\in C^{0, \frac{1}{|\log t|}}(\overline{\Omega})$, we have
	\begin{equation}\label{ESet}
	\Wpzero{\mathcal{S}}=W^{1,\mathcal{S}}(\Omega)\cap W_0^{1,1}(\Omega),
	\end{equation}
	see Fan--Zhao \cite{Fan-Zhao-2001}. For the general case, it is clear that $\Wpzero{\mathcal{S}}\subseteq W^{1,\mathcal{S}}(\Omega)\cap W_0^{1,1}(\Omega)$. Since we have been unable to verify \eqref{ESet}, we define weak solutions to problem \eqref{E-H} using a larger class of test functions.
\end{remark}

%********************************************************************
\section*{Acknowledgment}
%********************************************************************

%The authors thank the anonymous reviewers for their careful reading and valuable comments, which helped to improve the manuscript.
K. Ho was supported by the University of Economics Ho Chi Minh City (UEH), Vietnam. Y.-H. Kim was supported by the National Research Foundation of Korea (NRF) grant funded by the Korea government (MSIT) (IRIS RS-2025-24533710).

%********************************************************************
%\section*{Declarations}
%%********************************************************************
%
%\subsection*{Ethical Approval}
%Not applicable.
%
%\subsection*{Competing interests}
%The authors declare that they have no competing interests.
%
%\subsection*{Authors' contributions}
%The authors contributed equally to this work.
%
%\subsection*{Availability of data and materials}
%Data sharing is not applicable to this article as no new data were created or analyzed in this study.

%********************************************************************

\end{document}